\documentclass[12pt, oneside,reqno]{amsart}
\pdfoutput=1
\usepackage[letterpaper,margin=1in]{geometry}
\usepackage{graphicx}
\usepackage{subcaption}
\usepackage{float}
\usepackage{tikz}
\usepackage{amssymb}
\usepackage{appendix}
\usepackage{mathrsfs}
\usepackage{bbold}

\usepackage[normalem]{ulem}

\usepackage{chngcntr}
\usepackage{longtable}

\usepackage{comment}
\usepackage{array}

\usepackage{hyperref}
\hypersetup{colorlinks,urlcolor=blue,citecolor=blue}

\usepackage[nameinlink,noabbrev]{cleveref}
\numberwithin{figure}{subsection}

\newcommand{\C}{\mathbb{C}}
 
\newcommand{\R}{\mathbb{R}}
\newcommand{\RR}{\mathbb{R}}
\newcommand{\Z}{\mathbb{Z}}
\newcommand{\ZZ}{\mathbb{Z}}

\newcommand{\Q}{\mathbb{Q}}

\newcommand\Qt{Q_1}

\theoremstyle{plain}
\newtheorem{thm}{Theorem}[section]     
\newtheorem{lemma}[thm]{Lemma}
\newtheorem{prop}[thm]{Proposition}
\newtheorem{corollary}[thm]{Corollary}

\newtheorem{note}[thm]{Note}

\theoremstyle{definition}
\newtheorem{defn}[thm]{Definition}
\newtheorem{remark}[thm]{Remark}
\newtheorem{ex}[thm]{Example}
\def\z{\noindent}
\definecolor{darkorange}{RGB}{140, 70, 0}
\DeclareMathOperator{\ord}{ord}

\begin{document}

\author{Ovidiu Costin$^{*}$, Gerald V. Dunne$^{\dagger}$, Ali Saraeb$^{\ddagger}$}

\address{$^{*}$Mathematics Department, The Ohio State University, 231 W. 18th Avenue,  Columbus, Ohio 43210, USA}
\address{$^{\dagger}$Physics Department, University of Connecticut, 196 Auditorium Road,
Storrs, CT 06269, USA}
\address{$^{\ddagger}$Mathematics Department, The Ohio State University, 231 W. 18th Avenue,  Columbus, Ohio 43210, USA}

\setcounter{tocdepth}{2}

\email{costin@math.osu.edu}
\email{gerald.dunne@uconn.edu}
\email{saraeb.1@osu.edu}
\title[Resurgence and rigidity of mock theta functions] {Resurgent rigidity of mock theta functions: uniqueness and natural boundary crossing} 


\maketitle

\begin{abstract}
We develop a resurgent transseries approach to mock theta functions
and show that the associated Mordell--Appell integrals are the
foundational objects underlying their modular and resurgent structure. From these integrals,
we recover the modular transformation laws, obtain
uniqueness characterizations of the corresponding mock theta vectors, 
provide a canonical crossing of their natural boundary, and produce new mock theta functions.

\vspace{1mm}

\z After representing the Mordell--Appell integrals as Laplace transforms
of elementary resurgent functions, we rotate the Laplace contour to a
Stokes line.  Their Stokes phenomena and their own modular
transformation laws reproduce the modular transformation laws of
the associated unary series and, combined with our uniqueness
results, those of the mock theta functions themselves. Conversely, we regard these modular
transformation laws as functional equations governing the mock theta
vectors. 

\vspace{1mm}

\z For Ramanujan's order-$3$ pair $(f,\omega)$ and the distinguished
mock theta vectors of orders $5$, $7$, and $11$, satisfying the full
$SL(2,\mathbb Z)$ transformation laws, we prove that these functional
equations uniquely determine a holomorphic solution in the unit disk
under the canonical minimal-growth normalization. This unique solution
is precisely the corresponding classical mock theta vector. 
We show that the normalization is essential: for orders $5$, $7$, and
$11$ we determine explicit nontrivial solutions of the homogeneous
problem, yielding new mock theta functions with faster-growing coefficients.

\vspace{1mm}

\z Finally, we give two intrinsic and natural constructions
for crossing the natural boundary $|q|=1$. The first continues the
modular functional equations and takes their unique holomorphic
solutions on the other side. The second acts directly on the resurgent asymptotic series at the cusp: we replace $\tau$ by $-\tau$ and apply
\'Ecalle--Borel summation. We prove that these two
constructions agree, giving a canonical and explicit boundary
crossing.

\end{abstract}

\tableofcontents
\section{Introduction}

Ramanujan's mock theta functions have inspired a century of remarkable
developments, beginning with the pioneering work of Watson
\cite{Watson} and continuing through a vast literature which has
substantially advanced their theory
\cite{Andrews1986,Hickerson1988,FolsomOno,Zwe08,ono,Zag09,GM12}.

\vspace{2mm}

In this paper we analyze and characterize the distinguished vectors
among the known mock theta functions of prime orders \(3,5,7\), and
\(11\) which satisfy the full \(SL(2,\ZZ)\)  action, in the sense of
Zagier \cite{Zag09}: their components form vectors which are closed
under the action of \(S\) and \(T\), the generators of
\(SL(2,\ZZ)\). The same distinguished mock theta functions also arise
in  three-dimensional topological quantum field theories \cite{lawrence,Hikami04,hikamiTL,Dimofte:2015kkp,GMP,Andersen:2018khh,Cheng:2018vpl,CCKPS3d,Garoufalidis:2021osl}, where the
asymptotic growth of their integer-valued coefficients  encodes a corresponding effective central charge
\cite{JG24,Adams:2025qgj,Harichurn:2025suf,griffen}.
Further deep connections between resurgence and modularity arise in
quantum modular forms, knot theory, and quantum topology
\cite{zagier-quantum,GM21,sauzin2022,Wheeler:2023cye,Fantini:2024ihf,
Fantini:2025wap,Cheng2026,mrunmay}.

\vspace{2mm}

One of our main results is a characterization of uniqueness. We prove that,
within each of these orders, the distinguished mock theta vectors are
uniquely determined by their $S$ and $T$ transformation laws together
with a canonical normalization condition. Equivalently, within a natural class of
admissible transformation laws, the canonical normalization selects the
solutions of minimal cusp growth, a minimality criterion that parallels,
but does not coincide with, the optimality condition of Cheng et al. \cite{Cheng2016}.

\vspace{2mm}
Another main result concerns the longstanding problem of {\em crossing
the natural boundary} $|q|=1$.  Besides its intrinsic mathematical
interest, this problem has a direct interpretation in topological
quantum field theory, where crossing the natural boundary is associated
with orientation reversal of the underlying three-manifold
\cite{GMP,Cheng:2018vpl,CCKPS3d,CDGG,ACDGO}.
A number of prescriptions for such a continuation have been proposed,
typically depending on a particular $q$-series or $q$-Pochhammer
representation
\cite{hickerson-mortenson,mortenson,bringmann-inverse}.
We find that, for the distinguished mock theta vectors, the continuation
is instead determined intrinsically, in two remarkably natural and
ultimately simple ways.

\vspace{2mm}

The first construction is to continue the modular transformation laws
themselves across the boundary and take their unique holomorphic
solutions on the other side.  The second construction comes directly
from resurgent asymptotics. We prove that these two
constructions coincide, giving a canonical and explicit continuation across the boundary.

\vspace{2mm}

Both constructions follow from properties of the
Mordell--Appell integrals.
The modular relations of the mock theta
functions induce corresponding relations between their cusp asymptotic
series. Since these are also the asymptotic series of the associated
Mordell--Appell integrals, permanence of relations under 
\'{E}calle--Borel summation transfers
the same modular structure to the Mordell--Appell integrals. Unlike the mock theta functions,
however, the Mordell--Appell integrals are analytic across $\tau\in\RR$, and therefore carry
these relations analytically to the other side of the natural boundary.
The resurgent and modular explanation of this preservation of
relations is developed in \S\ref{sec:notation}. 

\vspace{2mm}

We further show that the modular structure is contained in the
Mordell--Appell kernels, which are elementary functions. Their Fourier
transforms determine the mixing matrices that appear in the modular transformation laws, as well as the fractional parts of
the rational powers occurring in the mock theta vectors. The remaining
integral shifts, equivalently the order vector (see \eqref{eq:order-vector}), are fixed by normalization. Thus the Mordell--Appell integrals determine the transformation laws up to this normalization.

\vspace{2mm}

On the Stokes line, the \'{E}calle--Borel-summed asymptotic parts of the
Mordell--Appell integrals are given by false unary theta series in the
reciprocal nome, while their Stokes contributions are finite linear combinations of unary
theta series in the dual reciprocal nome.
The two parts are related by the modular $S$ transformation: on the
Stokes line, the principal-value part is $S$-dual to the Stokes jump of
the dual integral. The unary character of both parts follows directly from the elementary
Mordell--Appell kernels: their simple poles and elementary trigonometric
structure produce quadratic exponents with periodic coefficients.
The modular $S$ transformation relates the two parts, mapping the
principal-value part to the Stokes jump of the dual integral, while the
$S$ and $T$ transformation laws determine their fractional powers.
These mechanisms are developed in \S\ref{sec:notation}; see also
\S\ref{S:median-jump}.

\vspace{2mm}

 We furthermore show that the canonical normalization, or equivalently
the minimal-growth condition, is essential for uniqueness. Once this
condition is relaxed, higher-growth solutions appear.  The structure of the Mordell--Appell integrals naturally
suggests explicit homogeneous solutions for every prime order
\(p\geq5\). This leads, in particular, to {\em new explicit mock theta
functions of orders \(5,7\), and \(11\)}. These solutions differ from
the classical ones by vector-valued modular forms and, to our
knowledge, have not previously appeared in the literature.

\vspace{2mm}

\'{E}calle's theory of transseries and resurgent asymptotics provides
the natural framework for these results. Resurgent functions may be
viewed as a natural extension of analytic functions, with resurgent
transseries playing the role of convergent Taylor series \cite{ecalle,Sauzin06,costin-book,Mitschi:2016fxp}. In the simplest situations relevant here, such a transseries consists
of a generally divergent, Borel--\'{E}calle summable, power series
together with convergent series in small exponentials.  Mock theta
functions provide concrete examples of precisely this structure: at
their cusps, their modular transformation laws decompose them into a
Borel--\'{E}calle summable asymptotic power-series part and a convergent
small-exponential part.

\vspace{2mm}

We view the $SL(2,\ZZ)$ transformation laws as a form of discrete
dynamics on the hyperbolic plane. Ordinary difference equations describe
discrete dynamics generated by translations, hence by discrete
isometries of the Euclidean plane; the modular transformations play the
corresponding role for the hyperbolic plane. Resurgent transseries and
Stokes phenomena for difference equations were developed by Braaksma
\cite{Braaksma}, following the corresponding theory for differential
equations \cite{CostinDuke}. Here, however, the dynamics is generated
by the noncommuting transformations $S$ and $T$, which must be taken into account.

\vspace{2mm}

In the setting needed here, the relevant Borel transforms are elementary
meromorphic functions, and \'{E}calle--Borel summation is realized by
Laplace transformation \cite{costin-book}. When the Borel transform is singular along the
direction of integration---that is, when the integration direction is a
Stokes line---the Laplace transform is understood in an appropriate
space of distributions. Since the singularities occurring here are
poles, the distributional integral is simply the Cauchy principal value,
equivalently the arithmetic mean of the two lateral contour integrals.
To make the paper self-contained, we use only basic elements of
resurgence theory, which we re-prove as needed.
\vspace{2mm}

Simple changes of variables put the Mordell--Appell integrals directly into this
Borel--Laplace form. Their Stokes line is $\tau\in i\RR^-$, and their
Stokes-line transseries are calculated explicitly by contour deformation
and residue calculus. Between Stokes lines, resurgent functions have
unique transseries decompositions, just as analytic functions have
unique Taylor expansions. From the resurgent point of view, Ramanujan's
and Watson's characterization of mock theta functions in terms of
``closed exponentials''  in their asymptotics \cite{Andrews1986,Watson} may be viewed as a
transseries integrability condition.

\subsection{Examples of mock theta vectors, transformation laws, and Mordell--Appell integrals}
\label{S:mfexamples}

Three familiar examples of the distinguished mock theta vectors are
the \((f,\omega)\) vector of order \(3\) mock theta functions, the
\((\chi_0,\chi_1)\) vector of order \(5\) mock theta functions, and the
\((\mathcal F_0,\mathcal F_1,\mathcal F_2)\) vector of order \(7\) mock theta functions:
\begin{eqnarray}\label{eq:Exmf}
\begin{pmatrix}
q^{-1/12}\,f(q^2)\cr
2q^{2/3}\,\omega(q)\cr
2q^{2/3}\,\omega(-q)
\end{pmatrix}
\quad ; \quad
\begin{pmatrix}
Q^{-1/60}\,(\chi_0(Q)-2)\cr
Q^{-49/60}\,Q\,\chi_1(Q)
\end{pmatrix}
\quad ; \quad
\begin{pmatrix}
Q^{-1/84}(\mathcal F_0(Q)-2)\cr
Q^{-25/84}\,\mathcal F_1(Q)\cr
Q^{-121/84}\,Q\,\mathcal F_2(Q)
\end{pmatrix}.
\end{eqnarray}
Here
\[
Q=e^{-2 t},\qquad
Q_1=e^{-2\pi^2/t},\qquad
q=e^{-t},\qquad
q_1=e^{-\pi^2/t},\qquad
t=-\pi i\tau.
\]
We refer to the vector
$\bigl(\mathcal M_{11,j}(\tau)\bigr)_{j=1}^{5}$ introduced by Zagier
\cite[p.~15]{Zag09}, and given in Eq \eqref{eq:zagier11}, as the
order-$11$ mock theta vector.

The order-$3$ functions were organized into the first vector in
\eqref{eq:Exmf} by Zwegers \cite{Zwe08, zwegers-p3}. Its components are defined by \cite{GM12} 
\begin{eqnarray}
\omega(q)
&:=&
\sum_{n=0}^{\infty}
\frac{q^{\,2n(n+1)}}{(q;q^2)_{n+1}^{\,2}},
\qquad |q|<1,
\label{eq: omeg def}
\\
f(q)
&:=&
\sum_{n=0}^{\infty}
\frac{q^{\,n^2}}{(-q;q)_{n}^{\,2}},
\qquad |q|<1.
\label{eq: f def}
\end{eqnarray}
The order 3 modular transformation laws follow from the formulas obtained by
Watson \cite{Watson,GM12},
\begin{eqnarray}
q^{2/3}\,\omega(-q)
&=&
-\sqrt{\frac{\pi}{t}}\,
q_1^{2/3}\,\omega(-q_1)
+\sqrt{\frac{12t}{\pi}}\,W_3(t),
\qquad t>0,
\label{eq:omega tlaw}
\\
q^{2/3}\omega(q)
&=&
\sqrt{\frac{\pi}{4t}}\,
q_1^{-1/12}f(q_1^2)
-\sqrt{\frac{3t}{\pi}}\,
W_2\!\left(\frac{t}{2}\right),
\qquad t>0,
\label{eq:omega f law}
\end{eqnarray}
Here the relevant  Mordell--Appell integrals are
(see, e.g., \cite[p.~100, 114]{GM12})
\begin{eqnarray}
W_3(t)
&:=&
\frac{1}{t} \int_0^\infty
e^{-3x^2/t}\,
\frac{\sinh(x)}{\sinh(3x)}\,dx,
\qquad \Re t>0,
\label{eq:w3}
\\
W_2(t)
&:=&
\frac{1}{t} \int_0^\infty
e^{-\frac32 x^2/t}\,
\frac{\cosh(x)}{\cosh(3x)}\,dx,
\qquad \Re t>0.
\label{eq:w2}
\end{eqnarray}
For the transformations of $f(q)$ and $f(-q)$, which will not be needed
in this paper, see \cite[p.~100, 114]{GM12}.

For mock orders 5, 7 and 11,
the modular transformation law under
\(S: t\mapsto \pi^2/t\) has the structural form
\begin{equation}
\label{eq:eqQ}
X(Q)
=
C\, t^{-1/2}MX(Q_1)+t^{1/2}F(t),
\end{equation}
where $X$ is a vector of $Q$-series, \(F\) is a
vector of Mordell--Appell integrals, \(M\) is an involutive mixing matrix, and $C$ is a real constant.  We show that the transformation laws for $X(-q)$, usually stated separately \cite{GM12}, follow from \eqref{eq:eqQ} by analytic
continuation and the action of a specific element of SL$(2,\ZZ)$.
Thus these transformation laws contain no independent information.

After the changes of variables which put the Mordell--Appell integrals into Laplace form,
the right-hand side of \eqref{eq:eqQ} is precisely the transseries of
\(X(Q)\) as \(t\to 0^+\): the Mordell-Appell term gives its asymptotic power
series, while \(X(Q_1)\) supplies the exponentially small part. At the
\(S\)-dual cusp, uniqueness of the transseries decomposition forces the
same mixing matrix \(M\) to act on both the mock theta vector and the
Mordell--Appell vector.

 \begin{remark}\label{R:sep-mf3}
{\rm
In this work, we treat the order-$3$ mock theta case separately, for
several reasons. Its uniqueness was previously proved by two of the
present authors and collaborators in \cite{CDGG}. Moreover, the
order-$3$ case is degenerate from the point of view of the general
theory developed here: its mock modular transformation laws
\eqref{eq:omega tlaw} and \eqref{eq:omega f law} have a slightly
different form from \eqref{eq:eqQ}, coupling $q$ to $q_1^2$ rather
than to $q_1$; see \S\ref{sec:mf$_3$}. Order $3$ is also special
because no normalization or growth condition is required for
uniqueness.

In \S\ref{P:mf3} we give a shorter proof of uniqueness for order 3 following the
approach of \cite{CDGG}, correcting some minor typographical errors in
the original argument. In \S\ref{alt P:mf3} we give a second,
independent proof.
}
\end{remark}
\begin{remark}
In \cite{Zag09} Zagier comments that of the five different 2-component vectors that satisfy the order 5 mock theta transformation laws, only the vector involving the order 5 mock theta functions $\chi_0$ and $\chi_1$ ``transforms under the full modular $SL(2,\mathbb Z)$ group". This is consistent with our construction and our uniqueness result.

\end{remark}

\subsection{Organization of the paper}
The structure of the paper is as follows. Section~\ref{sec:notation} introduces the notation and framework used throughout the paper, including the relevant Mordell--Appell integrals, false theta functions, and resurgent transseries. 
Section~\ref{S:MR} contains the main results on uniqueness, homogeneous solutions, new faster-growth mock theta functions, and natural boundary crossing. Section \ref{S:proofs} contains the proofs of the main results. Section~\ref{S:Mordell-transformations} develops the resurgent and
modular structure of the Mordell--Appell integrals, establishing the
structural results used in the main theorems.

\section{Notation and Framework}\label{sec:notation}

This section introduces the notation and resurgent framework used
throughout the paper. We also explain the preservation of modular relations described in the Introduction. In
particular, we show how the analyticity and modular properties of the
Mordell--Appell integrals carry these relations through the natural
boundary, and how, on the Stokes line, the principal-value and Stokes
parts are related by $S$-duality.

 \noindent Throughout the paper, we use the following notation, conventions, and terminology. 

We write $\mathbb H=\{\tau\in\mathbb C:\Im\tau>0\}$ for the upper half-plane, 
$\mathbb D=\{z\in\mathbb C:|z|<1\}$ for the unit disk, and $I_n$ for the $n\times n$ identity matrix. For $\tau\in\mathbb H$, we set
\begin{equation}
    t=-\pi i\tau, \qquad q=e^{\pi i\tau}=e^{-t}, \qquad Q=q^2=e^{2\pi i\tau}.
\end{equation}
We use $Q=q^2$ throughout.

$S$ and $T$ denote the standard generators of $SL(2,\mathbb Z)$,
\begin{equation}
    S:\tau\mapsto -1/\tau\quad ,\qquad T:\tau\mapsto\tau+1.
\end{equation}
In terms of $t$, these transformations are:
\begin{equation}
    S:t\mapsto \pi^2/t \quad ,\qquad T:t \mapsto t-i \pi.
\end{equation}
Under $S$ we define
\begin{equation}
    q_1=e^{\pi i(-1/\tau)}=e^{-\pi^2/t},
    \qquad
    Q_1=q_1^2=e^{2\pi i(-1/\tau)}=e^{-2\pi^2/t}.
\end{equation}
For $r\in\mathbb Q$, we define
\begin{equation}
    q^r:=e^{\pi i r\tau},\qquad Q^r:=e^{2\pi i r\tau},
\end{equation}
and, under $S$,
\begin{equation}
    q_1^r:=e^{\pi i r(-1/\tau)},
    \qquad
    Q_1^r:=e^{2\pi i r(-1/\tau)}.
\end{equation}

When $\tau$ is in the lower half-plane, we write power series in the reciprocal nomes 
\begin{equation}
    Q^{-1}=e^{-2 \pi i \tau},\qquad Q_1^{-1}= e^{-2\pi i(-1/\tau)}.
\end{equation}
These are the natural variables for the transformation equations in the region $|Q|>1$, because on this side of the natural boundary, $|Q|=1$, their solutions are naturally expressed in terms of unary false theta series in $Q^{-1}$ and $Q_1^{-1}$, multiplied by rational powers of these variables. Similarly, for the half-nomes, we use $q^{-1}$ and $q_1^{-1}$.

We use the $q$-Pochhammer symbols
\begin{equation}
    (a;b)_n=\prod_{j=0}^{n-1}(1-a\,b^j),
    \qquad
    (a;b)_\infty=\prod_{j=0}^\infty(1-a\,b^j),
\end{equation}
and denote by $\eta(\tau)$ the Dedekind eta function,
\begin{equation}
\label{eq:eta def}
    \eta(\tau)=Q^{1/24}(Q;Q)_\infty.
\end{equation}
Its transformation laws are
\begin{equation}
\label{eq:eta tlaws}
    \eta(\tau+1)=e^{\pi i/12}\eta(\tau),
    \qquad
    \eta\left(-\frac{1}{\tau}\right)
    =\sqrt{-i\tau}\,\eta(\tau)
    =\sqrt{\frac{t}{\pi}}\,\eta(\tau).
\end{equation}
We denote the Kronecker symbol \cite{Apostol,folsom}  by
\begin{equation}
    \Bigl(\frac{a}{n}\Bigr),
\end{equation}
and the M\"obius function by $\mu(n)$.

We adopt the following notation for false theta functions \cite{Bring15,GMP,sauzin2022}, for $m \in \Z_{\ge2}$ and $1 \le a \le m-1$:
\begin{equation}
\Psi^{(a)}_m (q) \;  :=  \; \sum_{n=0}^\infty \psi^{(a)}_{2m}(n) q^{\frac{n^2}{4m}} \qquad \in q^\frac{a^2}{4m}\,\mathbb{Z}[[q]]
\label{eq:false}
\end{equation}
where
\begin{eqnarray} \label{eq: psi false coef}
\psi^{(a)}_{2m}(n)  =  \left\{
\begin{array}{cl}
1, & n\equiv  a \quad ({\rm mod}\,\, 2m)\,, \\
- 1, & n\equiv - a \quad ({\rm mod}\,\, 2m)\,, \\
0, & \text{otherwise}.
\end{array} \right.
\end{eqnarray}

We extend this notation to arbitrary $a \in \Z$ by $2m$-periodicity and oddness: 
\begin{eqnarray}
    \Psi^{(a+2m)}_m(q)= \Psi^{(a)}_m(q), \qquad \Psi^{(-a)}_m(q) = -\Psi^{(a)}_m(q).
\end{eqnarray}
In particular, this gives
\begin{eqnarray}
    \Psi^{(0)}_m(q)= \Psi^{(m)}_m(q)=0, \qquad \Psi^{(m+a)}_m(q) = -\Psi^{(m-a)}_m(q).
\end{eqnarray}
We will also use the identity $\Psi^{(a)}_m(q^2)=\Psi^{(2a)}_{2m}(q)$.

For these combinations, we adopt the shorthand  notation
\begin{align}
    \Psi^{n_{a_1}(a_1)+n_{a_2}(a_2)+\dots}_m(q) := n_{a_1} \Psi^{(a_1)}_m(q) + n_{a_2} \Psi^{(a_2)}_m(q) + \dots.
    \label{eq:ns}
\end{align}

On the Stokes line, the Mordell--Borel integrals decompose $t\in \RR^-$ into combinations of these false theta functions. For prime-order $m=p \ge 5,$ these involve unary series in $Q^{-1}$ and $Q_1^{-1}$, while for $p=3,$ they also involve unary series in $q^{-1}$ and $q_1^{-1}$.

\bigskip

Assume that $f=(f_1,\ldots,f_n)$ is a vector of functions holomorphic
in $\mathbb D$. Its order-of-vanishing vector at zero is
\begin{eqnarray}
\operatorname{ord}_0 f
:=
\bigl(\operatorname{ord}_0 f_1,\ldots,\operatorname{ord}_0 f_n\bigr)
\label{eq:order-vector}
\end{eqnarray}
where $\operatorname{ord}_0 f_j$ is the order of the first nonzero
derivative of $f_j$ at zero. We will simply call $\operatorname{ord}_0 f$ the {\it order vector} of $f$.

\subsection{Borel summed transseries and their uniqueness}\label{S:t-uniq}

Let $\mathcal{L}$ denote the Laplace transform,
\[
(\mathcal LF)(x)=\int_0^\infty e^{-xw}F(w)\,dw.
\]
In the setting of this paper, when Borel summation is applied to
asymptotic power series in $1/x$ as $x\to\infty$ in the right
half-plane, the resulting transseries have the form
\begin{equation}\label{eq:iden}
y(x)
=
(\mathcal LF)(x)
+
\sum_{k=-m}^{\infty}c_ke^{-krx},
\qquad x>0,
\end{equation}
where the exponential series converges for $x>0$ and $F$ is meromorphic
in $\mathbb C$ (cf. Note~\ref{N:Duniq}, part 2 below) with the poles on $\RR^-$, and is bounded except for disks of small radius $\delta$ around the poles. When the Laplace transform contour is rotated to go through the line of poles, the Laplace transform is taken in the sense of distributions, \(-\operatorname{PV} \int_0^{\infty} e^{-x v} F(v)dv, \,v=-w\).  This is a very special case of \'Ecalle-Borel summation (see \cite{costin-book} and references therein), for which \'Ecalle's {\em medianization} is simply the half sum of the integral above and below the ray of poles-- the Cauchy principal value (PV) of the integral. 

 \subsubsection{\bf Stokes lines}
For the Borel transforms encountered in this paper, the Stokes line is
the ray of poles $\RR^-$.

\subsubsection{\bf Products of Laplace transforms}\label{S:prod}  A standard result in the theory of Laplace transforms is that, if $f_1,f_2$ have appropriate bounds to be Laplace transformable, then $(\mathcal{L}f_1)( \mathcal{L} f_2)=\mathcal{L}(f_1*f_2)$ where $(f_1*f_2)(v)=\int_0^v f_1(s) f_2(v-s)ds$; see e.g. \cite{costin-book} p. 22.
\subsubsection{\bf Uniqueness of transseries representations} The uniqueness of Borel summed transseries representations is one of the fundamental structural results of resurgence theory  (see \cite{costin-book} and references therein). To keep the paper self-contained, we prove here a special case that is sufficient for our purposes.

\begin{prop}\label{P:decay}
Assume $F\in L^1(\mathbb R^+)$ and that for some $\epsilon>0$,
\begin{equation}\label{eq:estlower1}
   (\mathcal LF)(x)=O(e^{-\epsilon x}),
   \qquad x\to+\infty.
\end{equation}
Then $F=0$ almost everywhere on $[0,\epsilon]$.
The same conclusion holds in the PV setting above.
\end{prop}
The proof of this proposition is given in \S\ref{S:P:decay}.

\begin{prop}[Uniqueness of transseries decompositions]\label{P:uniq-transseries} If $y(x)$ in \eqref{eq:iden} vanishes for $x>0$, then
\[
F\equiv0
\qquad\text{and}\qquad
c_k=0
\quad\forall k.
\]
The same holds if the Laplace transform is a PV.
\end{prop}
The proof of this proposition is given in \S\ref{S:proof-unique-transseries}.
\smallskip

\noindent {\bf Notation:} 
Let $F$ be the Borel transform of the asymptotic series of $y$. We write \(y=\mathcal Ay+\mathcal Ey\) for the transseries representation of \(y\),
where
\[
(\mathcal Ay)(x):=(\mathcal LF)(x)
\]
denotes its Borel-summed part, and
\[
(\mathcal Ey)(x):=
\sum_{k=-m}^\infty c_ke^{-krx}
\]
its exponential part.

\begin{note}\label{N:Duniq}{\rm
1. By Proposition~\ref{P:uniq-transseries}, if
\[
y_1=\mathcal Ay_1+\mathcal Ey_1
=
\mathcal Ay_2+\mathcal Ey_2
=y_2,
\]
then
\[
\mathcal Ay_1=\mathcal Ay_2,
\qquad
\mathcal Ey_1=\mathcal Ey_2.
\]

2. 
In our applications, a factor \(x^{-1/2}\) is usually present in front of the Mordell--Appell integrals of interest while $F$ has a $u^{-1/2}$ singularity at 0; this can be absorbed into the operator \(\mathcal A\). Indeed, since
\[
\int_0^\infty u^{-1/2}e^{-ux}\,du=\sqrt{\frac{\pi}{x}},
\]
 we have, cf. \S\ref{S:prod}, 
\begin{equation}
\label{eq:1/x}
\sqrt{\frac{\pi}{x}}
\int_0^\infty u^{-1/2}e^{-ux}F(u)\,du
=
\int_0^\infty e^{-vx}G(v)\,dv,
\end{equation}
where
\begin{equation}
\label{eq:G}
G(v)
=
\int_0^v
\frac{F(v-s)}{\sqrt{s(v-s)}}\,ds
=
\int_0^1
\frac{F(v(1-u))}{\sqrt{u(1-u)}}\,du.
\end{equation}
$G$ is now analytic at zero. The transformation \(F\mapsto G\)
preserves analyticity and bounds at infinity. If a resurgent asymptotic series $y(x)$ has leading order $x^m$, with
$m\in\ZZ_{\geq0}$, we write
\[
y(x)=x^{m+1}\widetilde y(x),
\]
so that $\widetilde y(x)$ is a formally small power series in $x^{-1}$,
amenable to \'{E}calle--Borel summation. Similarly, if a Laplace
transform $\mathcal L$ occurs with a prefactor $x^m$, $m\in\ZZ_{\geq0}$,
we keep the factor $x^m$ outside the Laplace transform.
}
\end{note}
\begin{prop}\label{P:sign-change}
Assume that $F$ is meromorphic, analytic at zero, and has poles on the
negative real line only. Assume that, excluding a $\delta$-neighborhood
of the poles, it is bounded at infinity. Then the asymptotic expansion
of $\mathcal{L}(F)(x)$ as $x\to+\infty$ and that of
\[
\operatorname{PV}\int_0^{-\infty} e^{-xs}F(s)\,ds
\]
as $x\to-\infty$ are obtained from each other by the substitution
$x\mapsto -x$.
\end{prop}
\begin{proof}
Let $x=-y$, where $y\to+\infty$. Then
\[
\operatorname{PV}\int_0^{-\infty}e^{-xs}F(s)\,ds
=
-\operatorname{PV}\int_0^\infty e^{-yu}F(-u)\,du.
\]
By Watson's lemma, the two asymptotic expansions are the formal Laplace
transforms of the power series at zero of $F(u)$ and $-F(-u)$,
respectively. The result follows from the fact that the latter is
obtained from the former by replacing $u$ by $-u$ and changing the sign.
\end{proof}

\begin{note}
A similar result holds in greater generality than meromorphicity, with essentially the
same proof.
\end{note}

\subsection{The Mordell-Appell integrals relevant for this paper} 
\label{sec:basis}
We use the following basis of Mordell--Appell integrals \cite{Watson,GM12}, initially defined for
$\Re t>0$, and for $a=1,2,\dots, 2p-1$:
\begin{eqnarray}
    JS_{(p,a)}(t)&:=&\frac1t \int_0^\infty e^{-p u^2/t}
    \frac{\sinh((p-a)u)}{\sinh(pu)}\,du
    \label{eq:JS def}
    \\
    JC_{(p,a)}(t)&:=&\frac1t \int_0^\infty e^{-p u^2/t}
    \frac{\cosh((p-a)u)}{\cosh(pu)}\,du.
    \label{eq:JC def}
\end{eqnarray}
For the applications considered in this paper we choose $p$ to be prime.  The normalized Mordell--Appell integrals defined below in \eqref{eq:l}, with
$\tau=i t/\pi$, are analytic in both the upper and lower half-planes
and, in fact, in a sector of width $3\pi$ centered on $i\RR^+$
(equivalently, on $t\in\RR^+$).  This follows directly by contour
deformation.  The negative real $t$-axis, corresponding to
$\tau\in i\RR^-$, is the Stokes line.

We form  vector-valued  Mordell-Appell integrals that have finite dimensional closed orbits under the modular $S$ and $T$ transformations. First, define the integrals
\begin{eqnarray}
    L_j^{(p)}(t)&:= \text{sign}(6j-p)JS_{(12p,|12j-2p|)}(t) + JS_{(12p,14p-12j)}(t) 
    \nonumber\\
    & \qquad + JS_{(12p,12j+2p)}(t) + JS_{(12p,10p-12j)}(t).
    \label{eq:ls}
\end{eqnarray} 
for $j=1, 2, \dots, \ell_p$, where 
\begin{eqnarray}
    \ell_p:= \frac{(p-1)}{2}
    \label{eq:ellp}
\end{eqnarray}

\begin{prop}{\bf Closure of Mordell-Appell Integrals Under Modular S-Transformation\\}
\label{prop:Sclosure}
Define the \underline{normalized} vector of integrals as
\begin{eqnarray}
    \mathcal L^{(p)}(t):= \sqrt{\frac{48p t}{\pi}}\, 
    \begin{pmatrix} L_1^{(p)}(t) \cr
    L_2^{(p)}(t) \cr
    \vdots \cr
    L_{\ell_p}^{(p)}(t) 
    \end{pmatrix}.
    \label{eq:l}
\end{eqnarray}
Then $\mathcal L^{(p)}$ is a vector-valued modular form of weight $1/2$, transforming under the modular $S$ transformation, $\tau\to -1/\tau$, ($t\to \pi^2/t$), as:
\begin{eqnarray}
    \mathcal L^{(p)}(t)  =\sqrt{\frac{\pi}{t}} \, M^{(p)}\, \mathcal L^{(p)}\left(\frac{\pi^2}{t}\right)
    \label{eq:lp-s}
\end{eqnarray}
Here $M^{(p)}$ is an $\ell_p\times \ell_p$ mixing matrix, with entries
\begin{eqnarray}
M^{(p)}_{jk} &= \frac{2}{\sqrt{p}}(-1)^{j+k+\lfloor \frac{p+3}{6} \rfloor}\sin{\left(\frac{6jk\pi}{p}\right)},
\label{eq:M}
\end{eqnarray}
\begin{proof}
The result follows by Fourier analysis and elementary trigonometric
identities; the details are given in \S\ref{sec:modularS}.
\end{proof}

\end{prop}
\begin{remark}
    Note that $\left(M^{(p)}\right)^2={\bf 1}$, consistent with inversion of the $S$ transformation in \eqref{eq:lp-s}.
    
\end{remark}
\begin{remark}
    The behavior of $\mathcal L^{(p)}$ under the modular $T$ transformation, $\tau\to\tau+1$, ($t\to t-i\, \pi$), is discussed in section \ref{sec:modularT}.
\end{remark}
For what follows, it is convenient to express the Mordell integral \eqref{eq:l} as an Eichler-type period integral involving a theta series. Recall the definition \eqref{eq: psi false coef} and let
\begin{equation}
\varepsilon_p^{(j)}(n)= \psi_{12p}^{(a_1)}(n) + \psi_{12p}^{(a_2)}(n) + \psi_{12p}^{(a_3)}(n) + \psi_{12p}^{(a_4)}(n),
\qquad
\begin{aligned}
a_1&=6j-p, & a_2&=7p-6j,\\
a_3&=6j+p, & a_4&=5p-6j.
\end{aligned}
\end{equation}
Then we define the theta vector $\Theta_p(\tau)= \bigl(\Theta_p^{(1)}(\tau), \cdots, \Theta_p^{(\ell_p)}(\tau) \bigr)^T$, with
\begin{equation}\label{eq: def THetap}
    \Theta_p^{(j)}(\tau)= \frac{1}{2} \sum_{n \in \Z} n\, \varepsilon_p^{(j)}(n)  \, Q^{n^2/(24p)}, \qquad \tau \in \mathbb H.
\end{equation}
\begin{lemma}\label{lem:int tran mordel theta}
    With the above notation, we have 
    \begin{eqnarray}
        \mathcal L_j^{(p)}(- \pi i \tau)= - \frac{i}{2 \sqrt{3p}} \int_0^{i \infty} \frac{\Theta_p^{(j)}(z)}{\sqrt{-i(z+\tau)}} dz, \qquad \tau \in \mathbb H.
    \end{eqnarray}
\end{lemma}
The proof of Lemma \ref{lem:int tran mordel theta} is given in \S \ref{sec: mordel theta func}. Moreover, by the Poisson summation formula, $\Theta_p$ has the following transformation laws (see \S \ref{sec: theta trasn laws mordell}).
\begin{lemma}\label{lem: thetap trans lawss}
    The theta vector $\Theta_p$, defined above, is a weight-$3/2$ vector-valued modular form satisfying 
    \begin{eqnarray}
        \Theta_p(-1/\tau)= (-i \tau)^{3/2} M^{(p)} \Theta_p(\tau),
    \end{eqnarray}
    and 
    \begin{eqnarray}
         \Theta_p(\tau+1) = (D^{(p)})^{-1} \Theta_p(\tau), \qquad D^{(p)}:=
\operatorname{diag}
\bigl(
e^{-2\pi i\Delta_{(p,1)}},
\ldots,
e^{-2\pi i\Delta_{(p,\ell_p)}}
\bigr).
    \end{eqnarray}
\end{lemma}
\subsection{Borel-\'Ecalle summation presentation of Mordell-Appell integrals;  half-integer power corrections}

The Mordell-Appell integrals we need can be rewritten, by a simple change of variable $u^2\mapsto w$ as the Laplace transform of a meromorphic
function decaying on \(\RR^+\), namely the standard Borel-Laplace form (with a $w^{-1/2}$ singularity at 0, that can be eliminated as described in Note \ref{N:Duniq} (2); likewise in the case of $x^m$, $m\in \ZZ_{\ge 0}$ multiplying a Laplace transform or a formal power series. Since these changes are elementary,  
we shall therefore use the terminology ``Borel sum",
``medianization", etc., interchangeably for the corresponding Mordell
integrals, or Laplace transforms with the  $w^{-1/2}$ singularity present.

\subsection{Resurgent Transseries of False Theta Functions on the Stokes Line}
\label{sec:false}

$ $

\noindent{\bf Definition:}  The following rational numbers $\Delta_{(p, j)}$ will play an important role as exponents
\begin{eqnarray}
    \Delta_{(p, j)} := \frac{(6j-p)^2}{24p}
    \qquad , \qquad j=1, 2, \dots, \ell_p.
    \label{eq:delta-mock}
\end{eqnarray}

For later use, define
\begin{equation}
\Phi_j^{(p)}(q)
=
q^{-\Delta_{(p,j)}}
\Psi_{6p}^{(a_1)+(a_2)+(a_3)+(a_4)}(q),
\qquad
\begin{aligned}
a_1&=6j-p, & a_2&=7p-6j,\\
a_3&=6j+p, & a_4&=5p-6j.
\end{aligned}
\label{eq:mock-x0}
\end{equation}
With the notation of \eqref{eq:ns}, the linear combination of false
theta functions is formed before factoring out
$q^{-\Delta_{(p,j)}}$; in particular,
$\Phi_j^{(p)}(q)\in\mathbb Z[[q]]$.

\begin{prop}\label{P:Stokes-decomposition}
On the Stokes line $t\in\RR^-$, the unique transseries decomposition of
$\mathcal L^{(p)}(t)$ is
\begin{align}
\mathcal L^{(p)}_j(t)
={}&
Q^{-\Delta_{(p,j)}}\,
\Phi_j^{(p)}(Q^{-1})
+
i\sqrt{\frac{\pi}{|t|}}
\sum_{k=1}^{\ell_p}
M^{(p)}_{jk}\,
Q_1^{-\Delta_{(p,k)}}\,
\Phi_k^{(p)}(Q_1^{-1}),
\qquad t\in\RR^-.
\label{eq:mock_Q_unary}
\end{align}
For the order 3 mock theta functions, the relevant decomposition of the components is
\begin{eqnarray}
(q^{-1})^{-2/3}\,\Phi^{(3)}_1(-q^{-1})
&=&
i\sqrt{\frac{\pi}{|t|}}\,
(q_1^{-1})^{-2/3}\,\Phi^{(3)}_1(-q_1^{-1})
+
i\sqrt{\frac{12|t|}{\pi}}\,W_3(t),
\quad t\in\RR^-,
\label{eq:omega tlaw vee}
\\
(q^{-1})^{-2/3}\,\Phi^{(3)}_1(q^{-1})
&=&
-i\sqrt{\frac{\pi}{4|t|}}\,
(q_1^{-1})^{1/12}\,
\Phi^{(3)}_2((q_1^{-1})^2)
-
i\sqrt{\frac{3|t|}{\pi}}\,
W_2\!\left(\frac{t}{2}\right),
\quad t\in\RR^-.
\label{eq:omega f law vee}
\end{eqnarray}

\end{prop}
\begin{proof}
    The proof of Proposition~\ref{P:Stokes-decomposition} is given in
\S\ref{S:St-dec}.
\end{proof}
\begin{remark}\label{R:Core-principle}
The appearance in \eqref{eq:mock_Q_unary} of the same rational
exponents $\Delta_{(p,j)}$, as in the mock theta vectors, follows from the stokes decomposition. The proof shows that the $Q_1^{-1}$ part of the
decomposition is obtained directly by residue calculus and hence has
an expansion in rational powers of $Q_1^{-1}$. The corresponding
expansion in $Q^{-1}$ follows from $S$-duality in a particularly direct
way: analytically continuing the median value from the Stokes line back
to $t\in\RR^+$, applying the $S$ transformation there, and continuing
back to $t\in\RR^-$ transforms the principal-value part into the Stokes
jump of the $S$-dual integral. The reversal of the lateral directions
under $t\mapsto\pi^2/t$ is essential here. Thus the $Q^{-1}$ and
$Q_1^{-1}$ unary expansions are the two $S$-dual parts of the same
Stokes decomposition; see \S\ref{S:median-jump} for details.

Their fractional powers are fixed by the modular transformation laws.
Under the $S$ and $T$ transformations, the mock theta components acquire
phases determined by their global rational powers of $Q$. Permanence
of relations implies that the corresponding Mordell--Appell components
satisfy the same transformation laws and hence acquire the same phases.
Their rational exponents therefore agree modulo integers.

The remaining integral shifts correspond to the normalization freedom
\[
Q^{-r}\Phi(Q)
=
Q^{-r+m}\bigl(Q^{-m}\Phi(Q)\bigr),
\qquad m\in\mathbb Z.
\]

The mechanism fixing the fractional powers can be seen explicitly for
order 3 mock theta functions. Temporarily denote the mock theta vector in
\S\ref{S:mfexamples} by $V$, and denote by $F$ the corresponding
Mordell--Appell vector including its $t^{1/2}$ prefactor. Under
$T^2$, the components of $V$ acquire the phases
\[
D=
\operatorname{diag}
\left(
e^{-\pi i/6},
e^{4\pi i/3},
e^{4\pi i/3}
\right).
\]
By Proposition~\ref{P:uniq-transseries} and Note~\ref{N:Duniq},
uniqueness of the transseries decomposition gives
\[
\mathcal A V=\mathcal A F,
\qquad
\mathcal A(T^2V)
=
\mathcal A(T^2F)
=
D\,\mathcal A F.
\]
Thus the $T^2$ action on the Mordell--Appell part is also multiplication
by $D$. Since $F$ is analytic across $\tau\in\RR$, the same relation
holds after continuation to the Stokes line. Uniqueness of the
transseries decomposition there then implies that its principal-value
part inherits the same $T^2$ phases.

The principal-value part is a series in negative rational powers of
$q$ and $q_1$ (see \S\ref{S:unarymf3} and \S\ref{S:median-jump}). For $T^2$ to act on each component by a single
multiplicative phase, all powers occurring in that component must be
congruent modulo integers. Hence each component has the form
\[
q^{r_j}\Phi_j(q),
\]
with 
\(\Phi_j(q)\) an integer power series in powers of $1/q$ and $e^{2\pi i r_j}$ equal to the corresponding entry of $D$. Thus
the modular phases determine the fractional powers $r_j\bmod\mathbb Z$,
while the integral shifts remain as the normalization freedom described
above.

Finally, the special trigonometric form of the kernels implies that
the series $\Phi_j$ are unary: the kernels have a periodic structure
of poles whose residues, up to an overall constant, belong to
$\{0,1,-1\}$.  The modular structure determines the fractional
powers, while the elementary structure of the kernels accounts for
the unary character of the remaining series.

\end{remark}

\begin{remark}
    Due to the special form of the Mordell--Appell kernels, the transseries decomposition in \eqref{eq:mock_Q_unary} coincides with the decomposition into real
and imaginary parts.  It is therefore unique. The real part consists of series in terms of $1/Q$, each multiplied by a rational power, while the imaginary part consists of a linear combination of the same series, each with its associated rational power, but  in terms of $1/Q_1$. 
\end{remark}

\begin{remark}\label{r:C5}
We see from \eqref{eq:mock_Q_unary} that there is a normalization
freedom in the decomposition into a rational power and a false theta
series. Namely, a false theta series $\Phi_j^{(p)}(Q^{-1})$ may be normalized
to start with $Q^{-m}$, for some $m\in\mathbb Z_{\geq 0}$, while
replacing $\Delta_{(p,j)}$ by $(\Delta_{(p,j)}-m)$ in the corresponding
power of $Q^{-1}$. Equivalently, in 
\eqref{eq:mock_Q_unary} this replaces
$Q^{-\Delta_{(p,j)}}$ by $Q^{-\Delta_{(p,j)}+m}$.
\end{remark}

\begin{remark}
For the orders $p=3,5,7,11$ considered here, our uniqueness theorem
shows that there is a maximal choice of these nonnegative integers, $m$, for
which a mock theta vector satisfying the prescribed transformation
laws still exists. This {\em maximal} choice gives the {\em minimal} cusp-growth
solution, and fixes the normalization. 

On the $|Q|>1$ side, by contrast, if the rational powers of $Q^{-1}$
and $Q_1^{-1}$ in \eqref{eq:mock_Q_unary} are combined with the
corresponding false theta series, then the decomposition is
unconditionally unique.

\end{remark}

\subsection{Resurgent Transseries of Mock Theta Functions}

Rotating back to the line $t\in \mathbb R^+$, we get 
\begin{eqnarray}
  \mathcal L_j^{(p)}(t) &=& Q^{-\Delta_{(p,j)}}\, X_j^{(p)}(Q) + 
    \sqrt{\frac{\pi}{t}} \sum_{k=1}^{\ell_p}  M^{(p)}_{jk} \, Q_1^{\, -\Delta_{(p,k)}}\, X_k^{(p)} \bigl(Q_1\bigr),
    \quad t\in \mathbb R^+
    \label{eq:mock_q2_nonunary}
    \end{eqnarray}
Here $X^{(p)}(q)$ is an $\ell_p$-component vector of holomorphic $q$-series,
\begin{eqnarray}
    X^{(p)}(q) = 
    \begin{pmatrix}
    X_1^{(p)}(q)\\
    \vdots\\
    X_{\ell_p}^{(p)}(q)
    \end{pmatrix}
\end{eqnarray}
In \eqref{eq:mock_q2_nonunary},
 the $j^{th}$ entry of $X^{(p)}$ is multiplied by a rational power  $Q^{-\Delta_{(p,j)}}$, where $\Delta_{(p, j)}$ is defined in \eqref{eq:delta-mock}.

\begin{remark}
    The rational exponents 
    $\Delta_{(p,j)}$ in \eqref{eq:mock_q2_nonunary} are the same as those appearing in \eqref{eq:mock_Q_unary}, as are the mixing matrices $M^{(p)}$.
\end{remark}

\begin{remark}
For the known distinguished mock theta vectors of orders $5$, $7$, and
$11$, the modular $S$ transformation laws take the compact self-dual
form in \eqref{eq:mock_q2_nonunary}, with the same series
$X_j^{(p)}$ appearing in the $Q$ and $S$-dual $Q_1$ terms on the RHS of \eqref{eq:mock_q2_nonunary}. This
self-duality simplifies the analysis and also underlies the numerical
method of \cite{CDGG,ACDGO,Adams:2025qgj}, in which the coefficients of
the $Q$-series are determined by matching \eqref{eq:mock_q2_nonunary}
near the self-dual point $t=\pi$.
\end{remark}

It is convenient to present the relation between the classical mock theta functions and the corresponding solutions $X^{(p)}$ of \eqref{eq:mock_q2_nonunary}. For $p=5,7$ set
\begin{eqnarray}\label{eq: X^(5)}
X^{(5)}(Q)
:=
\left(
-\frac{2}{3}(\chi_0(Q)-2),
-\frac{2}{3}Q\chi_1(Q)
\right)^T,
\end{eqnarray}
and
\begin{eqnarray}\label{eq: X^(7)}
X^{(7)}(Q)
:=
\left(
-\mathcal F_0(Q),
-\mathcal F_1(Q),
-Q\mathcal F_2(Q)
\right)^T.
\end{eqnarray}

For $p=11$, Zagier defines a vector-valued mock modular $\mathcal{M}_{11}= (\mathcal{M}_{11})_{j=1}^5$ by
\begin{eqnarray}
    \mathcal M_{11,j}(\tau) = \frac{1}{\eta(\tau)^3}\,\sum\limits_{\substack{m>2|n|/11 \\ n \equiv j \,(\text{mod } 11)}} \left(\frac{-4}{m}\right)\left(\frac{12}{n}\right) \left( m \,\text{sgn}(n) - \frac{n}{6} \right) \, Q^{m^2/8 - n^2/264},
    \label{eq:zagier11}
\end{eqnarray}
and shows that it admits a completion $\widehat {\mathcal M}_{11}$ transforming as a weight $1/2$ vector-valued modular form under $SL(2, \Z)$ \cite[p. 15]{Zag09}.  We define
\begin{eqnarray}\label{eq: X^(11) M^(11)}
   X^{(11)}(Q)=
\Bigl(
Q^{\tfrac{25}{264}}\mathcal M_{11,5}(\tau),
Q^{\tfrac{1}{264}}\mathcal M_{11,1}(\tau),
Q^{\tfrac{49}{264}}\mathcal M_{11,4}(\tau),
Q^{\tfrac{169}{264}}\mathcal M_{11,2}(\tau),
Q^{\tfrac{361}{264}}\mathcal M_{11,3}(\tau)
\Bigr)^{T} .
\end{eqnarray}

\begin{prop}\label{P:TL-unary}
Assume the above notation. Then for each $p \in \{5,7,11\}$, the mock theta vector $X^{(p)}(Q)$ is a vector of holomorphic $Q$-series in $\mathbb D$ satisfying \eqref{eq:mock_q2_nonunary}. 
\end{prop}
\begin{proof}
For $p=5$ and $p=7$ the proof of this proposition is straightforward algebra using the tables of transformation laws in \cite{GM12}.

For $p=11,$ it follows from the definition of $\mathcal M_{11}$ that the components of $X^{(11)}(Q)$ are holomorphic $Q$-series in $\mathbb D$. Moreover, applying to $p=11$ the same calculation used by Zagier for $p=7$ in \cite[p. 10]{Zag09} shows that $X^{(11)}(Q)$ satisfies the transformation law \eqref{eq:mock_q2_nonunary} for $p=11.$

\end{proof}
\begin{note}
To identify the classical order 7 mock theta vector $(\mathcal{F}_0,\mathcal{F}_1,\mathcal{F}_2)$ with $X^{(7)}$, to bring its transformation laws into the form satisfied by $X^{(7)}$, and to ensure decay of the Mordell-Appell kernel, the term $-2$ on the left-hand side of the transformation law for $\mathcal{F}_0$ in \cite{GM12} must be absorbed into the Mordell-Appell integrals; see   \S\ref{sec:mf7} and \eqref{eq:mf7-identity}.
\end{note}

\section{Statement of Main Results}\label{S:MR}

\subsection{Uniqueness for orders
$3,5,7$, and $11$}

As shown in Proposition~\ref{P:Stokes-decomposition} and further
explained in Remark~\ref{r:C5}, the transformation laws of the mock theta vectors are
already determined by the Mordell--Appell integrals, once the uniqueness results in this
section are established. 
Consider the system \eqref{eq:mock_q2_nonunary}, satisfied for $p=5, 7, 11$, and indeed for all prime $p\ge 5$. We state the order 3 case separately below in Theorem \ref{thm: uniq-mf3}.

Let $(\chi_0,\chi_1)$ be the classical order 5 mock pair and consider their transformation laws given in \S \ref{sec: TL mf5}, eq \eqref{eq:transf-chi}. Let $\mathcal{F}:=(\mathcal{F}_0, \mathcal{F}_1, \mathcal{F}_2)$ be the classical order 7 mock vector, and consider their transformation laws given in \S \ref{sec: TL mf7}, eq \eqref{eq: mf7 TL}.
For $p=11,$ recall Zagier's $\mathcal M_{11}$ in \eqref{eq:zagier11} (see also \S \ref{sec: TL mf11}) and the corresponding normalized vector $X^{(11)}$ in \eqref{eq: X^(11) M^(11)}.

Now note that we have
\begin{eqnarray}
    \ord_0 \mathcal F = (0,1,0), \qquad \ord_0 X^{(11)} = (0,0,1,1,2).
\end{eqnarray}
Fix $0<\theta_0< \pi/2.$ Then we have the following uniqueness results for mock theta functions of order $p=5, 7, 11$.
\begin{thm}\label{Thm: uniq mf5711 classical}
The following uniqueness results hold.
\begin{enumerate}
    \item[(1)]
    The classical order-$5$ mock theta pair $(\chi_0,\chi_1)$ is the
    unique pair of holomorphic $q$-series in $\mathbb D$ satisfying
    the mock modular transformation identities \eqref{eq:transf-chi}.

    \item[(2)]
    The classical order-$7$ mock theta vector
    $\mathcal F=(\mathcal F_0,\mathcal F_1,\mathcal F_2)$ is the unique
    vector of holomorphic functions on $\mathbb D$ satisfying the
    system \eqref{eq: mf7 TL} and the componentwise normalization
    \[
    \ord_0\mathcal F\geq(0,1,0).
    \]
    The same uniqueness conclusion holds under the minimal-growth
    condition
    \begin{eqnarray}\label{eq: growth mf7}
        \mathcal F_j(e^{-t})
        =
        O\bigl(|t|^{-1/2}|q_1|^{-1/42}\bigr),
        \qquad
        t\to0,\quad |\arg t|\leq\theta_0, \qquad j=0,1,2.
    \end{eqnarray}

    \item[(3)]
    The order-$11$ mock theta vector $X^{(11)}(Q)$, defined in
    \eqref{eq: X^(11) M^(11)}, is the unique vector of holomorphic
    $Q$-series in $\mathbb D$ satisfying the system
    \eqref{eq:mock_q2_nonunary} and the componentwise normalization
    \[
    \ord_0 X^{(11)}\geq(0,0,1,1,2).
    \]
    The same uniqueness conclusion holds under the minimal-growth
    condition\\
    \begin{eqnarray}
        X_j^{(11)}(e^{-2t})
        =
        O\bigl(|t|^{-1/2}|Q_1|^{-25/264}\bigr),
        \qquad
        t\to0,\quad |\arg t|\leq\theta_0, \qquad 1\le j\le5.
    \end{eqnarray}
\end{enumerate}
\end{thm}

\begin{remark}
    In \eqref{eq: growth mf7}, the growth is stated for $\mathcal F_j(e^{-t}),$ since the classical system \eqref{eq: mf7 TL} relates $\mathcal F_j(q)$ to $\mathcal F_j(q_1^4)$. See Theorem \ref{thm: general uniq thm} below for an equivalent formulation. 
\end{remark}

For convenience and to facilitate generalization, we formulate the next
results in terms of normalized vectors $X^{(p)}$. This notation
simplifies the proofs and allows a uniform formulation that extends to
arbitrary prime orders. Recall the dictionary from \eqref{eq: X^(5)}, \eqref{eq: X^(7)}, and \eqref{eq: X^(11) M^(11)}.

We set
\begin{eqnarray}\label{eq: normaliz init cond}
\nu^{(5)}=(0,1),\qquad
\nu^{(7)}=(0,1,1),\qquad
\nu^{(11)}=(0,0,1,1,2),
\end{eqnarray}
and
\begin{eqnarray}
\kappa_5=\frac{1}{120},\qquad
\kappa_7=\frac{1}{168},\qquad
\kappa_{11}=\frac{25}{264}.
\end{eqnarray}
Recall Proposition \ref{P:TL-unary}. Theorem~\ref{Thm: uniq mf5711 classical} is an immediate corollary of
the following uniform formulation.

\begin{thm}\label{thm: general uniq thm}
For each $p\in\{5,7,11\}$, the vector $X^{(p)}$ is the unique vector of
holomorphic $Q$-series in $\mathbb D$ satisfying the system
\eqref{eq:mock_q2_nonunary} and the componentwise normalization
\begin{equation}\label{eq: normalization gene uniq thm,}
    \ord_0X^{(p)}\geq\nu^{(p)}.
\end{equation}
The same uniqueness conclusion holds under the minimal-growth condition
\begin{eqnarray}
X_j^{(p)}(e^{-2t})
=
O\bigl(|t|^{-1/2}|Q_1|^{-\kappa_p}\bigr),
\qquad
t\to0,\quad |\arg t|\leq\theta_0, \qquad 1\le j \le \ell_p.
\end{eqnarray}
\end{thm}

The proof of Theorem \ref{thm: general uniq thm} is given in \S \ref{sec:H23 proof}, \ref{sec:H27 proof}, and \ref{sec:H211 proof} for the cases $p=5,7$ and $11$, respectively.

\begin{note}
    Theorem \ref{thm: general uniq thm} for the $p=11$ case shows that there is no optimal mock theta vector of order 11, in the sense of Cheng et al. \cite{Cheng2016}, for which $\kappa_p=1/(24p)$.
\end{note}
\begin{remark}
    For $p=5,$ unlike the absolute uniqueness for the classical pair $(\chi_0,\chi_1)$ established in Theorem \ref{Thm: uniq mf5711 classical}, the normalization \eqref{eq: normalization gene uniq thm,} in Theorem \ref{thm: general uniq thm} is essential. Indeed, if it is relaxed, we show below that a nontrivial homogeneous solution exists. There is no conflict since the classical pair $(\chi_0,\chi_1)$ corresponds to the vector $X^{(5)}$ defined in \eqref{eq: X^(5)}, which satisfies the normalization \eqref{eq: normalization gene uniq thm,} required for uniqueness, $\ord_0 X^{(5)} \ge  (0,1).$
\end{remark}

The proof of Theorem \ref{thm: general uniq thm} uses a Wronskian method, described in \S\ref{sec: Wronskian}, which generalizes to higher values of $p$ for which existence can be shown.

We now state the order-$3$ case separately; see also
Remark~\ref{R:sep-mf3}.
\begin{thm} \label{thm: uniq-mf3}
 There is a unique pair $(f,\omega)$ of functions holomorphic in $\mathbb D$  satisfying the system \eqref{eq:omega tlaw} and \eqref{eq:omega f law}. 
\end{thm}
The proof of this theorem is given in \S\ref{P:mf3}.

\begin{remark} 
We see that, as in order $5$, uniqueness of the order-$3$ pair
$(f,\omega)$ is absolute: no order vector (see
\eqref{eq:order-vector}) needs to be prescribed, in contrast with
orders $7$ and $11$.
\end{remark}

\begin{note} {\rm 
Theorems~\ref{Thm: uniq mf5711 classical} and~\ref{thm: uniq-mf3} show
that, had the order-$7$ and order-$11$ mock theta vectors been
normalized to have zero order vectors, as are the order-$3$ and
order-$5$ vectors, uniqueness would be unconditional for all known
prime-order mock theta vectors satisfying the full SL$(2,\ZZ)$
transformation laws.
}\end{note}

\subsection{The homogeneous system}

Suppose that two distinct vector-valued $Q$-series $X^{(p)}$ satisfy
the mock modular transformation law \eqref{eq:mock_q2_nonunary}. Their
difference then satisfies the homogeneous system
\begin{eqnarray}
Q^{-\Delta_{(p,j)}}H_j(Q)
+
\sqrt{\frac{\pi}{t}}
\sum_{k=1}^{\ell_p}
M^{(p)}_{jk}\,
Q_1^{-\Delta_{(p,k)}}H_k(Q_1)
=0.
\label{eq:mock_q2_nonunary_homog}
\end{eqnarray}

\begin{lemma}\label{lem: homog system}
Every homogeneous $\ell_p$-vector
\[
H=(H_1,\ldots,H_{\ell_p})^T
\]
of functions holomorphic in $\mathbb D$ satisfying
\eqref{eq:mock_q2_nonunary_homog} also satisfies
\begin{eqnarray}
(-1)^j q^{-\Delta_{(p,j)}} H_j(-q)
+\sqrt{\frac{\pi}{t}}
\sum_{k=1}^{\ell_p}
\mathcal N^{(p)}_{jk}\,(-1)^k
q_1^{-\Delta_{(p,k)}}H_k(-q_1)
=0,
\label{eq:mock_minusq_nonunary_homog}
\end{eqnarray}
where the mixing matrix $\mathcal N^{(p)}$ is given by
\begin{eqnarray}
\mathcal N^{(p)}_{jk}
:=
(-1)^{j+k}e^{\pi i\Delta_{(p,j)}}
\bigl(
M^{(p)}{D^{(p)}}^{-1}M^{(p)}D^{(p)}M^{(p)}
\bigr)_{jk}
e^{-\pi i\Delta_{(p,k)}},
\label{eq:N}
\end{eqnarray}
with
\begin{eqnarray}
D^{(p)}
:=
\operatorname{diag}
\bigl(
e^{-2\pi i\Delta_{(p,1)}},
\ldots,
e^{-2\pi i\Delta_{(p,\ell_p)}}
\bigr).
\label{eq: D^(p)}
\end{eqnarray}
\end{lemma}

\begin{remark}
    Let $H_j(z)= \sum_{n\ge 0} a_{j,n} z^n$ be holomorphic in $\mathbb D.$ Lemma \ref{lem: homog system} shows that if $(H_j(Q))_{j}$ satisfies the homogeneous $Q$-system \eqref{eq:mock_q2_nonunary_homog}, then $(H_j(-q))_{j}$ satisfies the homogeneous $(-q)$-system \eqref{eq:mock_minusq_nonunary_homog}. In particular, $(H_j)_{j}$ is a holomorphic solution of both systems. 
\end{remark}

The preceding lemma shows that the homogeneous relation in
\eqref{eq:mock_minusq_nonunary_homog} imposes no additional constraint:
it is already implied by \eqref{eq:mock_q2_nonunary_homog}.  To make
this consequence explicit, let $\mathcal V_{q^2}$ and
$\mathcal V_{-q}$ denote the complex vector spaces of $\ell_p$-vectors
of holomorphic $q$-series
\[
H=(H_1,\ldots,H_{\ell_p})^T
\]
satisfying \eqref{eq:mock_q2_nonunary_homog} and
\eqref{eq:mock_minusq_nonunary_homog}, respectively.  The homogeneous
solution space for the full system
\eqref{eq:mock_q2_nonunary_homog}--\eqref{eq:mock_minusq_nonunary_homog}
is
\begin{eqnarray}\label{eq: V space}
\mathcal V=\mathcal V_{q^2}\cap\mathcal V_{-q}.
\end{eqnarray}
The proof of Lemma \ref{lem: homog system} is given in \S \ref{sec proof of lem hom sys}.
Lemma~\ref{lem: homog system} gives
$\mathcal V_{q^2}\subseteq\mathcal V_{-q}$. Hence, we have the following corollary:

\begin{corollary}
    With the above notation, we have  
    \begin{eqnarray}
        \mathcal V= \mathcal V_{q^2}.
    \end{eqnarray}
    In particular, $\mathcal V=\{0\}$ iff $\mathcal V_{q^2}=\{0\}.$    
\end{corollary}
This fact will be used in our proofs of uniqueness for $p=5, 7, 11$.

\subsection{Homogeneous solutions for prime $p\geq 5$}

For $x \in \Z /p \Z,$ consider the theta functions associated with the Mordell Integrals (see Lemma \ref{lem:int tran mordel theta}):
\begin{eqnarray}\label{eq: Thetapx0}
    \Theta_{p,x}(\tau):= \sum_{\substack{n \in \Z \\ n \equiv x \pmod{p}}} n   \left( \frac{12}{n} \right)\,  Q^{\frac{n^2}{24p}},
\end{eqnarray}
and 
\begin{eqnarray}\label{eq: Phipx0}
    \theta_{p,x}(\tau):= \sum_{\substack{n \in \Z \\ n \equiv x \pmod{p}}} \left( \frac{-4}{n} \right)\,  Q^{\frac{n^2}{8p}}.
\end{eqnarray}
Let
\begin{eqnarray}\label{eq: p11 desired3}
    \Xi_{11,j}(\tau):= \Theta_{11,x_j}(\tau) \theta_{11,5x_j}(\tau) -  \Theta_{11,3x_j}(\tau) \theta_{11,2x_j}(\tau) +  \Theta_{11,4x_j}(\tau) \theta_{11,3x_j}(\tau).
\end{eqnarray}
\begin{thm}\label{thm: homog sol for any p}
    Let $\mathcal D:= \frac{1}{2 \pi i} \frac{d}{d \tau}.$ Let $p\ge 5$ be a prime. Let $x_j=6j-p$ for $1 \le j \le \ell_p$.
    \begin{enumerate}
        \item[(1)] If $p \equiv 1 \pmod{12}$, let $s \ge 1$ with $(s,6p)=1$ and $s^2 \equiv -1 \pmod{p}$, and define 
        \begin{eqnarray}\label{eq: homog p1mod12}
            H_j^{(p)}(Q)= (-1)^j Q^{\Delta_{(p,j)}} \frac{\Theta_{p,sx_j}(\tau)}{\eta(\tau)^2}= \frac{(-1)^j}{(Q;Q)_\infty^2} \sum_{\substack{n \in \Z \\ n \equiv s x_j \pmod{p}}}  n \left( \frac{12}{n} \right)  Q^{\frac{x_j^2+n^2-2p}{24p}}.
        \end{eqnarray}
        \item[(2)] If $p \equiv 5 \pmod{12}$, let $s \ge 1$ with $(s,6p)=1$ and $s^2 \equiv -1 \pmod{p}$, and define 
        \begin{eqnarray}
            H_j^{(p)}(Q)= (-1)^j Q^{\Delta_{(p,j)}} E_4(\tau) \frac{\Theta_{p,sx_j}(\tau)}{\eta(\tau)^{10}} = \frac{(-1)^j E_4(\tau)}{(Q;Q)_\infty^{10}}  \sum_{\substack{n \in \Z \\ n \equiv s x_j \pmod{p}}}  n \left( \frac{12}{n} \right)  Q^{\frac{x_j^2+n^2-10p}{24p}}.
        \end{eqnarray}
        \item[(3)] If $p \equiv 7 \pmod{12}$, let $s \ge 1$ with $(s,2p)=1$ and $3s^2 \equiv -1 \pmod{p}$
        \begin{eqnarray}\label{eq: homog p7mod12}
            H_j^{(p)}(Q)= (-1)^{j+1} Q^{\Delta_{(p,j)}} \frac{1}{\eta(\tau)^3}  \mathcal D \left( \frac{\theta_{p,sx_j}(\tau)}{\eta(\tau)} \right),
        \end{eqnarray}
        \item[(4)] For $p=11,$ define 
    \begin{eqnarray}
        H_j^{(11)}(Q)=(-1)^{j+1} \left( \frac{x_j}{11} \right) Q^{\Delta_{(11,j)}} \frac{1}{\eta(\tau)^3} \mathcal D \left( \frac{\Xi_{11,j}(\tau)}{\eta(\tau)^4} \right).
    \end{eqnarray}
    \end{enumerate}
    Then for every $C \in \C,$ $C H_j^{(p)}(Q)$ is a holomorphic $Q$-series for $|Q|<1$ satisfying the homogeneous system \eqref{eq:mock_q2_nonunary_homog}.
\end{thm} 
The proof of Theorem \ref{thm: homog sol for any p} is given in \S \ref{sec: proof thm homog sol p1}, \ref{sec: proof thm homog sol p2}, \ref{sec: proof thm homog sol p3}, and \ref{sec: proof thm homog sol p4} for items $(1), (2), (3),$ and $(4),$ respectively. 
\begin{remark}
    The construction for $p=11$ generalizes  straightforwardly 
    to (a more complicated formula for) all primes $p \equiv 11 \pmod{12}$. See \S \ref{sec:general-p-homogeneous}.
\end{remark}

\begin{note}
As a corollary, we note that given a homogeneous solution $H_j^{(p)}(Q)$ to \eqref{eq:mock_q2_nonunary_homog}, we automatically have a corresponding homogeneous solution $H_j^{(p)}(-q)$ to the T-dual homogeneous system in \eqref{eq:mock_minusq_nonunary_homog}.
    
\end{note}

\begin{note}
    These homogeneous solutions have faster growth near the unit circle than the ground-state mock thetas (the unique minimal-growth solutions to \eqref{eq:mock_q2_nonunary}). Thus, for every prime $p \ge 5,$ whenever a holomorphic solution $X^{(p)}$ of \eqref{eq:mock_q2_nonunary} exists,  the corresponding homogeneous solution above yields new excited-state mock theta functions of higher growth. In particular, for $p=5,7,11,$ this yields new mock theta functions of orders 5,7, and 11, respectively. Conversely, within this family of excited-state mock thetas, one may subtract the corresponding homogeneous solution to recover the ground-state mock theta.
   
\end{note}

\begin{note}
    For order 3 mock theta functions, no higher growth full SL($2,\ZZ$) solutions exist, by the uniqueness theorem \ref{thm: uniq-mf3}, and by the fact that $(f,\omega)$ already have the highest growth consistent with the modular transformations inherited from the
    unary side (see \cite{CDGG,ACDGO}). In other words, the continuation of higher growth full SL($2,\ZZ$)  order 3 mock theta pairs through the natural boundary is given by non-holomorphic $q^{-1}$ series.
\end{note}

\subsection{New faster-growth mock thetas for $p=5,7,11$}
\label{sec:faster}

As an immediate corollary of Theorem \ref{thm: homog sol for any p}, we obtain a new one-dimensional class of holomorphic mock theta functions of order 5,7, and 11. 

\begin{thm}\label{thm: new mock thetas 5,7,11}
    Let $p \in \{5,7,11\},$ let $X^{(p)}$ be given by \eqref{eq: X^(5)}, \eqref{eq: X^(7)}, and \eqref{eq: X^(11) M^(11)}, and let $H^{(p)}$ be as in Theorem \ref{thm: homog sol for any p}. Then we have the following new one-dimensional class of holomorphic mock theta vectors of order $p$ satisfying the system \eqref{eq:mock_q2_nonunary}:
    \begin{eqnarray}
         X^{(p)}(Q) + c H^{(p)}, \qquad c \in \C.
    \end{eqnarray}
     For $c=0$, we recover the corresponding unique minimal growth mock theta functions. For $c \neq 0,$ these are new mock theta functions of faster growth.
\end{thm}

Let $H^{(5)}$ be given by Theorem \ref{thm: homog sol for any p}-(2) with $p=5$ and $s=7$, and we have 
\begin{subequations}\label{eq: H^(5)}
 \begin{eqnarray}
    H_1^{(5)}(Q)= 7+ 1763 Q+ 35642 Q^2 + \cdots
\end{eqnarray}
and 
\begin{eqnarray}
    H_2^{(5)}(Q)= -1 -261 Q -7375 Q^2 + \cdots.
\end{eqnarray}   
\end{subequations}
In view of uniqueness, i.e. Theorem \ref{thm: general uniq thm} for the $p=5$ case, more is true for the order 5 mock thetas. 
\begin{prop}\label{prop general sol X_5}
Let $H^{(5)}$ be given by Theorem \ref{thm: homog sol for any p}-(2). Then the general pair of holomorphic $Q$-series $X^{(5)}(Q)$ satisfying the system \eqref{eq:mf5-identity} is given by Theorem \ref{thm: new mock thetas 5,7,11}.
\end{prop}
The proof of Proposition \ref{prop general sol X_5} is given in \S \ref{sec: proof prop gen sol X5}.

\begin{remark}
    Let $\chi^{G_2}(\tau)= \bigl(\chi^{G_2}_{2/5}(\tau), \chi^{G_2}_{0}(\tau) \bigr)^{T} $ be the vector-valued character defined by \cite[p. 24, eq (5.10)]{harvey2018hecke} and \cite[eq (5.7), (2.25), (2.26)]{harvey2018hecke}, and consider its transformation laws in \cite[p. 24, eq (5.8)-(5.9)]{harvey2018hecke}. Let $\chi^{E_8}$ be the character given in \cite[p. 22, eq (5.1)]{harvey2018hecke}\footnote{We note the small typo in \cite[p. 22, eq (5.1)]{harvey2018hecke}, where $\rho^{E_8}(T)$ should have $e^{-2 \pi i /3}$ rather than $e^{2 \pi i /3}$.}. \\
     These characters are defined in Appendix \ref{sec:appendix1}.     \\
    
    Interestingly, by Proposition \ref{prop general sol X_5}, it follows that $H^{(5)}$ has the alternative form:
    \begin{eqnarray}
        \tilde H^{(5)} (Q):= (Q;Q)_{\infty}  
        \begin{pmatrix}
            Q^{1/20} \chi^{E_8}(\tau) \chi^{G_2}_{2/5}(\tau) \\
            -  Q^{9/20} \chi^{E_8}(\tau)\chi^{G_2}_{0}(\tau)
        \end{pmatrix},
        \label{eq: H non uniq 5}
    \end{eqnarray}
    This follows because  $H$ is holomorphic in $\mathbb D$, is a homogeneous solution to \eqref{eq:mf5-identity}, and satisfies $\tilde H_2^{(5)}(0)=-1=H_2^{(5)}(0).$
\end{remark}

Similarly, let $H^{(7)}$ be given by Theorem \ref{thm: homog sol for any p}-(3) with $p=7$ and $s=3$, and we have 
\begin{subequations}\label{eq: H^(7)}
\begin{eqnarray}
    H_1^{(7)}(Q)= 5+ 62 Q+ 275 Q^2 + \cdots
\end{eqnarray}
\begin{eqnarray}
    H_2^{(7)}(Q)= 1 -38 Q -280 Q^2 + \cdots,
\end{eqnarray}
and 
\begin{eqnarray}
    H_3^{(7)}(Q)= 17Q +51 Q^2 + \cdots.
\end{eqnarray}
\end{subequations}

\begin{prop}\label{prop: general sol F_7}
Let $H^{(7)}$ be given by Theorem \ref{thm: homog sol for any p}-(3). Let $\bigl(\mathcal F_0(Q), \mathcal F_1(Q), \mathcal F_2(Q)\bigr)$ be the classical order 7 mock theta vector, and consider its transformation laws given in \S \ref{sec: TL mf7}, eq \eqref{eq: mf7 TL}. Then the general vector of holomorphic $Q$-series $\widetilde{\mathcal F}(Q)$ satisfying the system \eqref{eq: mf7 TL} is given by
\begin{eqnarray}\label{eq: general sol F_7}
    \widetilde{\mathcal F}(Q)=
    \begin{pmatrix}
        \mathcal F_0(Q) \\
        \mathcal F_1(Q) \\
        \mathcal F_2(Q)
    \end{pmatrix}
    + c \begin{pmatrix}
        H_1^{(7)}(Q) \\
        H_2^{(7)}(Q) \\
        H_3^{(7)}(Q)/Q
    \end{pmatrix}
    \qquad c \in \C.
\end{eqnarray}
\end{prop}
The proof of Proposition \ref{prop: general sol F_7} is given in \S \ref{sec: prop gen sol F7 proof}.
\begin{remark}
    Since $H_2^{(7)}(0) \neq 0,$ the normalization $\widetilde{\mathcal F}_1(0)=0$ forces $c=0,$ so that $\widetilde{\mathcal F}(Q)=\bigl(\mathcal F_0(Q), \mathcal F_1(Q), \mathcal F_2(Q)\bigr)$ is the unique holomorphic solution to the system \eqref{eq: mf7 TL} with the property $\widetilde{\mathcal F}_1(0)=0$. 
\end{remark}
\begin{remark}
    Let $\chi_0(\tau), \chi_{1/7}(\tau), \chi_{3/7}(\tau)$ be the characters defined as \cite[p. 11, eq (2.34)]{harvey2018hecke}\footnote{Do not confuse this character $\chi_0(\tau)$ with the order 5 mock theta function in \S \ref{sec: TL mf5}.}, and consider their transformation laws in \cite[p. 28, eq (5.27)]{harvey2018hecke}. Interestingly, by the proof of Proposition \ref{prop: general sol F_7}, it follows that $H^{(7)}$ has the alternative form  (cf. Appendix \ref{sec:appendix1}):

    \begin{eqnarray}
        \tilde H^{(7)}(Q)= \frac{42}{2 \pi i} (Q;Q)_\infty^{-3} 
        \begin{pmatrix}
            Q^{-5/42} \chi_{1/7}'(\tau) \\
            -Q^{1/42} \chi_{0}'(\tau) \\
            Q \bigl(Q^{-17/42} \chi_{3/7}'(\tau) \bigr)
        \end{pmatrix}.
        \label{eq:p7-H}
    \end{eqnarray}
    This follows because $H^{(7)}$ is holomorphic on $\mathbb D$, is a homogeneous solution to \eqref{eq:mf7-identity}, and satisfies $\tilde H^{(7)}_3(0)=0$ and $\tilde H^{(7)}_2(0)=1=H_2^{(7)}(0)$.
\end{remark}

\begin{remark}
    The normalization in Theorem \ref{thm: general uniq thm} for the $p=11$ case is necessary. Indeed, Theorem \ref{thm: homog sol for any p}-(4) gives a holomorphic homogeneous solution $H^{(11)}$ with
    \begin{align}
        H^{(11)}_1(Q)&= 1+54 Q+ 430 Q^2 +\cdots \\
        H^{(11)}_2(Q)&= -2 -80Q -485 Q^2 +\cdots \\
        H^{(11)}_3(Q)&= 1+3 Q -87 Q^2 +\cdots \\
        H^{(11)}_4(Q)&=24 Q+ 219 Q^2 +\cdots \\
        H^{(11)}_5(Q)&= -25 Q^2 -220 Q^3 +\cdots.
    \end{align}
    Then Theorem \ref{thm: new mock thetas 5,7,11} gives new order 11 mock theta functions for every $c \neq 0$. Moreover, since $H^{(11)}_3(0)=1$ while $X_{3}^{(11)}(0)=0$, the normalization forces $c=0$.
\end{remark}

\subsection{Natural boundary crossing}
\label{sec:boundary}

The problem of continuing mock theta functions across their natural
boundary has led to a number of   heuristic prescriptions in the literature,
often depending on a particular $q$-series or $q$-Pochhammer
representation \cite{mortenson,bringmann-inverse,Cheng:2018vpl}.

We give two intrinsic constructions. The first construction is
unique continuation of the modular transformation laws (see also Remark \ref{R:Core-principle}), regarded as
functional equations under $S$ and $T$;  in fact, the $S$
transformation alone suffices. 

The second construction is based on resurgence: we replace $\tau$ by
$-\tau$ in the resurgent asymptotic series at the cusp and then
\'Ecalle--Borel sum it.

We prove that these two constructions coincide.
The first construction is rigid and explicit.  For each distinguished
mock theta vector of order $3,5,7,11$, its modular transformation laws,
viewed on the  unary side $|Q|>1$, admit exactly one holomorphic solution.
No growth condition or order normalization is required.  This solution
is an explicit vector of unary false theta series in $Q^{-1}$,
multiplied by rational powers of $Q$.

Here and below, the rational powers and the corresponding unary series
are regarded together.  Separating them introduces the trivial freedom
of shifting a rational exponent by a nonnegative integer and
compensating by the opposite shift of all powers in the unary series;
the resulting series in rational powers of $Q^{-1}$ is unchanged.

\begin{note}{\rm 
The use of uniqueness of solutions as a means of continuation across,
or along, a natural boundary has precedents in other settings. For
example, \cite{CostinHuang} shows that, generically, the positive time
axis is a natural boundary for solutions of the Schr\"odinger equation,
even for analytic initial data, whereas the solution is known to exist
and to be unique. In the present setting, by contrast, the uniqueness
needed to define the continuation through the modular functional
equations is itself a nontrivial result. Moreover, the continuation so
obtained preserves the $SL(2,\ZZ)$ structure and coincides with the
independently defined resurgent continuation.}
\end{note}
\vspace{0.4cm}

Let \begin{equation}
  \label{eq:eqzeta}
 \zeta=Q^{-\frac{1}{24p}};\quad  \zeta_1=Q_1^{-\frac{1}{24p}}=[Q^{-\frac{1}{24p}}](-1/\tau)
\end{equation}
(cf. \S\ref{sec:notation}), where $\tau$ is in the {\em lower} half plane when $|Q^{-1}|<1$.
In this region, we have
\begin{equation}
  \label{eq:eqt}
  \tau\to -1/\tau \Leftrightarrow t\mapsto \pi^2/t
\end{equation}

In vector form, \eqref{eq:mock_Q_unary} takes the form
\begin{equation}
\label{eq:mock_q2_unary}
\mathcal L^{(p)}(t)
=
F^{(p)} (\zeta)
+
i\sqrt{\frac{\pi}{|t|}}\,
M^{(p)}F^{(p)}(\zeta_1),
\qquad t\in\mathbb R^-,
\end{equation}
where $F^{(p)}$ is holomorphic in the unit disk.
\begin{thm}[Absolute uniqueness beyond the natural boundary]
\label{T:abs-uniq}
Equation~\eqref{eq:mock_q2_unary} has a unique solution $F^{(p)}$
holomorphic in the unit disk.  Equivalently, when the rational
prefactors and the corresponding series in $Q^{-1}$ are regarded
together, the modular transformation laws admit a unique holomorphic
solution on the side $|Q|>1$. (No growth condition or order normalization is needed.)
\end{thm} 
\begin{note}
{\rm 
If the rational prefactor is separated from the unary series, there is
a trivial normalization freedom: one may shift its exponent by a
nonnegative integer and compensate by the opposite shift of all powers
in the unary series.  This does not change the solution viewed as a
series in rational powers of $Q^{-1}$.}
\end{note}

\begin{proof}
Let $F_1$ and $F_2$ be two solutions of
\eqref{eq:mock_q2_unary}, and set $G=F_1-F_2$. Then
\begin{equation}
\label{eq:homogeneous-outside}
G(t)
=
-i\sqrt{\frac{\pi}{|t|}}\,
M^{(p)}G(\pi^2/t).
\end{equation}
Replacing $t$ by $\pi^2/t$ gives
\[
G(\pi^2/t)
=
-i\sqrt{\frac{|t|}{\pi}}\,
M^{(p)}G(t).
\]
Substitution into \eqref{eq:homogeneous-outside}, together with
$(M^{(p)})^2=I$, yields
\[
G(t)
=
(-i)^2(M^{(p)})^2G(t)
=
-G(t).
\]
Hence $G(t)=0$, and therefore $F_1=F_2$. In the same way, for the pair $(f,\omega)$, the first equation for a possible homogeneous solution,
\[(q^{-1})^{2/3}\,G_1(-q^{-1})
=
-i\sqrt{\frac{\pi}{t}}\,(q_1^{-1})^{2/3}\,G_1(-q_1^{-1})
 \qquad t\in \RR^-\]
 has the only solution $G_1=0$.
\end{proof}

\begin{note}
{\rm
As we see, when formulated in terms of vectors of series in rational
powers of $Q^{-1}$, uniqueness on the unary side $|Q|>1$ is always absolute:
no order vector needs to be prescribed. There remains, however, a
trivial nonuniqueness in the false theta representation: a false theta
series may be multiplied by a nonnegative integer power of $Q^{-1}$,
with the corresponding power absorbed into the rational prefactor.
This freedom matters when, by unique continuation of
solutions, we associate a false theta vector with a mock theta vector.}
\end{note}

We recall that Proposition~\ref{P:Stokes-decomposition} gives the
transformation laws of the mock theta vectors on the unary side
$|Q|>1$, up to the freedom of changing the order vector of the false
theta series and compensating by corresponding integer shifts of the
$\Delta_{(p,j)}$. Such shifts change the normalization of the
transformation laws, but leave unchanged the corresponding equations
when written directly in terms of series in rational powers of
$Q^{-1}$.

\begin{thm}[Rigid and explicit crossing]
\label{T:unique-crossing}
For each distinguished mock theta vector of order $3,5,7,11$, the
modular transformation laws admit a unique holomorphic solution on the
side $|Q|>1$. These solutions are explicit unary false theta vectors:
\begin{enumerate}
    \item For order $3$, the pair
    \[
    \bigl(\Phi_1^{(3)}(Q^{-1}),\Phi_2^{(3)}(Q^{-1})\bigr)
    \]
    is the unique false theta pair satisfying
    \eqref{eq:omega tlaw vee} and \eqref{eq:omega f law vee}.

    \item For order $5$, the vector
    \[
    \bigl(\Phi_1^{(5)}(Q^{-1}),\,Q\Phi_2^{(5)}(Q^{-1})\bigr)
    \]
    is the unique false theta vector satisfying
    \eqref{eq:mock_Q_unary} for $p=5$.

    \item For order $7$, the vector
    \[
    \bigl(\Phi_1^{(7)}(Q^{-1}),\,Q\Phi_2^{(7)}(Q^{-1}),\,
    Q\Phi_3^{(7)}(Q^{-1})\bigr)
    \]
    is the unique false theta vector satisfying
    \eqref{eq:mock_Q_unary} for $p=7$.

    \item For order $11$, the vector
    \[
    \bigl(\Phi_1^{(11)}(Q^{-1}),\,\Phi_2^{(11)}(Q^{-1}),\,
    Q\Phi_3^{(11)}(Q^{-1}),\,Q\Phi_4^{(11)}(Q^{-1}),\,
    Q^2\Phi_5^{(11)}(Q^{-1})\bigr)
    \]
    is the unique false theta vector satisfying
    \eqref{eq:mock_Q_unary} for $p=11$.
\end{enumerate}
\end{thm}

\begin{proof}
The result follows immediately from the absolute uniqueness in
Theorem~\ref{T:abs-uniq}, together with the false theta
representations and transformation laws in
Proposition~\ref{P:TL-unary}.
\end{proof}
\begin{remark} 
Therefore, the distinguished mock theta functions have a rigid and explicit
continuation through the natural boundary: their transformation laws
admit exactly one solution on the unary  side, and that solution is an
explicit unary false theta vector.  The theorem below shows that this
unique continuation is precisely the one obtained independently by
resurgence.
\end{remark}
\begin{thm}
The correspondence between mock theta vectors $m$ and false theta
vectors $\Phi$ given by unique continuation in
Theorem~\ref{T:unique-crossing} coincides with the one obtained by
resurgence. More precisely, consider the asymptotic series at the cusp of a
mock theta vector with prefactors,
\[
\bigl(
Q^{r_1}m_1,\ldots,Q^{r_n}m_n
\bigr).
\]
Then
\begin{enumerate}
    \item replace $t$ by $-t$ in these asymptotic series at the cusp;
    \item apply \'{E}calle--Borel summation to the resulting series;
    \item identify the resulting principal-value integrals as described
    in \S\ref{S:indenf}.
\end{enumerate}
The resulting vector is
\[
\bigl(
Q^{r_1}\Phi_1,\ldots,Q^{r_n}\Phi_n
\bigr),
\]
where $\Phi=(\Phi_1,\ldots,\Phi_n)$ is the false theta vector associated
with $m$ by the unique continuation of
Theorem~\ref{T:unique-crossing}.
\end{thm}
\begin{proof}
By Proposition~\ref{P:sign-change}, the substitution $t\mapsto -t$
takes the asymptotic series at the cusp to the asymptotic series of the
corresponding principal-value integrals on the other side of the boundary.
Proposition~\ref{P:Stokes-decomposition} identifies their unique
Stokes-line transseries decomposition in terms of the associated false
theta vectors, while Proposition~\ref{P:TL-unary} identifies the resulting
transformation laws with those used in Theorem~\ref{T:unique-crossing}.
The conclusion then follows from the linearity of
\'{E}calle--Borel summation \cite{costin-book}.
\end{proof}
\begin{remark}
{\rm 
Solutions of higher than minimal growth do not have a unique
continuation through the natural boundary. Once the minimal-growth
condition is relaxed, the transformation laws admit distinct
holomorphic solutions on the side $|Q|<1$, whereas on the unary side
the corresponding solution is unconditionally unique.

This loss of uniqueness is also reflected in the resurgent
description. All the higher-growth mock theta vectors, taken together
with their rational-power prefactors, have the same cusp asymptotic
series as the minimal-growth solution. Indeed, the difference of two
distinct solutions of the transformation laws satisfies the associated
homogeneous system, and uniqueness of the transseries decomposition
implies that its asymptotic series vanishes. Thus the cusp asymptotic
series cannot distinguish the higher-growth solutions at the boundary,
consistently with the absolute uniqueness of the false theta vector on
the unary side.
}
\end{remark}
\subsection{Modular transformation laws of vector-valued Mordell integrals.}
\label{sec:modular}

In this section we define vectors of Mordell integrals which satisfy special modular transformation laws which generalize the modular transformation laws of the known order 5 and order 7 mock theta functions to ones naturally labeled by any prime $p\geq 5$. 

In sections \ref{sec:mf$_3$}--\ref{alt P:mf3} we discuss a different modular structure, based on a different class of vector-valued Mordell-Appell integrals, associated with order 3 mock theta functions and their natural generalizations \cite{ACDGO,Adams:2025qgj}.

\begin{thm}
   For the mock theta vectors of orders $3,5,7,$ and $11$, the vectors of Mordell-Appell vectors defined in \eqref{eq:l} generate the false theta  continuations beyond the boundary, uniquely so if combined with the rational power of $q$ multiplying them.  
\end{thm}
\begin{proof}
    This result follows immediately from Proposition \ref{P:Stokes-decomposition}.
\end{proof}

 \subsubsection{The Modular $S$ transformation.} 
 \label{sec:modularS}

We note that the building blocks $JS$ and $JC$ kernels (defined in \eqref{eq:JS def}-\eqref{eq:JC def}) are even functions, decaying at infinity.  Let
\[
\widehat F(\xi)
=
\int_{-\infty}^{\infty}
F(u)e^{-iu\xi}\,du.
\]
\begin{prop}
    We have
    \[
\int_{0}^{\infty}
e^{-c^2u^2/ t}F(u)\,du
=
\frac{1}{2\pi}
\frac{\sqrt{\pi t}}{c}
\int_{0}^{\infty}
e^{- t\xi^2/(4c^2)}
\widehat F(\xi)\,d\xi.
\]
\end{prop}
\begin{proof}
Using parity we write the integrals as half of the integral on the real line. Note that $F$ and the Gausssian factor are even. The result then follows from Parseval's identity and the fact that  \(
\widehat{
e^{-c^2u^2/ t}
}(\xi)
=
\frac{\sqrt{\pi t}}{c}
e^{- t\xi^2/(4c^2)}
\).
\end{proof}

More generally, we have
    \begin{lemma} Assume $F\in L^1(\RR^+)$. Then, for $c, t$ positive, we have
\begin{equation}\label{eq:fourier-cosine}
\int_0^\infty
e^{-c^2u^2/ t}F(u)\,du
=
2\sqrt{\frac{ t}{\pi}}\int_0^\infty
e^{- t y^2}
\left(
\int_0^\infty
\cos(2cuy)F(u)\,du
\right)dy.
\end{equation}
This identity generalizes to $F$ analytic in a sector in the right half plane with sufficiently slow exponential growth, by  writing $\cos(x)=\frac12 (e^{ix}+e^{-ix})$ and rotating the contour of the integrals of $e^{\pm ix}$ in the upper/lower half plane respectively.
\end{lemma}
\begin{proof}
    The $L^1$ case follows from using Fubini and the elementary identity 
    \[ \int_0^\infty e^{- t y^2} \cos(2 c u y)dy =\frac12\sqrt{\frac{\pi}{t}} e^{-c^2 u^2/ t}\]
    The more general case is by first decomposing the cosine by exponentials and performing the contour rotation for analytic $F$ with sufficient decay in a sector, and then by density.
\end{proof}

\begin{lemma}
The basis Mordell-Appell integrals satisfy the following identities related to the modular $S$ transformation $S:t\to \pi^2/t$. We write the Mordell-Appell kernels as $ S_{p,a}(u)=\sinh((p-a)u)/\sinh(pu)$,  $C_{p,a}(u)=\cosh((p-a)u)/\cosh(pu)$, $ \tilde{S}_{p,a}(u)= (\pi p t)^{-1/2}\sin(p^{-1}\pi a)[\cosh(2u)-\cos(\frac{\pi a}{p})]^{-1}$ and $\tilde{C}_{p,a}(u)=2 \cosh(u) \tilde{S}_{p,a}(u)$. Then
\begin{eqnarray}
\frac{1}{t}
\int_0^\infty
e^{-\frac{p\, u^2}{t}}
S_{p,a}(u) du &=&
\int_0^\infty
e^{-\frac{p\, u^2\, t}{\pi^2}}
\tilde{S}_{p,a}(u) du
\label{eq:sinh}
\\
\frac{1}{t}
\int_0^\infty
e^{-\frac{p\, u^2}{t}}
C_{p,a}(u) du &=&
\int_0^\infty
e^{-\frac{p\, u^2\, t}{\pi^2}}
\tilde{C}_{p,a}(u) du.
\label{eq:cosh}
\end{eqnarray}
\end{lemma}

 \begin{proof}
     For \(s>0\) and \(|b|<1\), we have
\[
\frac{\sinh(bs)}{\sinh( s)}
=
\sum_{n=0}^\infty
\left(
e^{-(2n+1-b)s}
-
e^{-(2n+1+b)s}
\right).
\]
Integrating term by term and using the Mittag--Leffler decomposition
\begin{equation}\label{eq:ML-dec}
2y\sum_{k=0}^\infty
\frac{1}{\pi^2k^2+y^2}
-\frac{1}{y}
=
\coth y,
\qquad y\ne0,
\end{equation}
we obtain
\[
\int_0^\infty
\frac{\cos(vs)\sinh(bs)}{\sinh( s)}\,ds
=
\frac{\pi}{2}
\frac{\sin(\pi b)}
{\cosh(\pi v)+\cos(\pi b)}.
\]
Applying this identity with
\(
b=1-\frac{a}{p}
\), where $p>0$ and $0<a<2p$, we obtain the sinh transformation law in \eqref{eq:sinh}. 
The cosh transformation \eqref{eq:cosh} is obtained similarly.

 \end{proof}

\subsubsection{The Modular $T$ transformations.}
\label{sec:modularT}
Let \(P=6p\).  At the common Gaussian scale of  the Stokes-line kernels, this transformation is best seen as a kernel transform in the Gaussian integral. 
\begin{prop}\label{P:trT}
  We have  
  \[
TS_{P,a}=C_{P,a},\qquad TC_{P,a}=S_{P,a}.
\]
\end{prop}
The proof of this proposition is given in \S\ref{S:trT}.
 
\begin{remark}
   Based on the $S,T$ transformations of the Mordell-Appell integrals, it is not difficult to check that any one component of the vector-valued Mordell-Appell integral in \eqref{eq:l} generates all the others.  A similar statement holds for the order 3 Mordell-Appell integrals in \eqref{eq:w3}-\eqref{eq:w2}.
\end{remark}

\section{Proofs}\label{S:proofs}

\subsection{Proof of Lemma \ref{lem: homog system}}\label{sec proof of lem hom sys}

\begin{proof}
For $z\in \mathbb H,$ write $Q(z):= e^{2 \pi i z}$ and $Q(z)^r= e^{2 \pi i r z}$ for $r \in \Q$. Let $H=(H_1, \cdots, H_{\ell_p})^{T}$ be an $\ell_p$-vector of  holomorphic power series in $\mathbb D$ satisfying \eqref{eq:mock_q2_nonunary_homog}, and set
    \begin{eqnarray}\label{eq: tilde H}
        \tilde H(z):= \Bigl( Q(z)^{-\Delta_{(p,1)}} H_1 \bigl ( Q(z) \bigr), \cdots, Q(z)^{-\Delta_{(p,\ell_p)}}  H_{\ell_p} \bigl ( Q(z) \bigr) \Bigr)^{T}.
    \end{eqnarray}    
    Then the homogeneous system \eqref{eq:mock_q2_nonunary_homog} can be rewritten as
    \begin{eqnarray} \label{eq: S law for tilde H}
        \tilde H(z)=- (-iz)^{-1/2} M^{(p)} \tilde H \bigl(-1/z\bigr), \qquad z \in \mathbb H.
    \end{eqnarray}
    On the other hand, since the functions $H_j(Q)$ are holomorphic $Q$-series, we know that $ H_j \bigl ( Q(z+1) \bigr)=  H_j \bigl ( Q(z) \bigr)$ for every $1\le j \le \ell_p.$ Therefore, it follows by \eqref{eq: tilde H} that
    \begin{eqnarray} \label{eq: T law for tilde H}
        \tilde H(z+1)= D^{(p)} \tilde H(z),
    \end{eqnarray}
    where $D^{(p)}$ is given in \eqref{eq: D^(p)}. We now make the following observation. To relate $-q= -e^{\pi i \tau}$ to $-q_1= -e^{\pi i (-1/\tau)},$ it suffices to relate $Q(z)$ to $Q(\gamma z),$ where 
    \begin{eqnarray}
        z= \frac{\tau +1}{2}, \qquad \gamma:= 
        \begin{pmatrix}
            -1 & 1 \\
            -2 & 1
        \end{pmatrix},
    \end{eqnarray}
    since, with this choice of $z$ and $\gamma$, one has  
    \begin{eqnarray} \label{eq: Q gamma z}
        Q(z)= -q, \qquad  \gamma z=\frac{1}{2} - \frac{1}{2 \tau}, \qquad Q(\gamma z)= -q_1.
    \end{eqnarray}
    A direct calculation, on the other hand, gives 
    \begin{eqnarray}
        \gamma= S T^{-1} S TS.
    \end{eqnarray}
    Therefore, using the transformation laws \eqref{eq: S law for tilde H} and \eqref{eq: T law for tilde H} for $\tilde H$ under $S$ and $T,$ respectively, we obtain
    \begin{eqnarray} \label{eq: tilde H gamma z}
        \tilde H(z)= - \bigl(-i  (2z-1) \bigr)^{-1/2} M^{(p)} {D^{(p)}}^{-1} M^{(p)} D^{(p)}M^{(p)}    \tilde H \bigl(\gamma z\bigr), \qquad z \in \mathbb H.
    \end{eqnarray}
    Lastly, evaluating \eqref{eq: tilde H gamma z} at $z= \frac{\tau +1}{2}$ for $\tau \in \mathbb H$, using \eqref{eq: tilde H} and \eqref{eq: Q gamma z}, multiplying the $j$-th component by $(-1)^j e^{\pi i \Delta_{(p,j)}},$ and using the definition of $\mathcal N^{(p)}$ in \eqref{eq:N}, we obtain
    \begin{eqnarray} \label{eq: H satisfies -q}
        (-1)^j q^{-\Delta_{(p,j)}}\, H_j(-q) + \sqrt{\frac{\pi}{t}} \sum_{k=1}^{\ell_p}  \mathcal N^{(p)}_{jk}\, (-1)^{k} q_1^{\,- \Delta_{(p,k)}}\, 
        H_k \left(-q_1\right)=0,
    \end{eqnarray}
    as desired. This completes the proof.
\end{proof}

\subsection{The Wronskian Method} \label{sec: Wronskian} Define the holomorphic functions on $\mathbb H$ 
\begin{equation}\label{eq:v def}
    \mathcal H_j(Q):=\frac{H_j(Q)}{(Q;Q)_\infty}, \qquad
    v_j(\tau):= Q(\tau)^{-\tilde \Delta_{(p,j)}} \mathcal H_j\bigl(Q(\tau)\bigr), \qquad
    v(\tau):=
    \begin{pmatrix}
    v_1(\tau)\\
    \vdots\\
    v_{\ell_p}(\tau)
    \end{pmatrix},
\end{equation}
where 
\begin{eqnarray} \label{eq: tilde Delta}
    \tilde \Delta_{(p,j)}= \Delta_{(p,j)} + \frac{1}{24}.
\end{eqnarray}
For uniqueness, it suffices to prove $v \equiv 0.$ We first show that the Wronskian of $v$ must be identically zero. Once this is established, we will use these transformation laws, a property of the Wronskian, and some elementary linear algebra to deduce that $v$ is identically zero. 

To begin with, observe that \eqref{eq:mock_q2_nonunary_homog} can be written as
\begin{equation}\label{eq:v M}
    v(\tau)=-M^{(p)} \,v \left(-\frac1\tau\right)
\end{equation}
where $M^{(p)}$ is given by \eqref{eq:M}. Then since $M^{(p)}$ squares to the $\ell_p\times \ell_p$ identity we have
\begin{equation}\label{eq:S v}
    v \left(-\frac1\tau\right)=-M^{(p)}\,v(\tau).
\end{equation}
On the other hand, since $\mathcal H_j(Q)$ are holomorphic $Q$-series, with $Q=e^{2\pi i \tau}$, we have $\mathcal H_j(Q(\tau+1))=\mathcal H_j(Q(\tau))$, and thus \eqref{eq:v def} gives
\begin{equation}\label{eq:T v}
    v(\tau+1)=D\,v(\tau),
\end{equation}
where 
\begin{equation}
    D:=\text{diag} \bigl(e^{-2\pi i \tilde \Delta_{(p,1)}}, \dots, e^{-2\pi i \tilde \Delta_{(p,\ell_p)}} \bigr).
    \label{eq:D}  
\end{equation}
Hence, we now have a vector-valued modular object $v(\tau),$ so we consider its Wronskian determinant and show that, after normalization by a suitable power of $\eta,$ it gives rise to a scalar modular object.\footnote{This is the classical modular Wronskian method \cite[Section 3.2]{FrancMason}. For properties of Wronskians, see \cite[p. 210-218]{Weld}} Indeed, consider the Wronskian
\begin{equation}
    W(\tau):=\det
    \begin{pmatrix}
    v_1(\tau) & v_1'(\tau) & \cdots & v_1^{(\ell_p-1)}(\tau)\\
    v_2(\tau) & v_2'(\tau) & \cdots & v_2^{(\ell_p-1)}(\tau)\\
    \vdots & \vdots & & \vdots\\
    v_{\ell_p}(\tau) & v_{\ell_p}^\prime (\tau) & \cdots & v_{\ell_p}^{(\ell_p-1)}(\tau)
    \end{pmatrix},
    \label{eq:Wronskian}
\end{equation}
where derivatives are taken with respect to $\tau.$ Then since $v_j(\tau)$ are all holomorphic functions on $\mathbb H,$ so is $W(\tau)$.

Now using \eqref{eq:T v}, we obtain 
\begin{equation}\label{eq:W T}
    W(\tau+1)=\det(D)\,W(\tau).
\end{equation}
Similarly, using the chain rule for Wronskians \cite[Ex. 22, p. 217]{Weld} and \eqref{eq:S v}, we get
\begin{equation}\label{eq:W S}
    W \left(-\frac1\tau\right) = \det(-M^{(p)})\,\tau^{\ell_p(\ell_p-1)}W(\tau).
\end{equation}

In later sections, we will specialize to $p=5,7,11,$ and show that the corresponding Wronskian is identically zero. The uniqueness will then follow from the following lemma. 

\begin{lemma}\label{lem: wronskian}
    Assume the above and suppose that $W(\tau) \equiv 0.$ Then there is at most one $\ell_p$-vector $X_p(q)$ of holomorphic $q$-series on $\mathbb D$ satisfying \eqref{eq:mock_q2_nonunary}.
\end{lemma}

\begin{proof}
    We first recall that a family of analytic functions has an identically zero Wronskian only if the functions are linearly dependent \cite[p. 91]{Bocher}. Thus, there exist complex numbers $c_1, \cdots, c_\ell$, not all zero, such that 
    \begin{equation} \label{eq: linear dep}
        c_1 v_1(\tau) + \cdots + c_{\ell_p} v_{\ell_p}(\tau)=0, \qquad \tau \in \mathbb H.
    \end{equation}
    Say $c_{j_1},  \cdots c_{j_m} \neq 0$ where $1\le m\le\ell_p$ and $1\le j_1 < \cdots <j_m\le\ell_p.$ Then observe that \eqref{eq:T v} implies 
    \begin{equation} \label{eq: v D T}
        v_j(\tau +n)= D_{jj}^n v_j(\tau), \qquad \tau \in \mathbb H, \quad 1\le j \le \ell_p, \quad n \ge 0,
    \end{equation}
    where $D_{jj}= e^{-2\pi i \tilde \Delta_{(p,j)}}$ are the diagonal entries of $D$, defined in \eqref{eq:D}. For $1\le j\le\ell_p$, the numbers $D_{jj}$ are all nonzero, not equal to $1$, and pairwise distinct.  This follows from the fact that $p$ is prime, $p \ge 5$, and $1\le j\le \ell_p :=(p-1)/2$. 
    
    Using \eqref{eq: v D T} in \eqref{eq: linear dep}, we obtain 
    \begin{equation}\label{eq: cDv}
        \sum_{k=1}^m c_{j_k} D_{j_k, j_k}^n v_{j_k}(\tau) =0,  \qquad \tau \in \mathbb H, \quad 0 \le n \le m-1.
    \end{equation}
    Rewriting this system in terms of the corresponding invertible Vandermonde matrix, we see that 
    \begin{equation} \label{eq: cvj}
        c_{j_k} v_{j_k}(\tau)= 0, \qquad \tau \in \mathbb H, \quad 1 \le k \le m.
    \end{equation}
    Therefore, since $c_{j_k} \neq 0,$ this gives $v_{j_k} \equiv0$ for all $1\le k \le m.$ In particular, $v_{j_1} \equiv0.$
    
    Hence taking the $j_1^{th}$ component of \eqref{eq:S v}, we obtain 
    \begin{equation}
        \sum_{k=1}^{\ell_p} M_{j_1, k} v_k(\tau) = -v_{j_1}(-1/\tau) \equiv 0.
    \end{equation}
    However, $M_{j_1, k} \neq0$ for all $1\le k \le \ell_p,$ so it follows by repeating the same argument as in \eqref{eq: cDv} and \eqref{eq: cvj}, that $v_k \equiv0$ for all $1 \le k\le\ell_p.$ This completes the proof.
\end{proof}

\subsection{Construction of homogeneous solutions for all $p\ge 5$}
\label{sec:general-p-homogeneous}
In this section, we prove Theorem \ref{thm: homog sol for any p} by explicit construction. We show that the Mordell integrals \eqref{eq:ls}, viewed as the central objects in this theory, naturally give rise to an explicit homogeneous solution of the system \eqref{eq:mock_q2_nonunary_homog} for every prime $p \ge 5$. Our proof is constructive: we first identify theta series associated with the Mordell integrals and then use their transformation laws to construct explicit homogeneous solutions.

\subsubsection{Theta functions} \label{sec: mordel theta func} We begin by representing the Mordell integrals \eqref{eq:ls} as integral transforms of theta series, which will play an important role in the construction of homogeneous solutions.

Consider the hyperbolic kernels of each of the Mordell integrals \eqref{eq:JS def} appearing in \eqref{eq:ls}. For each term $JS_{(12p,2a)}(t)$ in \eqref{eq:ls}, expanding its kernel as a geometric series, we get, for $u>0$ and $0<a<12 p,$ 
\begin{align}
    \frac{\sinh((12p-2a)u)}{\sinh(12pu)} &= \sum_{n \ge 0} \Bigl( e^{-2(a+12pn)u} - e^{-2(12p-a+12pn)u} \Bigr)  \\
    &= \sum_{n \ge 1} \Bigl( \mathbb 1_{n \equiv a \pmod{12 p}} - \mathbb 1_{n \equiv -a \pmod{12 p}} \Bigr) e^{-2nu}.
    \label{eq: sinh geo exp}
\end{align}
Let $x_j=6j-p$. Then the parameters $a$ corresponding to the Mordell terms $JS_{(12p,2a)}(t)$ in \eqref{eq:ls} are, respectively, 
\begin{eqnarray}\label{eq: aij0}
    a_{1,j} = |x_j|, \qquad a_{2,j} = 6p-x_j, \qquad  a_{3,j} = 2p+x_j, \qquad a_{4,j} = 4p-x_j.
\end{eqnarray}
Therefore, after multiplying the first hyperbolic kernel by the term $\text{sign}(x_j)$ that comes from \eqref{eq:ls} \footnote{
Note that $\text{sign}(x_j)\Bigl( \mathbb 1_{n \equiv |x_j| \pmod{12 p}} - \mathbb 1_{n \equiv -|x_j| \pmod{12 p}} \Bigr) = \mathbb 1_{n \equiv x_j \pmod{12 p}} - \mathbb 1_{n \equiv -x_j \pmod{12 p}}.$
}, the sum of the four hyperbolic kernels in \eqref{eq:ls} becomes 
\begin{eqnarray}
    \sum_{n \ge 1} \varepsilon_p^{(j)}(n) e^{- 2 nu}.
\end{eqnarray}
Here, $\varepsilon_p^{(j)}(n)$ is a sum of indicator functions as in \eqref{eq: sinh geo exp}. However, since $p \ge 5,$ the Chinese remainder theorem gives $ \mathbb 1_{n \equiv a \pmod{12 p}}=  \mathbb 1_{n \equiv a \pmod{12}}  \mathbb 1_{n \equiv a \pmod{p}}$. Therefore, studying the corresponding residue classes modulo $12$ and $p$ separately, one arrives at the simpler expression 
\begin{eqnarray}
    \varepsilon_p^{(j)}(n)= (-1)^j \left( \frac{12}{p} \right) \left( \frac{12}{n} \right) \left( \mathbb 1_{n \equiv x_j \pmod{p}} - \mathbb 1_{n \equiv -x_j \pmod{p}} \right).
\end{eqnarray}
Hence, the Mordell integral \eqref{eq:ls} becomes
\begin{eqnarray}\label{eq: mordell psipj}
     L_j^{(p)}(t)= \frac{1}{t} \int_0^\infty e^{-12 p u^2/t} \sum_{n \ge 1} \varepsilon_p^{(j)}(n) e^{-2 nu} du.
\end{eqnarray}
Note that $\varepsilon_p^{(j)}(n)$ is periodic and hence bounded. Thus, using the standard identity 
\begin{eqnarray}
    e^{-2 n u}= \frac{n}{\sqrt{\pi}} \int_0^\infty x^{-1/2} e^{-n^2 x -u^2/x} dx, 
\end{eqnarray}
the smoothing $e^{-\varepsilon n^2}$ in the integrand of \eqref{eq: mordell psipj}, Fubini's theorem, and then sending $\varepsilon \to 0^+,$ we obtain 
\begin{eqnarray}\label{eq: mordel theta}
     L_j^{(p)}(t)= \frac{1}{24p} \int_0^\infty \frac{ \sum_{n \ge 1} n \varepsilon_p^{(j)}(n) e^{-\frac{
     n^2tv}{12p}}}{\sqrt{1+v}} dv.    
\end{eqnarray}
Similar expressions, involving different Mordell integrals and different theta series, were found by Zwegers in his thesis \cite{Zwe08}. For an alternative proof of \eqref{eq: mordel theta}, one may use the methods of \cite[Thm. 1.16, p. 18]{Zwe08}.

Thus, the Mordell integrals \eqref{eq:ls} naturally give rise to theta functions whose transformation laws resemble those of the homogeneous solutions \eqref{eq:mock_q2_nonunary_homog}. Since $\varepsilon_p^{(j)}(n)$ is odd, the natural corresponding theta series is
\begin{eqnarray}
    \Theta_p^{(j)}(\tau)= \frac{1}{2} \sum_{n \in \Z} n\varepsilon_p^{(j)}(n) Q^{\frac{n^2}{24p}}= (-1)^j \left( \frac{12}{p} \right) \sum_{\substack{n \in \Z \\ n \equiv x_j \pmod{p}}} n \left( \frac{12}{n} \right)  Q^{\frac{n^2}{24p}}.
\end{eqnarray}
We thus define, for $x \in \Z /p \Z,$ the theta functions 
\begin{eqnarray}\label{eq: Thetapx}
    \Theta_{p,x}(\tau):= \sum_{\substack{n \in \Z \\ n \equiv x \pmod{p}}} n \left( \frac{12}{n} \right)  Q^{\frac{n^2}{24p}}.
\end{eqnarray}
Since $\left( \frac{12}{-n} \right)= \left( \frac{12}{n} \right)$, we observe that $\Theta_{p,x}(\tau)$ is odd in $x$:
\begin{eqnarray}\label{eq: thetap odd}
    \Theta_{p,-x}(\tau)= -\Theta_{p,x}(\tau), \qquad \Theta_{p,0}(\tau) \equiv0.
\end{eqnarray}

\subsubsection{Theta transformation laws}\label{sec: theta trasn laws mordell} We next present the transformation laws of $\Theta_{p,x}(\tau)$ under $SL(2,\Z)$. For the purpose of our constructions, we are concerned only with the transformation laws of the theta functions of the form $\Theta_{p,s x_j}(\tau)$. For $s=1$, the transformation laws follow from \cite[Proposition 2]{hikamiTL} and \eqref{eq: thetap odd} (or simply using the Poisson summation formula): 
\begin{eqnarray}\label{eq: TL thetap s=1}
    \Theta_{p,x_j}\left(-\frac{1}{\tau}\right)= \left( \frac{12}{p} \right) (-i \tau)^{3/2} \sum_{k=1}^{\ell_p} \frac{2}{\sqrt{p}} \sin \left( \frac{6 \pi j k}{p} \right) \Theta_{p,x_k}(\tau).
\end{eqnarray} 

\begin{corollary}\label{cor: thetasxj}
    Let $p\ge 5$ be a prime. Let $x_j$, $\ell_p$, and $\Theta_{p,x}(\tau)$ be as above. Let $s$ be such that $(s,6p)=1.$ Then 
    \begin{eqnarray} \label{eq: TL thetap s}
        \Theta_{p,sx_j}\left(-\frac{1}{\tau}\right)= \left( \frac{12}{p} \right) (-i \tau)^{3/2} \sum_{k=1}^{\ell_p} \frac{2}{\sqrt{p}} \sin \left( \frac{6 \pi s^2 j k}{p} \right) \Theta_{p,sx_k}(\tau),
    \end{eqnarray}
    and 
    \begin{eqnarray} \label{eq: TL thetap T}
        \Theta_{p,sx_j}(\tau+1)=e^{2 \pi i s^2 x_j^2 /24p} \,\Theta_{p,sx_j}(\tau).
    \end{eqnarray}
\end{corollary}
\begin{proof}
    By definition of $x_j$, we have $sx_j \equiv 6sj \pmod{p}.$ Since $(sj,p)=1,$ it follows that $\exists !~ 1 \le j'\le \ell_p = \frac{p-1}{2}$ such that $sj \equiv \pm j' \pmod{p},$ which gives $sx_j \equiv \pm 6 j' \equiv \pm x_{j'} \pmod{p}.$ Thus, using \eqref{eq: thetap odd}, we get $\Theta_{p,sx_j}(\tau)= \pm \Theta_{p,x_{j'}}(\tau).$ Applying \eqref{eq: TL thetap s=1} to $\Theta_{p,x_{j'}}(\tau)$ and writing similarly $\Theta_{p,sx_k}(\tau)= \pm \Theta_{p,x_{k'}}(\tau)$ give \eqref{eq: TL thetap s}.

    For \eqref{eq: TL thetap T}, note that by definition \eqref{eq: Thetapx}, the phase factor under $\tau \to \tau+1$ is $e^{2 \pi i n^2/24p}$. However, $n^2 \equiv s^2  x_j^2 \pmod{p}$, and whenever $\Bigl( \frac{12}{n} \Bigr) \neq 0,$ we have $(n,6)= (s, 6)= (x_j, 6)=1,$ so that $n^2 \equiv 1 \equiv s^2 x_j^2 \pmod{24}.$ Hence if $\Bigl( \frac{12}{n} \Bigr) \neq 0,$ we have $n^2 \equiv  s^2 x_j^2 \pmod{24p}$ by the Chinese remainder theorem, and the claim follows.
\end{proof}

We show that the theta functions arising from the Mordell integrals can be used to construct explicit homogeneous solutions to the system \eqref{eq:mock_q2_nonunary} for every prime $p \ge 5.$ For a homogeneous solution $H_j$ solving \eqref{eq:mock_q2_nonunary_homog}, writing $\tilde H_j(Q)= Q^{- \Delta_{(p,j)}} H_j(Q)$ and $\tilde H(\tau)= \bigl(\tilde H_1(\tau), \cdots \tilde H_{\ell_p}(\tau) \bigr)^T,$ we see that \eqref{eq:mock_q2_nonunary_homog} becomes 
\begin{eqnarray}\label{eq: Htilde S}
    \tilde H(-1/\tau)=- (-i \tau)^{1/2} M^{(p)} \tilde H(\tau)
\end{eqnarray}

The construction depends on the residue class of $p$ modulo $12$, and we thus consider these cases separately. Note that for a prime $p \ge 5$, we have $(p,12)=1,$ so $p \equiv 1,5,7,11 \pmod{12}$.

We first begin with the cases where $p \equiv 1 \pmod{4}$. Recall that $p \equiv 1 \pmod{4}$ iff $p$ is congruent to $1$ or $5$ modulo $12$. For these cases, we use the theta series \eqref{eq: Thetapx} to construct a homogeneous solution.

\subsubsection{Proof of Theorem \ref{thm: homog sol for any p} (1): $p \equiv 1 \pmod{12}$}\label{sec: proof thm homog sol p1}
By Corollary \ref{cor: thetasxj}, we know that $\Theta_{p,x}$ has an $S$-transformation law  that is similar in shape to the homogeneous system \eqref{eq:mock_q2_nonunary_homog}, except that it is of weight $3/2$. Hence, it is natural to consider objects of the type $\Theta_{p,x}(\tau)/\eta(\tau)^2,$ which is of weight $1/2$ just as \eqref{eq:mock_q2_nonunary_homog} is. To determine the parameter $x$, we use the $T$-transformation laws given by Corollary \ref{cor: thetasxj} (for $(s,6p)=1$) and \eqref{eq:eta tlaws}: 
\begin{eqnarray}\label{eq: theta  by eta 2}
    \frac{\Theta_{p,sx_j}(\tau+1)}{\eta(\tau+1)^2}= e^{2 \pi i (s^2 x_j^2-2p) /24p} \frac{\Theta_{p,sx_j}(\tau)}{\eta(\tau)^2}.
\end{eqnarray}
Now since $\Delta_{(p,j)}= \frac{x_j^2}{24p}$, we have
\begin{eqnarray}\label{eq: H tilde T trans}
    \tilde H_j(\tau+1)= e^{-2 \pi i \frac{x_j^2}{24p}} \tilde H_j(\tau),
\end{eqnarray}
and hence we choose $s$ so that $(s,6p)=1$ and $s^2 x_j^2-2p \equiv - x_j^2 \pmod{24p}.$ Equivalently, we want $(s^2 +1) x_j^2 \equiv 2p \pmod{24p}.$ For $p \equiv 1 \pmod{12}$, however, we have $p \nmid x_j$ and $x_j^2 \equiv 1 \pmod{24}$, which immediately give the following lemma.

\begin{lemma}\label{lem: sxj}
    With the above notation, if $p \equiv 1 \pmod{12}$, we have
    \begin{eqnarray}
        (s^2 +1) x_j^2 \equiv 2p \pmod{24p} \iff (s,6p)=1 \text{ and } s^2 \equiv -1 \pmod{p}.
    \end{eqnarray}
    Moreover, such an integer $s \ge 1$ exists.
\end{lemma}
\begin{proof}[Proof of Lemma \ref{lem: sxj}]
    The equivalence follows by the discussion preceeding the lemma and the fact that $(x,6)=1$ iff $x^2 \equiv 1 \pmod{24}$. For existence, note that since $p \equiv 1 \pmod{4}$, $-1$ is a quadratic residue mod $p$ i.e. there exists $s_0$ such that $s_0^2 \equiv -1 \pmod{p}$. Hence, since $p \ge 5,$ the Chinese remainder theorem implies that there exists $s$ such that $s \equiv s_0 \pmod{p}$ and $s \equiv 1 \pmod{6},$ which completes the proof.
\end{proof}
Hence we get Theorem \ref{thm: homog sol for any p} (1).
\begin{proof}[Proof of Theorem \ref{thm: homog sol for any p} (1)]
    We note that the terms $\Bigl( \frac{12}{p} \Bigr)$ and $(-1)^{\lfloor \frac{p+3}{6} \rfloor}$ appearing in \eqref{eq: TL thetap s} and \eqref{eq:M}, respectively, are both equal to $1$ when $p \equiv 1 \pmod{12}$. Moreover, if $s^2 \equiv -1 \pmod{p}$, then the $\sin$-term appearing in \eqref{eq: TL thetap s} becomes $\sin \bigl( \frac{6 \pi s^2 j k}{p} \bigr)= -\sin \bigl( \frac{6 \pi j k}{p} \bigr).$ The claim about satisfying the homogeneous system then follows immediately from Corollary \ref{cor: thetasxj} and \eqref{eq:eta tlaws}.

    For holomorphy, note that $H_j$ is a holomorphic function of $\tau \in \mathbb H$ that is invariant under $T.$ Moreover, by Lemma \ref{lem: sxj}, we know that such an integer $s$ exists and can be constructed as in the proof of the lemma. Moreover, whenever $n \equiv s x_j \pmod{p}$ and $\Bigl( \frac{12}{n} \Bigr) \neq0,$ we have $x_j^2+n^2-2p \equiv 0 \pmod{24p},$ so that $\frac{x_j^2+n^2-2p}{24p} \in \Z_{\ge 0}$. This shows that $H_j(Q)$ is a holomorphic $Q$-series for $|Q|<1$.
\end{proof}

\subsubsection{Proof of Theorem \ref{thm: homog sol for any p} (2): $p \equiv 5 \pmod{12}$}\label{sec: proof thm homog sol p2}
In this case, the equation $s^2 \equiv -1 \pmod{p}$ still has a solution with $(s,6)=1$ and the corresponding equation \eqref{eq: homog p1mod12} still satisfies the $S$-transformation law \eqref{eq:mock_q2_nonunary_homog}. However, for $p \equiv 5 \pmod{12}$, \eqref{eq: homog p1mod12} is no longer invariant under $\tau \to  \tau+1$, and in particular, it is no longer a holomorphic $Q$-series for $|Q|<1.$ To understand the required correction to the formula \eqref{eq: homog p1mod12}, we note that, in view of the discussion above, the ratio of $T$-multipliers of \eqref{eq: theta  by eta 2} and \eqref{eq: H tilde T trans} is now given by 
\begin{eqnarray}
    \frac{e^{2 \pi i (s^2 x_j^2-2p) /24p}}{e^{-2 \pi i x_j^2 / 24p}}= e^{2 \pi i \bigl( (s^2+1) x_j^2-2p \bigr) /24p}= e^{2 \pi i /3},
\end{eqnarray}
where the last equality is because $(s^2+1)x_j^2 \equiv 10 p \pmod{24p}$ for $p \equiv 5 \pmod{12}$. This is true since both sides agree modulo $p$ and $24$ separately. 

We thus correct \eqref{eq: homog p1mod12} by multiplying it by a holomorphic function of $\tau \in \mathbb H$ that is invariant under $S$ and contributes the phase $e^{-2 \pi i /3}$ under $T$. For this purpose, we use $E_4(\tau)/\eta(\tau)^8$, where $E_4(\tau)$ is the Eisenstein series of weight 4. Indeed, 
\begin{eqnarray}\label{eq: E4 eta8}
    \frac{E_4(-1/\tau)}{\eta(-1/\tau)^8}=\frac{E_4(\tau)}{\eta(\tau)^8}, \qquad \frac{E_4(\tau+1)}{\eta(\tau+1)^8}= e^{-2 \pi i /3}\frac{E_4(\tau)}{\eta(\tau)^8}.
\end{eqnarray}
Hence we get Theorem \ref{thm: homog sol for any p} (2).
\begin{proof}[Proof of Theorem \ref{thm: homog sol for any p} (2)]
    We note that $\Bigl( \frac{12}{p} \Bigr) = (-1)^{\lfloor \frac{p+3}{6} \rfloor}=-1$ when $p \equiv 5 \pmod{12}$. The claim about satisfying the homogeneous system then follows immediately from Corollary \ref{cor: thetasxj}, \eqref{eq: E4 eta8}, and \eqref{eq:eta tlaws}. For holomorphy, note that $H_j$ is a holomorphic function of $\tau \in \mathbb H$ that is invariant under $T.$ Moreover, if $p \equiv 5 \pmod{12}$, $n \equiv s x_j \pmod{p}$, and $\Bigl( \frac{12}{n} \Bigr) \neq0,$ then as noted above, we have $(s^2+1)x_j^2 \equiv 10 p \pmod{24p},$ so that $\frac{x_j^2+n^2-10p}{24p} \in \Z_{\ge 0}$. This shows that $H_j(Q)$ is a holomorphic $Q$-series for $|Q|<1$.
\end{proof}

We now turn to the cases $p \equiv 3 \pmod{4}$, equivalently $p \equiv 7,11 \pmod{12}$. In the previous section, the construction of homogeneous solutions relied on the fact that $-1$ is a quadratic residue modulo primes $p \equiv 1 \pmod{4}$, which is no longer true for primes $p \equiv 3 \pmod{4}$. To remedy this, we seek instead solutions to the equation $r s^2 \equiv -1 \pmod {p}$ for some unit $r$ modulo $p.$ For example, if $p \nmid r$ and $-r$ is quadratic residue modulo $p,$ i.e. $\Bigl( \frac{-r}{p} \Bigr)=1,$ then there exists a unit $u$ such that $u^2 \equiv -r \pmod{ p},$ so that $s= r^{-1} u \pmod{ p}$ works. Hence, evaluating the theta series at $r \tau$, instead of $\tau,$ would fix the $T$-transformation law.

\subsubsection{Proof of Theorem \ref{thm: homog sol for any p} (3): $p \equiv 7 \pmod{12}$}\label{sec: proof thm homog sol p3}
Working simply with $\Theta_{p,x}(r\tau)$, given by \eqref{eq: Thetapx}, does not yield homogeneous solutions for this case. Thus, since $\Theta_{p,x}(\tau)$ is associated with the character $\Bigl( \frac{12}{n} \Bigr) $, the next natural candidates with similar transformation laws are the weight-$1/2$ theta series associated with the odd primitive factors of $\Bigl( \frac{12}{n} \Bigr)= \Bigl( \frac{-3}{n} \Bigr) \Bigl( \frac{-4}{n} \Bigr).$ That is, we consider 
\begin{eqnarray}
    \sum_{\substack{n \in \Z \\ n \equiv x \pmod{p}}} \left( \frac{-3}{n} \right)  Q^{\frac{rn^2}{24p}}, \qquad \sum_{\substack{n \in \Z \\ n \equiv x \pmod{p}}} \left( \frac{-4}{n} \right)  Q^{\frac{rn^2}{24p}}.
\end{eqnarray}
Under the $S$-transformation law and Poisson summation, the Gaussian $Q^{\frac{rn^2}{24p}}$ maps to $Q^{ \frac{16}{r} \frac{n^2}{24p}}$ in the former series and to $Q^{ \frac{9}{r} \frac{n^2}{24p}}$ in the latter series. Therefore, to obtain a symmetric $S$-transformation law similar to the homogeneous system \eqref{eq:mock_q2_nonunary_homog}, we require $r=4$ for the former and $r=3$ for the latter. On the other hand, matching the multipliers in the $T$-transformation law requires a solution of the equation $r s^2 \equiv -1 \pmod {p}$, which is impossible for $r=4$ when $p \equiv 3 \pmod{4}$, so we are led to the second series 
\begin{eqnarray}
    \theta_{p,x}(\tau):= \sum_{\substack{n \in \Z \\ n \equiv x \pmod{p}}} \left( \frac{-4}{n} \right)  Q^{\frac{n^2}{8p}}.
\end{eqnarray}
For $\theta_{p,x}(\tau),$ the equation $3 s^2 \equiv -1 \pmod {p}$ is solvable when $p \equiv 7 \pmod{12}$ (see Lemma \ref{lem s p7(12)} below), and since $\Bigl( \frac{-4}{n} \Bigr)$ is odd in $n,$ we have 
\begin{eqnarray}
    \theta_{p,-x}(\tau)= -\theta_{p,x}(\tau), \qquad \theta_{p,0}(\tau) \equiv0.
\end{eqnarray}
Thus, analogous to Corollary \ref{cor: thetasxj}, we obtain the following transformation laws. 
\begin{lemma}\label{lem: Phipx trans}
    Let $p\ge 5$ be a prime. Let $x_j$, $\ell_p$, and $\theta_{p,x}(\tau)$ be as above. Let $s$ be such that $(s,2p)=1.$ Then 
    \begin{eqnarray} \label{eq: TL Phip s}
        \theta_{p,sx_j}\left(-\frac{1}{\tau}\right)= -\left( \frac{-4}{p} \right) (-i \tau)^{1/2} \sum_{k=1}^{\ell_p} \frac{2}{\sqrt{p}} \sin \left( \frac{18 \pi s^2 j k}{p} \right) \theta_{p,sx_k}(\tau),
    \end{eqnarray}
    and 
    \begin{eqnarray} \label{eq: TL Phip T}
        \theta_{p,sx_j}(\tau+1)=e^{2 \pi i s^2 x_j^2 /(8p)} \theta_{p,sx_j}(\tau).
    \end{eqnarray}
\end{lemma}
\begin{proof}
    Similar to the proof of Corollary \ref{cor: thetasxj}, the $S$-transformation law follows from the Poisson summation formula, and the $T$-transformation law follows by a direct substitution. 
\end{proof}

We first make some observations. 

\begin{lemma}\label{lem s p7(12)}
    If $p \equiv 7 \pmod{12}$, then there exists an integer $s= 3^{-1} \sqrt{-3} \pmod{p}$ such that $(s,2p)=1$ and $3 s^2 \equiv -1 \pmod {p}$.
\end{lemma}
\begin{proof}
    Indeed, $3 s^2 \equiv -1 \pmod {p}$ is solvable with $s= 3^{-1} \sqrt{-3} \pmod{p}$ when $p \equiv 7 \pmod{12}$, since $\Bigl( \frac{-3}{p} \Bigr)= \Bigl( \frac{p}{3} \Bigr)=1$ when $p \equiv 1 \pmod{3}$. Then clearly $(s,p)=1$, and if $s$ is even, we replace $s$ with $s+p.$
\end{proof}

We next observe that if $p \equiv 7 \pmod{12}$, $(s,2p)=1$ and $3 s^2 \equiv -1 \pmod {p}$, then $\Bigl( \frac{-4}{p} \Bigr)=-1$ and \eqref{eq: TL Phip s} becomes closer in shape to \eqref{eq: Htilde S}:
\begin{eqnarray}\label{eq: phi close S}
    (-1)^{j+1} \theta_{p,sx_j}\left(-\frac{1}{\tau}\right)= (-i \tau)^{1/2} \sum_{k=1}^{\ell_p} M^{(p)}_{jk} (-1)^{k+1} \theta_{p,sx_k}(\tau).
\end{eqnarray}
Note that the factor $(-1)^{\lfloor \frac{p+3}{6} \rfloor}$ in \eqref{eq:M} equals $-1$, since $\lfloor \frac{p+3}{6} \rfloor$ is odd when $p \equiv 7 \pmod{12}$.

On the other hand, comparing the multipliers in the $T$-transformation laws \eqref{eq: H tilde T trans} and \eqref{eq: TL Phip T}, we see that we need $(3s^2 +1)x_j^2 \equiv 0 \pmod{24p}.$ However, for $p \equiv 7 \pmod{12}$, $(s,2p)=1$ and $3 s^2 \equiv -1 \pmod {p}$, it is easy to check, as before, that $(3s^2 +1)x_j^2 \equiv 4p \pmod{24p}.$ Hence, from \eqref{eq: TL Phip T}, we get a $T$-transformation law closer in shape to \eqref{eq: H tilde T trans}:
\begin{eqnarray}\label{eq: phi close T}
    (-1)^{j+1} \theta_{p,sx_j}(\tau +1)= e^{\pi i /3} e^{-2 \pi i x_j^2/(24p)}  (-1)^{j+1} \theta_{p,sx_j}(\tau).
\end{eqnarray}
We thus use elementary operations and normalizations to modify \eqref{eq: phi close S} and \eqref{eq: phi close T} to match \eqref{eq: Htilde S} and \eqref{eq: H tilde T trans}, respectively, and obtain Theorem \ref{thm: homog sol for any p} (3).
\begin{proof}[Proof of Theorem \ref{thm: homog sol for any p} (3)]
    The transformation law \eqref{eq:mock_q2_nonunary_homog} and invariance under $T$ follow from the discussion above. For holomorphy, note that $H_j$ is a holomorphic function of $\tau \in \mathbb H$ that is invariant under $T.$ Moreover, using the standard identity $\mathcal D \eta(\tau)= \frac{E_2(\tau)}{24} \eta(\tau)$, where $E_2$ is the weight-$2$ Eisenstein series, we see that exponents of $Q$ in $H_j(Q)$ are of the form $\frac{x_j^2 +3 n^2-4p}{24p} + m$ with $m \in \Z_{\ge 0}$, and one can easily check that $\frac{x_j^2 +3 n^2-4p}{24p} \in \Z_{\ge 0}$ whenever $n \equiv sx_j \pmod{p},$ $\Bigl( \frac{-4}{n} \Bigr) \neq 0,$ and $s$ and $p$ are as in the statement. This shows $H_j(Q)$ is a holomorphic $Q$-series for $|Q|<1$.
\end{proof}

\subsubsection{Proof of Theorem \ref{thm: homog sol for any p} (4)}\label{sec: proof thm homog sol p4}
Lastly, we turn to the construction of homogeneous solutions for the only remaining case $p \equiv 11 \pmod{12}$. For such primes, the equations $s^2 \equiv -1 \pmod {p}$ and $3 s^2 \equiv -1 \pmod {p}$ that we encountered in the previous constructions are not solvable. Thus, the series $\Theta_{p,sx_j}(\tau)$ and $\theta_{p,sx_j}(\tau)$ cannot individually yield the correct transformation laws. 

We observe, however, that for $(r,6p)=1$ and $(s,2p)=1,$ Corollary \ref{cor: thetasxj} and Lemma \ref{lem: Phipx trans} give
\begin{eqnarray}\label{eq: Thetapx Phipx}
    \Theta_{p,rx_j}(\tau+1) \theta_{p,sx_j}(\tau+1)= e^{2 \pi i (r^2+3s^2) x_j^2 /(24p)} \Theta_{p,rx_j}(\tau) \theta_{p,sx_j}(\tau),
\end{eqnarray}
and for this to match \eqref{eq: H tilde T trans} modulo $p$, we require a solution to $r^2+3s^2 \equiv -1 \pmod{p}$, which is solvable. 

\begin{lemma}
    Let $p \ge 5$ be a prime with $p \equiv 11 \pmod{12}$. Then there exists $r,s \in \Z$ such that $(r,6p)=1$, $(s,2p)=1$, and $r^2+3s^2 \equiv -1 \pmod{p}$.
\end{lemma}
\begin{proof}
    We first show that $r^2+3s^2 \equiv -1 \pmod{p}$ has a solution. Indeed, since there are $\frac{p+1}{2}$ distinct $r^2 \pmod{p}$ and $\frac{p+1}{2}$ distinct $-1-3s^2 \pmod{p}$, the pigeonhole principle guarantees the existence of a solution. Moreover, one can construct a solution by looping over $s \pmod{p}$ until $-1-3s^2$ is a square modulo $p$ (note that a random $s \pmod{p}$ works with probability $\approx 1/2$).

    However, as discussed above, for $p \equiv 11 \pmod{12}$, a solution with $r \equiv 0 \pmod{p}$ or $s\equiv 0 \pmod{p}$ is impossible. Also, by the Chinese remainder theorem, we may lift the solution pair $r$ and $s$ so that $r \equiv 1 \pmod{6}$ and $s \equiv 1 \pmod{2}$, so in particular, we also have $(r,6p)=1$ and $(s,2p)=1$.
\end{proof}

On the other hand, since $- \Bigl( \frac{12}{p} \Bigr) \Bigl( \frac{-4}{p} \Bigr)=1$ for $p \equiv 11 \pmod{12}$, Corollary \ref{cor: thetasxj} and Lemma \ref{lem: Phipx trans} also give 
\begin{eqnarray}\label{eq: Thetapx Phipx S}
    \Theta_{p,rx_j}(-1/\tau) \theta_{p,sx_j}(-1/\tau)=  (-i \tau)^{2} \sum_{k,m=1}^{\ell_p} \frac{4}{p} \sin \left( \frac{6 \pi r^2 j k}{p} \right) \sin \left( \frac{18 \pi s^2 j m}{p} \right) \Theta_{p,rx_k}(\tau) \theta_{p,sx_m}(\tau). \nonumber \\
\end{eqnarray}
To make the $S$-transformation law symmetric, we sum over all pairs $(r,s) \pmod{p}$ such that $r^2+3s^2 \equiv -1 \pmod{p}$. We carry out the construction for $p=11,$ and the construction for $p \equiv 11 \pmod{12}$ is symmetric. 

For $p=11,$ the equation $r^2+3s^2 \equiv -1 \pmod{11}$ has 12 solutions modulo 11:
\begin{eqnarray}\label{eq: desired r,s}
    (r,s)= \qquad (\pm 1, \pm 5), \qquad (\pm 3, \pm2), \qquad (\pm 4, \pm 3).
\end{eqnarray}
Since $\Theta_{p,x}$ and $\theta_{p,x}$ are odd in $x$, we thus choose constants $c_1,c_2,c_3$ so that the $S$-transformation law of
\begin{eqnarray}\label{eq: p11 desired}
   c_1  \Theta_{11,x_j}(\tau) \theta_{11,5x_j}(\tau) +c_2  \Theta_{11,3x_j}(\tau) \theta_{11,2x_j}(\tau) +c_3  \Theta_{11,4x_j}(\tau) \theta_{11,3x_j}(\tau)
\end{eqnarray}
has the desired symmetric form. Now recall that $\{\pm x_j\}_{j=1}^{5} = (\Z / 11\Z)^\times$, and hence for every $r \in (\Z / 11\Z)^\times$ and $1\le j \le 5,$ there exists $1 \le k\le 5$ such that $r x_j = \pm x_k$. 

Therefore, expanding \eqref{eq: p11 desired} using \eqref{eq: Thetapx Phipx S} and requiring the coefficients of all products $\Theta_{11,x_k}(\tau) \theta_{11,x_m}(\tau)$ that are not of the form $\Theta_{11,\pm x_j} \theta_{11, \pm5x_j},  \Theta_{11,\pm3x_j} \theta_{11,\pm 2x_j},$ or $  \Theta_{11,\pm4x_j} \theta_{11,\pm3x_j}$, to vanish, we obtain a homogeneous linear system for $(c_1,c_2,c_3)^{T}.$ A direct calculation then shows that the solutions space is spanned by $(c_1,c_2,c_3)^{T} = (1,-1,1).$ Thus, \eqref{eq: p11 desired} becomes 
\begin{eqnarray}\label{eq: p11 desired2}
    \Xi_{11,j}(\tau):= \Theta_{11,x_j}(\tau) \theta_{11,5x_j}(\tau) -  \Theta_{11,3x_j}(\tau) \theta_{11,2x_j}(\tau) +  \Theta_{11,4x_j}(\tau) \theta_{11,3x_j}(\tau).
\end{eqnarray}
Moreover, using \eqref{eq: Thetapx Phipx S} and the choice of $(c_1,c_2,c_3)^{T}$, we obtain, for $\Xi_{11}(\tau):=\bigl(\Xi_{11,1}(\tau), \cdots \Xi_{11,5}(\tau) \bigr)^T$, the transformation law:
\begin{eqnarray}\label{eq: Psi A S}
    \Xi_{11}(-1/\tau)= (-i \tau)^{2} A \Xi_{11}(\tau),
\end{eqnarray}
where 
\begin{eqnarray}\label{eq: A psi}
    A_{jk}= \left( \frac{x_j}{11} \right) \left( \frac{x_k}{11} \right) \frac{2}{\sqrt{11}} \sin \left(\frac{6 \pi j k}{11} \right).
\end{eqnarray}
Note the similarity between the matrix $A$ and $M^{(11)}$ in \eqref{eq:M}. The matrix $A$ can be simplified to the form \eqref{eq: A psi} using the standard identity 
\begin{eqnarray}
    \sum_{k=1}^{(p-1)/2} \left( \frac{k}{p} \right) \sin \left(\frac{2 \pi m k}{p} \right)= \frac{\sqrt{p}}{2} \left( \frac{m}{p} \right),
\end{eqnarray}
which holds for every prime $p \equiv 3 \pmod{4}$ and every $m \ge 1$ such that $p \nmid m.$ This identity follows from the separability of Gauss sums: 
\begin{eqnarray}
     \sum_{k=0}^{p} \left( \frac{k}{p} \right) e^{2 \pi i mk/p} = i \sqrt{p} \left( \frac{m}{p} \right), \qquad p \equiv 3 \pmod{ 4}.
\end{eqnarray}
Therefore, using \eqref{eq: Psi A S}, \eqref{eq: A psi}, and $M^{(11)}$ in \eqref{eq:M}, we get 
\begin{eqnarray}
    (-1)^{j+1} \left( \frac{x_j}{11} \right) \Xi_{11,j}(-1/\tau)= (-i\tau)^2 \sum_{k=1}^5 M_{jk}^{(11)} (-1)^{k+1} \left( \frac{x_k} {11} \right) \Xi_{11,k}(\tau).
\end{eqnarray}
On the other hand, since $\Theta_{11,x}$ and $\theta_{11,x}$ depend only on $x \pmod p,$ we can lift each pair $(r,s)= (1,5), (3,2), (4,3) \pmod{11}$ in \eqref{eq: p11 desired2}, by the Chinese remainder theorem, so that we also have $r \equiv 1 \pmod{6}$ and $s \equiv 1 \pmod{2}$, so that $(r,66)=(s,22)=1.$ For example, the corresponding pairs become $(r,s)= (1,5), (25,13), (37,3)$, for each of which we have $(r^2+3s^2+1)x_j^2 \equiv 0 \pmod{11}$  and $(r^2+3s^2+1)x_j^2 \equiv 5 x_j^2 \equiv 5 \pmod{24}$, so that the Chinese remainder theorem gives $(r^2+3s^2+1)x_j^2 \equiv 77 \pmod{(11) (24)}$. Using \eqref{eq: Thetapx Phipx} and \eqref{eq: p11 desired2}, this in turn implies that 
\begin{eqnarray}
    (-1)^{j+1} \Bigl( \frac{x_j}{11} \Bigr) \Xi_{11,j}(\tau+1) = e^{2 \pi i 7/24} e^{-2 \pi i x_j^2/264}   (-1)^{j+1} \Bigl( \frac{x_j}{11} \Bigr) \Xi_{11,j}(\tau).
\end{eqnarray}
As before, elementary operations and normalizations now give Theorem \ref{thm: homog sol for any p} (4).
\begin{proof}[Proof of Theorem \ref{thm: homog sol for any p} (4)]
    The transformation law follows from the discussion above, and holomorphy follows as in previous proofs.
\end{proof}
The construction generalizes straightforwardly  to all primes $p \equiv 11 \pmod{12}$.

\subsection{Order 5 Uniqueness}
In this section, we prove the uniqueness of the order 5 mock theta identities satisfying \eqref{eq:transf-chi}. Theorem \ref{thm: general uniq thm} for the $p=5$ case then follows as a corollary of it. We then give the general holomorphic solution to \eqref{eq:mf5-identity}, giving rise to new excited mock theta functions of order $5$.

\subsubsection{The transformation laws for order 5 mock thetas} \label{sec: TL mf5}
Consider two of Ramanujan's original order 5 mock theta functions, denoted $\chi_0$ and $\chi_1$, \cite{Watson,GM12}, defined for $|q|<1$ as
\begin{eqnarray}
    \chi_0(q):= \sum_{n=0}^\infty \frac{q^n}{(q^{n+1}; q)_n}
    \qquad, \qquad
    \chi_1(q):= \sum_{n=0}^\infty \frac{q^n}{(q^{n+1}; q)_{n+1}}
\end{eqnarray}
The order-5 mock theta functions $\chi_0,\chi_1$ satisfy the following mock modular transformation laws for $\Re (t)>0$  \cite[p.\ 118]{GM12} 
\begin{subequations}\label{eq:transf-chi}
\begin{align}
q^{-1/120}\bigl(\chi_0(q)-2\bigr)
= &-\sqrt{\frac{\pi(5-\sqrt5)}{5t}}\,q_1^{-1/30}\bigl(\chi_0(q_1^{4})-2\bigr)
   -\sqrt{\frac{\pi(5+\sqrt5)}{5t}}\,q_1^{71/30}\chi_1(q_1^{4}) \\ \nonumber
   &-\sqrt{\frac{135t}{2\pi}}\,L\!\left(\frac15,5t\right),
\end{align}
\begin{align}
q^{71/120}\chi_1(q)
= &-\sqrt{\frac{\pi(5+\sqrt5)}{5t}}\,q_1^{-1/30}\bigl(\chi_0(q_1^{4})-2\bigr)
   +\sqrt{\frac{\pi(5-\sqrt5)}{5t}}\,q_1^{71/30}\chi_1(q_1^{4}) \\ \nonumber
   &-\sqrt{\frac{135t}{2\pi}}\,L\!\left(\frac25,5t\right).
\end{align}
\end{subequations}
Here the Mordell-Appell integrals are expressed in terms of (\cite{GM12}, page 118):
\begin{align}
L(r,t)
&=\int_{0}^{\infty} e^{-\frac{3}{2}t x^{2}}\, 
  \frac{\cosh \bigl((3r-2)t x\bigr)+\cosh \bigl((3r-1)t x\bigr)}
       {\cosh(\frac{3}{2}t x)}\,dx.
\label{eq:L def}
\end{align}

The transformation laws in \eqref{eq:transf-chi} are more naturally 
expressed in terms of $Q$ and $Q_1$, where we recall the notation defined in \S \ref{sec:notation}:
\begin{eqnarray}
    Q:=q^2=e^{-2t} \qquad; \qquad \Qt:=q_1^2 = e^{-2\pi^2/t}
    \label{eq:q2}
\end{eqnarray}
Evaluating \eqref{eq:transf-chi} at $2t$, we define the $\ell_5$-component vector of Mordell integrals
\begin{eqnarray}
    \mathcal L^{(5)} (t)=\sqrt{\frac{60\,t}{\pi}}\, \begin{pmatrix} L\left(1/5, 10t\right)
    \cr  {L\left(2/5, 10t\right)}
    \end{pmatrix}
    \label{eq:lint5}
\end{eqnarray}
This vector transforms as follows under the modular $S$-transformation:
\begin{eqnarray}
    \mathcal L^{(5)} (t)=\sqrt{\frac{\pi}{t}}\, M^{(5)}\, \mathcal L^{(5)}\left(\frac{\pi^2}{t}\right)
    \label{eq:lint}
\end{eqnarray}
where the matrix $M^{(5)}$ in \eqref{eq:mf5-identity} is the $2\times 2$ mixing matrix from \eqref{eq:M}
\begin{eqnarray}
    M^{(5)}=\frac{2}{\sqrt{5}} \begin{pmatrix}
    \sin\left(\frac{\pi}{5}\right) & \sin\left(\frac{2\pi}{5}\right)
    \cr \cr
    \sin\left(\frac{2\pi}{5}\right) & -\sin\left(\frac{\pi}{5}\right)
    \end{pmatrix}
    \label{eq:m}
\end{eqnarray}
Then the modular transformation laws in \eqref{eq:transf-chi} can be written in a compact matrix form:
\begin{eqnarray} 
\label{eq:mf5-identity}
    \mathcal L^{(5)}(t)= 
    \begin{pmatrix}
    Q^{-1/120} \, X_1^{(5)}(Q) \cr 
    Q^{-49/120} \, X_2^{(5)}(Q)
    \end{pmatrix}
    +\sqrt{\frac{\pi}{t}} \, M^{(5)}\, 
    \begin{pmatrix}
    \Qt^{-1/120} \, X_1^{(5)}(\Qt) \cr 
    \Qt^{-49/120} \, X_2^{(5)}(\Qt)  
    \end{pmatrix}
    \quad, \quad \Re t>0
\end{eqnarray}
which is precisely the general expression \eqref{eq:mock_q2_nonunary} at $p=5$, with the identifications:
\begin{eqnarray}
     X_1^{(5)}(Q)=-\frac{2}{3} \bigl(\chi_0(Q)-2\bigr) \qquad; \qquad X_2^{(5)}(Q)=- \frac{2}{3} Q\, \chi_1(Q)
    \label{eq:mf$_5$-solution}
\end{eqnarray}

\subsubsection{Proof of Theorem \ref{thm: general uniq thm} for $p=5$} \label{sec:H23 proof}
We first note that \eqref{eq:mf$_5$-solution} proves existence. Moreover, assuming \eqref{eq:mf$_5$-solution}, holomorphic solutions $X^{(5)}$ of \eqref{eq:mf5-identity} with $X_2^{(5)}(0)=0$ correspond exactly to holomorphic solutions $(\chi_0, \chi_1)$ satisfying \eqref{eq:transf-chi}. Thus, for uniqueness, it suffices to prove that the pair of order $5$ mock theta functions $(\chi_0,\chi_1)$ is the unique pair of holomorphic $q$-series in $\mathbb D$ satisfying the mock modular transformation identities \eqref{eq:transf-chi} for $\Re (t) > 0$.

\begin{proof}
    Suppose $(\chi_0,\chi_1)$ and $(\tilde \chi_0, \tilde \chi_1)$ are two such solutions, and let
    \begin{eqnarray}
        H_1(Q)=-\frac{2}{3} \bigl(\tilde \chi_0(Q)-\chi_0(Q) \bigr), \qquad H_2(Q)=-\frac{2}{3} Q \bigl(\tilde \chi_1(Q)-\chi_1(Q) \bigr).
    \end{eqnarray}
    Let $v(\tau)$ and $W(\tau)$ be the corresponding functions given by \eqref{eq:v def} and \eqref{eq:Wronskian}, respectively. In view of \S \ref{sec: Wronskian}, it suffices to prove the corresponding Wronskian \eqref{eq:Wronskian} is identically zero. Indeed, when $p=5,$ the exponents are $\Delta_{(5,1)}=\frac{1}{120}$ and $\Delta_{(5,2)}=\frac{49}{120}$, and the matrix by \eqref{eq:D} and \eqref{eq: tilde Delta} is thus $D_5= \text{diag} \bigl(e^{-\pi i /10},e^{-9\pi i /10} \bigr).$ It thus follows from \eqref{eq:m}, \eqref{eq:W T}, \eqref{eq:W S}, and \eqref{eq:eta tlaws} that
    \begin{equation}\label{eq:G def}
        G(\tau):=\frac{W(\tau)^3}{\eta(\tau)^{12}}
    \end{equation}
    is invariant under both $S:\tau\mapsto -1/\tau$ and $T:\tau\mapsto\tau+1$. Moreover, $G$ is holomorphic on $\mathbb H,$ and one can easily show,  using \eqref{eq:mf$_5$-solution}, \eqref{eq:Wronskian}, \eqref{eq:eta def}, and \eqref{eq:G def}, that $G(\tau)=Q h(Q)$ with $h$ holomorphic at $Q=0.$ It follows that $G$ extends holomorphically to the cusp $\infty$ with $G(\infty)=0$. However, the only holomorphic modular functions for $SL(2,\Z)$ are constants, and since G is such a function with $G(\infty)=0,$ it follows that $G \equiv 0$ on $\mathbb H$. This in turn implies that $W(\tau)\equiv 0$ on $\mathbb H$. Applying Lemma \ref{lem: wronskian} completes the proof.
\end{proof}

\subsubsection{Proof of Proposition \ref{prop general sol X_5}}\label{sec: proof prop gen sol X5}
\begin{proof}
    Let $\mathcal V_5$ be the $\C$-vector space  of homogeneous holomorphic solutions $H^{(5)}$ to \eqref{eq:mf5-identity}. Then by Theorem \ref{thm: general uniq thm} for the $p=5$ case, the map $H \mapsto H_2(0)$ is an injective linear map from $\mathcal V_5$ to $\C.$ Thus, $ \text{dim}_{\C} \mathcal V_5 \le 1$, and since $H^{(5)} \in \mathcal V_5$ and $H_2^{(5)}(0) \neq 0$ (see \eqref{eq: H^(5)}), where $H^{(5)}$ is defined in Theorem \ref{thm: homog sol for any p} (2), the claim follows.
\end{proof}

\subsection{Order 7 Uniqueness}
\label{sec:mf7}

In this section, we prove the ground-state uniqueness for order 7 mock thetas, i.e. Theorem \ref{thm: general uniq thm}. We also give the explicit general solution for the same system.

\subsubsection{The transformation laws for order 7 mock thetas} \label{sec: TL mf7}
Consider three of Ramanujan's original order $7$ mock theta functions, denoted $\mathcal{F}_0,$ $\mathcal{F}_1$ and $\mathcal{F}_2$, \cite[p. 100, eq (2.3)]{GM12}
\begin{eqnarray}
    \mathcal{F}_0(q)= \sum_{n=0}^\infty \frac{q^{n^2}}{(q^{n+1};q)_n}, \quad \mathcal{F}_1(q)= \sum_{n=0}^\infty \frac{q^{(n+1)^2}}{(q^{n+1};q)_{n+1}}, \quad \mathcal{F}_2(q)= \sum_{n=0}^\infty \frac{q^{n(n+1)}}{(q^{n+1};q)_{n+1}}.
\end{eqnarray}
The vector satisfies the following mock modular transformation laws for ${\rm Re}(t)>0$  \cite[p. 119]{GM12}  
\begin{subequations} \label{eq: mf7 TL}
\begin{align} \label{eq: mf7 F0}
    q^{-\frac{1}{168}}  \left( \mathcal{F}_0(q) -2 \right)= &\sqrt{\frac{8 \pi }{7 t}} \sin \left(\frac{\pi}{7} \right) q_1^{-\frac{1}{42}} \mathcal{F}_0(q_1^4) + \sqrt{\frac{8 \pi }{7 t}} \sin \left(\frac{2 \pi}{7} \right) q_1^{-\frac{25}{42}} \mathcal{F}_1(q_1^4) \nonumber \\
    &+ \sqrt{\frac{8 \pi }{7 t}} \sin \left(\frac{3 \pi}{7} \right) q_1^{\frac{47}{42}} \mathcal{F}_2(q_1^4) - \sqrt{\frac{42 t }{\pi}} L \left( \frac{1}{7}, 7 t \right),
\end{align}
\begin{align} \label{eq: mf7 F1}
    q^{-\frac{25}{168}} \mathcal{F}_1(q) = &\sqrt{\frac{8 \pi }{7 t}} \sin \left(\frac{2\pi}{7} \right) q_1^{-\frac{1}{42}} \mathcal{F}_0(q_1^4) - \sqrt{\frac{8 \pi }{7 t}} \sin \left(\frac{3 \pi}{7} \right) q_1^{-\frac{25}{42}} \mathcal{F}_1(q_1^4) \nonumber \\
    &+ \sqrt{\frac{8 \pi }{7 t}} \sin \left(\frac{ \pi}{7} \right) q_1^{\frac{47}{42}} \mathcal{F}_2(q_1^4) - \sqrt{\frac{42 t }{\pi}} L \left( \frac{2}{7}, 7 t \right),
\end{align}
\begin{align} \label{eq: mf7 F2}
    q^{\frac{47}{168}} \mathcal{F}_2(q) = &\sqrt{\frac{8 \pi }{7 t}} \sin \left(\frac{3\pi}{7} \right) q_1^{-\frac{1}{42}} \mathcal{F}_0(q_1^4) + \sqrt{\frac{8 \pi }{7 t}} \sin \left(\frac{ \pi}{7} \right) q_1^{-\frac{25}{42}} \mathcal{F}_1(q_1^4) \nonumber \\
    &- \sqrt{\frac{8 \pi }{7 t}} \sin \left(\frac{ 2\pi}{7} \right) q_1^{\frac{47}{42}} \mathcal{F}_2(q_1^4) - \sqrt{\frac{42 t }{\pi}} L \left( \frac{3}{7}, 7 t \right),
\end{align}
\end{subequations}
where the Mordell-Appell integral $L(r,t)$ is defined in \eqref{eq:L def}.
\begin{remark}
    Notice that the first expression \eqref{eq: mf7 F0} is asymmetric in the sense that the LHS involves $(\mathcal F_0-2)$, while the RHS involves just $\mathcal F_0$, as do the RHSs of the other two transformation laws in \eqref{eq: mf7 F1}-\eqref{eq: mf7 F2}. This asymmetry is easily removed by transferring the $(-2)$ term on the LHS to the RHS by combining it with the $- \sqrt{\frac{42 t }{\pi}} L \left( \frac{1}{7}, 7 t \right)$ integral.  This subtraction  is automatically incorporated in the expressions for $\mathcal L^{(p)}$ in \eqref{eq:l}, producing the symmetric transformation laws in \eqref{eq:mf7-identity}.
\end{remark}

We begin by rewriting the system \eqref{eq: mf7 TL} in terms of $Q$ and $Q_1$. Noting that $\ell_7=3$, we define the $3$-component vector of Mordell integrals 
\begin{eqnarray}
    \mathcal L^{(7)} (t)= \sqrt{\frac{84 t }{\pi}}  \, 
    \begin{pmatrix} 
    L \left( \frac{1}{7}, 14 t \right) \cr  
    L \left( \frac{2}{7}, 14 t \right) \cr 
    L \left( \frac{3}{7}, 14 t \right)
    \end{pmatrix}
    - 2 Q^{-1/168} e_1
    \label{eq:lint7}
\end{eqnarray}
which transforms under the modular $S$-transformation as 
\begin{eqnarray}
    \mathcal L^{(7)} (t)=\sqrt{\frac{\pi}{t}}\, M^{(7)}\, \mathcal L^{(7)}\left(\frac{\pi^2}{t}\right)
    \label{eq:SL7}
\end{eqnarray}
Here the matrix $M^{(7)}$  is the $3\times 3$ mixing matrix from \eqref{eq:M}
\begin{eqnarray} \label{eq: M7}
    M^{(7)}= \frac{2}{\sqrt{7}}
    \begin{pmatrix}
        -\sin(\frac{\pi}{7}) & -\sin(\frac{2\pi}{7}) & -\sin(\frac{3\pi}{7}) \\
        -\sin(\frac{2\pi}{7}) & \sin(\frac{3\pi}{7}) & -\sin(\frac{\pi}{7}) \\
        -\sin(\frac{3\pi}{7}) & - \sin(\frac{\pi}{7}) & \sin(\frac{2\pi}{7})
    \end{pmatrix}.
\end{eqnarray}
Then the transformation laws \eqref{eq: mf7 TL} take the more compact and symmetric matrix form. 
\begin{eqnarray} 
\label{eq:mf7-identity}
    \mathcal L^{(7)}(t)= 
    \begin{pmatrix}
    Q^{-1/168} \, X_1^{(7)}(Q) \cr 
    Q^{-25/168} \, X_2^{(7)}(Q) \cr
    Q^{-121/168} \, X_3^{(7)}(Q)
    \end{pmatrix}
    +\sqrt{\frac{\pi}{t}} \, M^{(7)}\, 
    \begin{pmatrix}
    Q_1^{-1/168} \, X_1^{(7)}(Q_1) \cr 
    Q_1^{-25/168} \, X_2^{(7)}(Q_1) \cr
    Q_1^{-121/168} \, X_3^{(7)}(Q_1)
    \end{pmatrix}
    \quad, \quad \Re t>0
\end{eqnarray}
This is precisely the general expression \eqref{eq:mock_q2_nonunary} at $p=7$, with the identifications:
\begin{eqnarray}
    X_1^{(7)}(Q) = -\mathcal{F}_0(Q), \qquad 
    X_2^{(7)}(Q) = -\mathcal{F}_1(Q), \qquad
    X_3^{(7)}(Q) = -Q \,\mathcal{F}_2(Q),
   \label{eq:mf$_7$-solution}  
\end{eqnarray}

\subsubsection{Proof of Theorem \ref{thm: general uniq thm} for $p=7$}
\label{sec:H27 proof}

In this section, we prove Theorem \ref{thm: general uniq thm} for the $p=7$ case. Indeed, existence follows from \eqref{eq:mf$_7$-solution}. Moreover, assuming \eqref{eq:mf$_7$-solution}, holomorphic solutions $X^{(7)}$ of \eqref{eq:mf7-identity} satisfying $X_2^{(7)}(0)=X_3^{(7)}(0)=0$ correspond exactly to holomorphic vectors $\bigl(\mathcal{F}_0, \mathcal{F}_1,\mathcal{F}_2\bigr)$ of \eqref{eq: mf7 TL} satisfying $\mathcal{F}_1(0)=0.$ Thus, for uniqueness, it suffices to prove that the classical vector $(\mathcal{F}_0, \mathcal{F}_1, \mathcal{F}_2)$ is the unique holomorphic solution of \eqref{eq: mf7 TL} satisfying $\mathcal{F}_1(0)=0.$

\begin{proof} 
    Suppose $(\widetilde {\mathcal{F}}_0,\widetilde {\mathcal{F}}_1, \widetilde {\mathcal{F}}_2 )$ is another such solution. We first claim that the growth condition \eqref{eq: growth mf7} for $\mathcal F_j$ and $ \widetilde {\mathcal F}_j$ is equivalent to the assumptions $\mathcal F_1(0)=0$ and $ \widetilde {\mathcal F}_1(0)=0,$ respectively. Indeed, this is clear from \eqref{eq: mf7 TL}, for we have 
    \begin{eqnarray}
        q=e^{-t}\to 1, \qquad q_1 = e^{- \frac{\pi^2}{t}} \to 0, \qquad t \to 0, \quad |\arg t| \le \theta_0.
    \end{eqnarray}
    Furthermore, it easy to verify that $ \sqrt{t} L(r,t)$, defined in \eqref{eq:L def}, is $O(1)$ as $t \to 0,$ with $|\arg t| \le \theta_0$. Therefore, recalling that $\mathcal F_j$ and $ \widetilde {\mathcal F}_j$ are holomorphic on $\mathbb D$ and noting that $|q_1|^{-1/42}$ is a strictly smaller growth than $|q_1|^{-25/42}$ as $t \to 0,$ with $|\arg t| \le \theta_0$, the claim follows from \eqref{eq: mf7 TL}. Thus, from \eqref{eq:mf$_7$-solution} and the above, the corresponding homogeneous functions $H_j^{(7)}(Q)= X_j^{(7)}(Q) - \widetilde{X_j^{(7)}}(Q)$ satisfies
    \begin{eqnarray} \label{eq: Hj mf7 bd}
        H_1^{(7)}(Q)= O(1), \qquad H_2^{(7)}(Q)= O(Q), \qquad H_3^{(7)}(Q)=O(Q), \qquad Q \to 0. 
    \end{eqnarray}
    Let $v(\tau)$ and $W(\tau)$ be the corresponding functions given by \eqref{eq:v def} and \eqref{eq:Wronskian}, respectively. In view of \S \ref{sec: Wronskian}, to finish the proof, it suffices to prove the corresponding Wronskian \eqref{eq:Wronskian} is identically zero. Indeed, when $p=7,$ the exponents are $\Delta_{(7,1)}=\frac{1}{168}$, $\Delta_{(7,2)}=\frac{25}{168}$, and $\Delta_{(7,3)}=\frac{121}{168}$, and the matrix in \eqref{eq:D} is thus $D^{(7)}= \text{diag} \bigl( e^{-2 \pi i /21}, e^{-8 \pi i/21 }, e^{-32 \pi i/21 } \bigr).$ It thus follows from \eqref{eq: M7}, \eqref{eq:W T}, \eqref{eq:W S}, and \eqref{eq:eta tlaws} that 
    \begin{eqnarray} \label{eq: G mf7}
        G(\tau):= \frac{W(\tau)^2}{\eta(\tau)^{24}}
    \end{eqnarray}
    is invariant under both $S:\tau\mapsto -1/\tau$ and $T:\tau\mapsto\tau+1,$ where we have used \eqref{eq:eta tlaws}, \eqref{eq:W T}, and \eqref{eq:W S}. Note that $G$ is holomorphic on $\mathbb H.$

    However, from \eqref{eq: Hj mf7 bd}, \eqref{eq: G mf7}, \eqref{eq:v def}, \eqref{eq:Wronskian}, and \eqref{eq:eta def}, it follows that $G(\tau)= Q h(Q),$ with $h(Q)$ holomorphic on $\mathbb D.$ Thus, $G$ extends holomorphically to the cusp $\infty$ with $G(\infty)=0$. However, the only holomorphic modular functions for $SL(2,\Z)$ are constants, and since $G$ is such a function with $G(\infty)=0,$ it follows that $G \equiv 0$ on $\mathbb H$. This in turn implies that $W(\tau)\equiv 0$ on $\mathbb H$. Applying Lemma \ref{lem: wronskian} completes the proof.
\end{proof}

\subsubsection{Proof of Proposition \ref{prop: general sol F_7}} \label{sec: prop gen sol F7 proof}
\begin{proof}
    Let $\mathcal V_7$ be the complex vector space of homogeneous holomorphic solutions $H^{(7)}$ to \eqref{eq:mf7-identity} satisfying $H^{(7)}_3(0)=0.$ Then by the uniqueness theorem, i.e. Theorem \ref{thm: general uniq thm} for the $p=7$ case, the map $H^{(7)} \mapsto H^{(7)}_2(0)$ is an injective linear map $\mathcal V_7 \hookrightarrow \C$. Hence $\text{dim} ~\mathcal V_7 \le 1.$ However, $H^{(7)}$ from Theorem \ref{thm: homog sol for any p} (3) satisfies $H^{(7)} \in \mathcal V_7$ and $H_2^{(7)}(0) \neq 0$ (see \eqref{eq: H^(7)}). The claim then follows. 
\end{proof}

\subsection{Order 11 Uniqueness}
\label{sec:mf11}

In this section, we prove, under a normalization condition, the uniqueness of the Mock 11 identities listed below in \eqref{eq:mf11-identity}.

\subsubsection{The transformation laws for order 11 mock thetas} \label{sec: TL mf11}

There is no accepted set of "order 11" mock theta functions. However, the structure \eqref{eq:mock_q2_nonunary} that we have described for order 5 and order 7 mock thetas extends straightforwardly to $p=11$, and we will refer to this as the order 11 mock theta case.

For $p=11$, we have $\ell_{11}=5$, so the integral vectors defined in \eqref{eq:ls} have $5$ components. The vector $X_{11}$ satisfies the mock modular transformation laws \eqref{eq:mock_q2_nonunary} at $p=11$:
\begin{eqnarray} 
\label{eq:mf11-identity}
\sqrt{\frac{528 t}{\pi}}\, L^{(11)}_j(t) &=& Q^{-\Delta_{(11,j)}}\, X^{(11)}_{j}(Q) + 
    \sqrt{\frac{\pi}{t}} \sum_{k=1}^{5}  M_{jk}^{(11)} \, Q_1^{\, -\Delta_{(11,k)}}\, X^{(11)}_{k} \bigl(Q_1\bigr)
\end{eqnarray}
For $p=11$ the exponents in \eqref{eq:delta-mock} are: 
\begin{eqnarray}\label{eq: delta11}
    \Delta_{(11,1)}=\frac{25}{264}, \quad \Delta_{(11,2)}=\frac{1}{264}, \quad \Delta_{(11,3)}=\frac{49}{264}, \quad \Delta_{(11,4)}=\frac{169}{264}, \quad \Delta_{(11,5)}=\frac{361}{264}.
\end{eqnarray}

In \cite{ACDGO} the $p=11$ decomposition in \eqref{eq:mock_q2_nonunary} was solved numerically, leading to dual $q$-series, the first few terms of which are listed in Table \ref{tab:mock11-table-match}. As noted in \cite{ACDGO}, the $Q$-series in Table \ref{tab:mock11-table-match} agree with Zagier's proposed expression, (on page 15 of \cite{Zag09}), given here in \eqref{eq:zagier11}, up to a different indexing convention. Explicitly, we have the identification in \eqref{eq: X^(11) M^(11)} of $X^{(11)}$ with Zagier's $\mathcal M_{11}$ in \eqref{eq:zagier11}. 
\begin{table}[H]
\centering
\begin{tabular}{ | m{1em} | m{8cm}| } 
  \hline
  $j$ & Non-unary dual $Q$-series $X_{j}^{(11)}(Q)$  \\ 
  \hline
  $1$ & $1/6 \left(-1+15 Q+65 Q^2+175 Q^3+\dots\right)$ 
  \\
  \hline
  $2$ & $1/6\left(5+5 Q+15 Q^2+60 Q^3+125 Q^4+\dots\right)$ 
  \\
  \hline
  $3$ & $1/6\left(-16 Q-55 Q^2-155 Q^3-385 Q^4+\dots\right)$ 
  \\
  \hline
  $4$ & $1/6\left(-5 Q-22 Q^2-60 Q^3-155 Q^4+\dots\right)$ 
  \\
  \hline
  $5$ & $1/6\left(5 Q^2+33 Q^3+99 Q^4+268 Q^5+\dots\right)$ 
  \\
  \hline
\end{tabular}
\caption{This table displays the first coefficients of the mock theta $Q$-series $X_{j}^{(11)}(Q)$, for $1\leq j \leq 5$, appearing in the order-11 mock theta identities \eqref{eq:mf11-identity}.}  
\label{tab:mock11-table-match}
\end{table}

\subsubsection{Proof of Theorem \ref{thm: general uniq thm} for $p=11$}\label{sec:H211 proof} 
\begin{proof}
    The equivalence follows from \eqref{eq:mf11-identity} and \eqref{eq: delta11}. Now suppose $\widetilde{X}^{(11)}$ is another such solution, and let $H= \widetilde{X}^{(11)}(Q) - X^{(11)}(Q)$, so that $H^{(11)}_3(Q),H^{(11)}_4(Q)=O(Q)$ and $H^{(11)}_5(Q)=O(Q^2)$ as $Q\to 0.$ Let $v(\tau)$ and $W(\tau)$ be the corresponding functions given by \eqref{eq:v def} and \eqref{eq:Wronskian}, respectively. From \eqref{eq: delta11}, we see that the matrix in \eqref{eq:D} is 
    \begin{eqnarray}
        D^{(11)}= \text{diag} \bigl( e^{-3 \pi i /11}, e^{- \pi i /11}, e^{17 \pi i /11}, e^{7 \pi i /11}, e^{13 \pi i /11} \bigr).
    \end{eqnarray}
    Thus, \eqref{eq:W T} and \eqref{eq:W S} become
    \begin{equation}\label{eq:W T11}
        W(\tau+1)=\det(D^{(11)})\,W(\tau)= -W(\tau),
    \end{equation}
    \begin{equation}\label{eq:W S11}
        W \left(-\frac1\tau\right) = \det(-M^{(11)})\,\tau^{20}W(\tau)=- \tau^{20}W(\tau).
    \end{equation}
    Now recall that the modular discriminant
    \begin{eqnarray} \label{eq: Delta tau}
        \Delta(\tau):= \eta(\tau)^{24}= Q \prod_{n\ge 1} (1-Q^n)^{24},
    \end{eqnarray}
    is a weight 12 cusp form for $SL(2,\Z).$ Thus, it follows that 
    \begin{eqnarray}
        G(\tau):= \frac{W(\tau)^2}{\Delta(\tau)^3}
    \end{eqnarray}
    is a holomorphic modular form of weight 4 for $SL(2,\Z).$ Indeed, holomorphy on $\mathbb H$ is clear and holomorphy at the cusp $Q=0$ follows from the normalization with \eqref{eq: normaliz init cond}, \eqref{eq:Wronskian}, and \eqref{eq:v def}, since they imply that $W(\tau)^2= Q^3 h(Q)$ for some holomorphic function $h$ on $\mathbb D.$

    Therefore, $G$ is a holomorphic modular form of weight 4 for $SL(2,\Z)$:
    \begin{eqnarray}
        G \in M_4(SL(2,\Z))= \C E_4,
    \end{eqnarray}
    where $E_4$ is the Eisenstein series of weight 4 (see \cite[Thm. 4, p. 88]{Serrebook}). Thus, there exists $c \in \C$, such that 
    \begin{eqnarray}
        G= c E_4.
    \end{eqnarray}
    We prove that $c=0$ to finish the proof. To see this, we consider the values of each side at the elliptic fixed point $\tau=i.$ By definition, we have $E_4(i) \neq 0$. However, $G(i)=0,$ for $\Delta(i) \neq 0$ and $W(i)=0,$ which follows from \eqref{eq:W S11}: 
    \begin{eqnarray}
        W(i)= W(-1/i) = - i^{20} W(i)= -W(i).
    \end{eqnarray}
    This proves that $c=0,$ so that $G \equiv 0,$ which implies that $W \equiv 0.$ Applying Lemma \ref{lem: wronskian} completes the proof.
\end{proof}

\subsection{Proof of Proposition \ref{P:decay}}(Uniqueness of \'Ecalle-Borel summed transseries) \label{S:P:decay}

\begin{proof}
Split the Laplace transform as
\begin{equation}\label{eq:splitlap}
(\mathcal LF)(x)
=
\int_0^\epsilon e^{-v x}F(v)\,dv
+
\int_\epsilon^\infty e^{-v x}F(v)\,dv=:e^{ -\epsilon x}h(x)+\int_\epsilon^\infty e^{-v x}F(v)\,dv
\end{equation}
Since $F\in L^1(0,\epsilon)$, differentiation under the integral sign is
justified by dominated convergence, and therefore $h$ is entire.  Furthermore, the modulus of the last integral in \eqref{eq:splitlap} is bounded by $e^{-\epsilon x}\|F\|_1$, hence $h$ is bounded as $x\to +\infty$. 
Moreover, since $F\in L^1$,   $h$ is bounded on the imaginary axis.

Applying the Phragmén--Lindelöf theorem in the first and fourth quadrants,
we conclude that $h$ is bounded in the closed right half-plane.

If $x=-s<0$, then
\begin{equation}\label{eq:estlower4}
\left|
e^{-s\epsilon}
\int_0^\epsilon e^{s v}F(v)\,dv
\right|
=
\left|
\int_0^\epsilon
e^{s(v-\epsilon)}F(v)\,dv
\right|
\le
\|F\|_1.
\end{equation}
Applying the Phragmén--Lindelöf theorem again in the second and third
quadrants, we conclude that $h$ is bounded on the whole complex plane.
Hence, by Liouville's theorem, $h$ is constant.
The PV case is very similar.

Finally, by the Riemann--Lebesgue lemma, $h(is)\to 0$ as $s\to +\infty$, so the constant is zero.
Therefore
\begin{equation}\label{eq:fourier2}
\int_0^\epsilon
F(v)\, e^{-i s v}\,dv
=
0,
\qquad s\in\mathbb R.
\end{equation}
The injectivity of the Fourier transform on $L^1$ now implies
$F=0$ almost everywhere on $[0,\epsilon]$.
\end{proof}

 \subsection{Proof of Proposition \ref{P:uniq-transseries}}(Uniqueness of transseries decompositions)
 \label{S:proof-unique-transseries}

Since $F\in L^1(\mathbb R^+)$,
\(
(\mathcal LF)(x)\to 0\) as 
\(x\to+\infty.\)
Multiplying \eqref{eq:iden} by $e^{ -mrx}$ gives
\(
c_{-m}+o(1)=0,
\)
hence $c_{-m}=0$. Repeating the argument inductively yields
\(
c_k=0,
\,\,
-m\le k\le0.
\)
Therefore
\begin{equation}\label{eq:lap}
(\mathcal LF)(x)
=
-\sum_{k\ge1}c_ke^{-krx}
=
O(e^{-rx}),
\qquad x\to+\infty.
\end{equation}
By Proposition~\ref{P:decay},
$F=0$ almost everywhere on $(0,r)$.
Since $F$ is real analytic,
it follows that
$F\equiv0$ on $\mathbb R^+$ (and therefore on its maximal domain, or Riemann surface, of analyticity).

Equation \eqref{eq:iden} therefore reduces to
\(
\sum_{k\ge0}c_ke^{-krx}=0,\,\, x>a.
\)
Setting
\(
z=e^{-rx},
\) and $f(z)=\sum_{k\ge0}c_kz^k$. By the assumptions, $f$ is analytic in a small disk and vanishes on  an interval $(0,\delta)$, entailing $c_k=0$ for all $k\ge 1$.

\subsection{Order 3 mock theta functions}
\label{sec:mf$_3$}

We now recall some basic properties of Jacobi theta functions that are relevant for this section. Denote by $\vartheta_3$ the Jacobi theta function \cite[eq. (2.5), (2.12c)]{ThetaVocab}
\begin{eqnarray}\label{eq: theta3 def}
    \vartheta_3(\tau)= \sum_{k \in \Z} q^{k^2} = \prod_{n=1}^\infty (1-q^{2n})(1+q^{2n-1})^2.
\end{eqnarray}
In particular, $\vartheta_3$ is holomorphic and nonvanishing on $\mathbb H$. Moreover, its transformation laws are \cite[eq. (2.15c)]{ThetaVocab}
\begin{equation}\label{eq:theta3 S}
    \vartheta_3(\tau+2)=\vartheta_3(\tau), \qquad \vartheta_3(-1/\tau)=\sqrt{-i\tau}\,\vartheta_3(\tau), \qquad \tau\in\mathbb H.
\end{equation}
We also denote by $\vartheta_2$ the Jacobi theta function \cite[eq. (2.5), (2.12b)]{ThetaVocab}
\begin{eqnarray}\label{eq: theta2 def}
    \vartheta_2(\tau)= \sum_{k \in \Z} q^{(k+1/2)^2} = 2 q^{\frac{1}{4}} \prod_{n=1}^\infty (1-q^{2n})(1+q^{2n})^2.
\end{eqnarray}
Another useful set of transformation laws are \cite[eq. (2.14c), (2.15d)]{ThetaVocab}
\begin{eqnarray}
    \vartheta_3(\tau+1)= \vartheta_4(\tau), \qquad \vartheta_4(-1/\tau)= \sqrt{-i \tau} \vartheta_2(\tau),
\end{eqnarray}
where $\vartheta_4$ is defined in \cite[eq. (2.5)]{ThetaVocab}. Therefore, if $\tau= 1- 1/z,$ then 
\begin{eqnarray}\label{eq: mf3 theta identity}
    \vartheta_3(\tau)= \vartheta_3(1- 1/z)= \vartheta_4(-1/z)= \sqrt{-i z} \vartheta_2(z).
\end{eqnarray}

We now recall the standard notions of a \emph{cusp}, the \emph{width} of a cusp, and \emph{meromorphy at a cusp} \cite[p. 5-7]{Stein}.

\begin{defn} \label{def:cusp}
Let \(\Gamma\leq SL(2, \Z)\) be a subgroup of finite index.
\begin{enumerate}
    \item A \emph{cusp} of \(\Gamma\) is a \(\Gamma\)-equivalence class in
    \(\Q\cup\{\infty\}\).

    \item Let \(\mathfrak a\in \Q\cup\{\infty\}\) be a cusp of \(\Gamma\), and
    choose \(\sigma_{\mathfrak a}\in SL(2, \Z)\) such that
    \(\sigma_{\mathfrak a}(\infty)=\mathfrak a\). We define the \emph{width} of the cusp
    \(\mathfrak a\) to be the smallest positive integer \(w_{\mathfrak a}\) such that
    \[
    \pm T^{w_{\mathfrak a}}
    \in
    \sigma_{\mathfrak a}^{-1}\Gamma_{\mathfrak a}\sigma_{\mathfrak a},
    \qquad
    \Gamma_{\mathfrak a}:=\{\gamma\in\Gamma:\gamma\mathfrak a=\mathfrak a\}.
    \]

    \item A meromorphic \(\Gamma\)-invariant function \(F\) on \(\mathbb H\) is
    said to be meromorphic at the cusp \(\mathfrak a\) if
    \(F(\sigma_{\mathfrak a}\tau)\) admits a Laurent expansion in
    \[
    q_{\mathfrak a}:=\exp\!\Bigl(\frac{2\pi i\tau}{w_{\mathfrak a}}\Bigr)
    \]
    as \(\Im\tau\to\infty\), of the form
    \begin{equation}\label{eq:cusp-laurent}
    F(\sigma_{\mathfrak a}\tau)=\sum_{n\ge n_\mathfrak a} c_n q_{\mathfrak a}^{\,n},
    \qquad c_{n_\mathfrak a}\neq 0.
    \end{equation}
    We then define the order of $F$ at the cusp $\mathfrak a$ to be
    \[
    \ord_{\mathfrak a}(F)=n_\mathfrak a.
    \]
\end{enumerate}
\end{defn}
For more details, see \cite[Sections 2, 4]{MilneMF} and \cite{Paule}. 
\begin{remark}
Indeed, the coefficients $c_n$ in \eqref{eq:cusp-laurent} depend on the choice of $\sigma_{\mathfrak a}$. However, the width $w_{\mathfrak a}$ and the order $\ord_{\mathfrak a}(F)=n_{\mathfrak a}$ are well-defined, i.e. they do not depend on the choice of $\sigma_{\mathfrak a}$ (see, e.g., \cite[p.~27-29]{Bhattacharya}). 
\end{remark}

\begin{remark}
    If $\ord_{\mathfrak a}(F) \ge0$ for some cusp $\mathfrak a$,  we say $F$ is holomorphic at $\mathfrak a$.
\end{remark}

\begin{ex}
    It is easy to show that the action of $SL(2, \Z)$ on $\Q \cup \{\infty\}$ has only one orbit, represented by $\infty,$ and thus the only cusp of $SL(2,\Z)$ is $\infty$ \cite[p. 5]{Stein}.
\end{ex}

\begin{ex} \label{ex:theta-group}
The theta group
\begin{equation}
    \Gamma_\theta:=\langle\, S,T^2\,\rangle,\qquad S:\tau\mapsto-1/\tau,\quad T^2:\tau\mapsto\tau+2.
\end{equation}
has two cusps: $1$ and $\infty$ (note that $-1\sim1$ under $T^2$), since one can show that every rational number is $\Gamma_\theta-$equivalent to either $1 \, (=1/1)$ or $\infty \, (=1/0)$ (see e.g., \cite[p.~ 20-29]{Schultz} or \cite[p.~ 76, 77]{GottscheZagier}). It is also standard and easy to show that the cusps $1$ and $\infty$ of $\Gamma_\theta$ have widths $w_1=1$ and $w_\infty=2,$ respectively (see, for example, \cite[p.~65]{maass} or apply Stein's algorithm \cite[p.~ 9]{Stein}).
\end{ex}

\subsection{Proof of Theorem \ref{thm: uniq-mf3}} Uniqueness for Order 3 Mock Theta Functions
\label{P:mf3}
\begin{proof}[Proof 1]
Suppose $(\omega_1,f_1)$ and $(\omega_2, f_2)$ are two such solutions, and let $\Delta \omega:= \omega_1-\omega_2$ and $\Delta f= f_1-f_2.$ Then we get the homogeneous equations 
\begin{equation}\label{eq:omega homog}
q^{2/3}\,\Delta \omega(-q) = -\sqrt{\frac{\pi}{t}}\,q_1^{2/3}\,\Delta \omega(-q_1), 
\end{equation}
\begin{eqnarray} \label{eq: homg f w}
    q^{\frac{2}{3}} \Delta \omega(q)= \sqrt{\frac{\pi}{4 t}} q_1^{-\frac{1}{12}} \Delta f(q_1^2).
\end{eqnarray}
If $\Delta \omega \equiv 0,$ the claim follows from \eqref{eq: homg f w}. Thus, assume for contradiction that $\Delta \omega \not\equiv 0$.

Now consider the holomorphic function on $\mathbb H$ \footnote{Both sides of each of \eqref{eq:omega homog} and \eqref{eq: homg f w} are holomorphic on $\{t:\Re(t)>0\},$ and they coincide for $t\in \mathbb{R}^+$. Hence,  by the identity theorem,  \eqref{eq:omega homog} and \eqref{eq: homg f w} hold for all $\Re (t)>0$.}
\begin{equation}\label{eq:h def}
h(\tau):=\left(\frac{q^{2/3}\Delta \omega(-q)}{\vartheta_3(\tau)}\right)^6,
\end{equation}
where $\vartheta_3$, given in \eqref{eq: theta3 def}, is holomorphic and nonvanishing on $\mathbb H$. Hence, by \eqref{eq:omega homog} and \eqref{eq:theta3 S}, $h$ is holomorphic on $\mathbb H$ and invariant under the theta group $\Gamma_\theta=\langle S, T^2 \rangle$, which has only two cusps: $1$ and $\infty$ \cite[p.~ 76, 77]{GottscheZagier} ($-1\sim1$ under $T^2$). Moreover, $h$ has the following behavior at the cusps.

\begin{lemma} If the function $h$, defined in \eqref{eq:h def}, is not identically zero, then it has the following behavior at the cusps of $\Gamma_\theta$.
    \begin{enumerate}
        \item[a.] $h$ is holomorphic at the cusp $\infty$ and $ \ord_\infty(h) \ge 4.$
        \item[b.] $h$ is meromorphic at the cusp $1$ and $\ord_1(h) \ge -1.$
\end{enumerate}
\label{lem: h beh}
\end{lemma}
\begin{remark}
    Since $\Delta \omega \not\equiv 0$, we have $h \not\equiv 0,$ and it follows by Lemma \ref{lem: h beh} that $\ord_j(h)<\infty $ for $j \in \{1, \infty\}.$ Moreover, Lemma \ref{lem: h beh}$(a)$ implies that $h(\infty)=0,$ so if $h$ is a constant function, we immediately get the contradiction $h \equiv 0$. We thus assume WLOG that $h$ is non-constant.
\end{remark}
\noindent Assume Lemma \ref{lem: h beh}. Then $h$ extends to a non-constant meromorphic function on the compact Riemann surface  \cite[p.~13]{Duke} $ X_{\Gamma_\theta}:=\Gamma_\theta\backslash(\mathbb H\cup \Q \cup \{1, \infty\}),$ and we have \cite[cor. 10.22.]{Forster}
\begin{equation}\label{sum: ord}
    \sum_{P\in X_{\Gamma_\theta}}ord_P(h)=0.
\end{equation}
 Since $h$ is holomorphic on $\mathbb H,$ we have $\ord_P(h) \ge 0$ for all $P\in X_{\Gamma_\theta} \setminus \{1, \infty\}$ \footnote{{Holomorphy of $h$ on $\mathbb H$ implies holomorphy on $X_{\Gamma_\theta} \setminus \{1, \infty\}$ (this follows from \cite[Lemma 4.11]{MilneMF} with $k=0,$ for example).}}. Therefore, if $h \not \equiv 0,$ \eqref{sum: ord} and Lemma \ref{lem: h beh} yield the contradiction 
\begin{equation}
0=\sum_{P\in X_{\Gamma_\theta}}\ord_P(h) \ge \ord_1(h)+\ord_\infty(h)\ge 3.
\end{equation}
Thus, $h\equiv 0,$ which implies $\Delta \omega \equiv 0.$ Equation \eqref{eq: homg f w} then implies $\Delta f \equiv 0.$ To complete the proof, it thus suffices to prove Lemma \ref{lem: h beh}.
\begin{proof}[Proof of Lemma \ref{lem: h beh}]
For \textbf{(a)}, recall that the width of the cusp $\infty$ is $2$, so that the local parameter is $q= e^{\pi i \tau}.$ By definition $h$ in \eqref{eq:h def} and that of $\vartheta_3$ in \eqref{eq: theta3 def}, we get
\begin{eqnarray}
    h(\tau)= q^4 \left(\frac{\Delta \omega(-q)}{\prod_{n=1}^\infty (1-q^{2n})(1+q^{2n-1})^2}\right)^6.
\end{eqnarray}
The factor multiplying $q^4$ is holomorphic at $q=0.$ This proves \textbf{(a)}. For \textbf{(b)}, choose
\begin{equation} \label{eq: z tau}
    \sigma_1:=\begin{pmatrix}1&-1\\ 1&0\end{pmatrix}\in SL(2,\Z), \qquad \tau:=\sigma_1 z=\frac{z-1}{z}=1-\frac1z, 
\end{equation}
so that $\sigma_1(\infty)=1.$ Therefore, since the cusp 1 has width 1, we need to study behavior of $h(\sigma_1 z)$ in the local parameter $u= e^{2 \pi i z}$ near $u=0$.

Indeed, by \eqref{eq: z tau}, we have $-q(\tau)= e^{- \pi i /z}.$ Thus, using \eqref{eq: homg f w} at $-1/z \in \mathbb H $ and \eqref{eq: mf3 theta identity} and \eqref{eq: theta2 def}, we get 
\begin{eqnarray} \label{eq: h cusp 1}
    h(\sigma_1 z)= \left(\frac{q(\tau)^{2/3}\Delta \omega(-q(\tau))}{\vartheta_3(\tau)}\right)^6 = \frac{1}{2^{12} u} \left(\frac{\Delta f(u)}{\prod_{n= 1}^\infty (1-u^n)(1+u^n)^2}\right)^6.
\end{eqnarray}
The factor multiplying $u^{-1}$ is holomorphic at $u=0.$ This proves \textbf{(b)}.
\end{proof}
This completes the proof.
\end{proof}

\subsection{An alternative proof  of uniqueness for Order 3 Mock Theta functions}\label{alt P:mf3} In this section, we give an alternative proof of Theorem \ref{thm: uniq-mf3}  that uses the modular lambda function.

\subsubsection{The modular lambda function} We define 
\begin{eqnarray}
    \lambda(\tau):= \frac{\vartheta_2(\tau)^4}{\vartheta_3(\tau)^4}= 16 q \prod_{n=1}^\infty \left( \frac{1+q^{2n}}{1+q^{2n-1}} \right)^8, \qquad \tau \in \mathbb H.
\end{eqnarray}
Its transformation laws are 
\begin{eqnarray}\label{eq: lambda transflaw}
    \lambda(\tau+1)= \frac{\lambda(\tau)}{\lambda(\tau)-1}, \qquad \lambda(-\frac{1}{\tau}) = 1- \lambda(\tau).
\end{eqnarray}
Using \eqref{eq: lambda transflaw}, it follows that $\lambda$ is invariant under $\Gamma(2)= \langle -I, T^2, S T^2 S \rangle$. In particular, we have  $\lambda(\tau+2)= \lambda(\tau)$, and we define 
\begin{eqnarray}
    \lambda_q(q):= \lambda(\tau)= 16 q - 128 q^2 +704q^3 + \cdots.
\end{eqnarray}
Thus, since $\lambda_q(0)=0$ and $\lambda_q'(0) = 16 \neq 0$, the inverse function theorem shows that the function $q \to \lambda_q(q)$ has a holomorphic inverse near $q=0$, and thus $q$ is a holomorphic function of $s= \lambda_q(q)$ near $s=0$, with $q= \frac{s}{16}+ O(s^2).$

An important property used in the proof is that $\lambda$ effects a one-to-one conformal map of the domain $\Omega$ bounded by $i \R^+,$ $1+i \R^+,$ and the circle $|\tau-1/2|=1/2$ onto $\mathbb H$ \cite[Thm. 7, p. 281]{Ahlfors}. Letting $\Omega'$ be the region that is symmetric to $\Omega$ with respect to the imaginary axis, we obtain the bijection $\lambda: \bar \Omega \cup \Omega' \to \C \setminus \{0,1\}$ \cite[p. 281]{Ahlfors}. Moreover, every $\tau \in \mathbb H$ is $\Gamma(2)$-equivalent to exactly one point in $\bar \Omega \cup \Omega'$ \cite[Thm. 8, p. 281]{Ahlfors}. The latter and injectivity of $\lambda$ thus imply that for every $\tau_1 , \tau_2 \in \mathbb H,$ 
\begin{eqnarray} \label{eq: prop of lambda}
    \lambda(\tau_1) = \lambda(\tau_2) \implies \tau_2 = \gamma \tau_1
\end{eqnarray}
for some $\gamma \in \Gamma(2).$ This property will be used in the proof below.

\subsubsection{The proof} We carry out the same steps as in the previous proof up to the construction of the holomorphic function in $\mathbb H$
\begin{equation}\label{eq:h def2}
h(\tau):=\left(q^{2/3}\Delta \omega(-q) \vartheta_3(\tau)^{-1} \right)^6,
\end{equation}
which is invariant under $\Gamma_\theta \supset \Gamma(2)$. As discussed above, it suffices to prove that $h \equiv 0.$

For $\tau \in \mathbb H,$ define 
\begin{eqnarray}\label{eq: F h}
    F(\lambda(\tau))= h(\tau), \qquad \tau \in \mathbb H.
\end{eqnarray}
Then $F(s)$ is well-defined and single-valued on $\C \setminus \{0,1\}$ since we have \eqref{eq: prop of lambda} and $h$ is $\Gamma(2)$-invariant. Moreover, $F: \C \setminus \{0,1\} \to \C$ is holomorphic since $\lambda$ is locally biholomorphic at every $\tau \in \mathbb H$.

On the other hand, since $h$ is invariant under $\Gamma_\theta$ and $\lambda$ satisfies \eqref{eq: lambda transflaw}, we have
\begin{eqnarray}
    F(\lambda(\tau))=h(\tau)= h\left(-\frac{1}{\tau}\right)= F\left(\lambda\left(- \frac{1}{\tau}\right)\right) = F(1- \lambda(\tau)).
\end{eqnarray}
Therefore, since the image of $\lambda$ is $\C \setminus \{0,1\}$, it follows that $F(s)=F(1-s)$ for all $s \in \C \setminus \{0,1\}$. However, since $q$ is a holomorphic function of $s= \lambda_q(q)$ near $s=0,$ with $q(s)= \frac{s}{16}+ O(s^2),$ it follows from \eqref{eq: F h} that $s=0$ is a removable singularity of $F(s)$, and $F$ is holomorphic at $s=0$ of order at least $4$. By symmetry, the same is true for $s=1.$ Hence $F(s)= s^4 (1-s)^4 F_1(s),$ for some entire function $F_1$.

We complete the proof by showing that $F_1 \equiv 0.$ Indeed, this will follow from its behavior as $|s| \to \infty.$ As in the previous proof, for $z \in \mathbb H$ let
\begin{equation} \label{eq: z tau2}
    \sigma_1 z=1-\frac1z, \qquad \sigma_1:=\begin{pmatrix}1&-1\\ 1&0\end{pmatrix}\in SL(2,\Z).
\end{equation}
Next observe that from \eqref{eq: lambda transflaw}, we have
\begin{eqnarray}\label{eq: sigma z.}
    \lambda(\sigma_1 z)= \lambda(1-\frac{1}{z})= 1- \frac{1}{\lambda(z)}.
\end{eqnarray}
Now for sufficiently large $|s|,$ let $q_s:= \lambda_q^{-1} \bigl(\frac{1}{1-s} \bigr)$ ($\lambda_q$ is biholomorphic near $0$). But then $q_s= \frac{1}{16(1-s)} +O(|s|^{-2})$. Thus, for all sufficiently large $|s|,$ $0< |q_s|<1,$ and there exists $z_s \in \mathbb H$ such that $q_s= e^{\pi i z_s}.$ In particular, taking absolute values, we see that $\Im z_s= - \frac{1}{\pi} \log|q_s| \to \infty$ as $|s| \to \infty.$ Moreover, $\lambda(z_s)= \lambda_q(q_s)= \frac{1}{1-s}.$

Therefore, from \eqref{eq: sigma z.}, we get for all sufficiently large $|s|,$
\begin{eqnarray}
    \lambda(\sigma_1 z_s)= 1- \frac{1}{\lambda(z_s)}= s.
\end{eqnarray}
Thus, putting $u_s= q_s^2$ and using \eqref{eq: h cusp 1}, we get
\begin{eqnarray} \label{eq: h cusp 1,2}
    F(s)= F(\lambda(\sigma_1 z_s) )= h(\sigma_1 z_s)= \frac{1}{2^{12} u_s} \left(\frac{\Delta f(u_s)}{\prod_{n= 1}^\infty (1-u_s^n)(1+u_s^n)^2}\right)^6.
\end{eqnarray}
However, the factor multiplying $u_s^{-1}$ is bounded near $u_s=0$, and 
\begin{eqnarray}
    u_s^{-1}= q_s^{-2}= \Bigl[\frac{1}{16(1-s)} \bigl(1 +O(|s|^{-1})\bigr) \Bigr]^{-2}= O(|s|^{2}), \qquad |s| \to \infty,
\end{eqnarray}
so $F(s)=O(|s|^{2})$ as $|s| \to \infty.$ This in turn implies that 
\begin{eqnarray}
    F_1(s)= \frac{F(s)}{s^4(1-s)^4}= O(|s|^{-6}), \qquad |s| \to \infty,
\end{eqnarray}
but then $F_1(s)$ is entire and vanishing at $\infty,$ so by Liouville's theorem, $F_1 \equiv0.$ \hfill $\square$

\section{Results for  resurgent and
modular structure of Mordell-Appell integrals}\label{S:Mordell-transformations}
\subsection{The \(S\)-transform}

\begin{remark}\label{R:Borel}
The Mordell-Appell integrals occurring in mock theta transformation laws can be brought to the form 
\begin{equation}\label{eq:Mordell-Borel}
\int_0^\infty e^{-rx^2/t}F(x)\,dx
=
\frac12
\int_0^\infty
e^{-rp/t}p^{-1/2}F(\sqrt p)\,dp
\end{equation}
where \(F\) is a rational function of $e^{-x}$ with exponential decay along 
\(\RR^+\).

Thus every Mordell integral is a Laplace transform of a meromorphic
function decaying on \(\RR^+\), namely the standard Borel-Laplace form.
We shall therefore use the terminology ``Borel sum",
``medianization", etc., interchangeably for the corresponding Mordell
integrals.
\end{remark}

Since all kernels occurring below are even,
\[
\widehat F(\xi)
=
2
\int_0^\infty
F(u)\cos(\xi u)\,du,
\]
and therefore we have:
\begin{lemma}

\begin{equation}\label{eq:fourier-cosine2}
2\sqrt{\frac{\pi}{t}}
\int_0^\infty
e^{-c^2u^2/t}F(u)\,du
=
\int_0^\infty
e^{-t y^2}
\left(
\int_0^\infty
\cos(2cuy)F(u)\,du
\right)dy.
\end{equation}
\end{lemma}
\begin{remark}
    Equation \eqref{eq:fourier-cosine2} holds in considerably
more general settings: neither meromorphicity nor parity of the Borel kernel is
essential.  Under suitable resurgent continuation and exponential
bounds, the two exponential components of the Fourier kernel may be
treated separately in the upper and lower half-planes, yielding an
analogous large-to-small connection.
\end{remark}

 \subsection{False theta identification of Principal-Value (PV) integrals}\label{S:trT}

For \(0<a<p\), one has in the upper half-plane
\[
\frac{\sin((p-a)u)}{\sin(pu)}
=
\sum_{k=0}^{\infty}
\left(
e^{i(2pk+a)u}
-
e^{i(2p(k+1)-a)u}
\right),
\qquad \Im u>0,
\]
with the conjugate expansion in the lower half-plane. Pairing the two
lateral integrals and using
\[
\frac12\left(
\int_0^\infty e^{-rx^2/A+imx}\,dx
+
\int_0^\infty e^{-rx^2/A-imx}\,dx
\right)
=
\frac12\sqrt{\frac{\pi A}{r}}\,
e^{-Am^2/(4r)},
\]
gives
\begin{equation}\label{eq:PViden}
\operatorname{PV}\int_0^\infty
e^{-rx^2/A}
\frac{\sin((p-a)x)}{\sin(px)}\,dx
=
\frac12\sqrt{\frac{\pi A}{r}}
\sum_{k=0}^{\infty}
\left[
e^{-A(2pk+a)^2/(4r)}
-
e^{-A(2p(k+1)-a)^2/(4r)}
\right].
\end{equation}

Equivalently, if \(\chi_{p,a}\) is the \(2p\)-periodic function
\[
\chi_{p,a}(m)
=
\begin{cases}
1,&m\equiv a\pmod{2p},\\
-1,&m\equiv -a\pmod{2p},\\
0,&\text{otherwise},
\end{cases}
\]
and
\[
\Psi_{p,a}^{(r)}(A)
=
\sum_{m=1}^{\infty}
\chi_{p,a}(m)e^{-Am^2/(4r)},
\]
then
\[
\operatorname{PV}\int_0^\infty
e^{-rx^2/A}
\frac{\sin((p-a)x)}{\sin(px)}\,dx
=
\frac12\sqrt{\frac{\pi A}{r}}\,
\Psi_{p,a}^{(r)}(A).
\]

The cosine case is obtained in exactly the same way. Namely, with the
\(4p\)-periodic function
\[
\widetilde\chi_{p,a}(m)
=
\begin{cases}
1,&m\equiv a,\ 2p-a\pmod{4p},\\
-1,&m\equiv 2p+a,\ 4p-a\pmod{4p},\\
0,&\text{otherwise},
\end{cases}
\]
and
\[
\widetilde\Psi_{p,a}^{(r)}(A)
=
\sum_{m=1}^{\infty}
\widetilde\chi_{p,a}(m)e^{-Am^2/(4r)},
\]
one has
\[
\operatorname{PV}\int_0^\infty
e^{-rx^2/A}
\frac{\cos((p-a)x)}{\cos(px)}\,dx
=
\frac12\sqrt{\frac{\pi A}{r}}\,
\widetilde\Psi_{p,a}^{(r)}(A).
\]

\subsection{The Modular \(T\)-transformation for general prime $p$}
$ $

We now apply the preceding false-theta identification to the Mordell-Appell 
integrals in the general prime $p$ family \eqref{eq:l}.  To put the sine and cosine
combinations on the same footing, we first use
\[
JS_{(12p,2a)}(t/2)=JS_{(6p,a)}(t)
\]
and set \(\widehat L_j^{(p)}(t)=L_j^{(p)}(t/2)\).  Thus
\(\widehat L_j^{(p)}\) and \(L_1^{(p,j)}\) have the same Gaussian
scale, where  \(L_1^{(p,j)}\) is defined \cite{ACDGO} in terms of the $JC$ Mordell-Appell integrals in \eqref{eq:JC def}

\begin{align}
L_1^{(p, j)}(t)&:=  \text{sign}(12j-p) JC_{(6p,12j-p)}(t) + JC_{(6p,12j+p)}(t) 
\nonumber\\
    &\qquad + JC_{(6p,7p-12j)}(t) + \text{sign}(12j-5p)JC_{(6p,|12j-5p|)}(t).
    \label{eq:l1s}
\end{align}
On the unary side, the four sine terms in \eqref{eq:ls} then give false theta series supported on
\(n\equiv\pm a\pmod{12p}\), with exponents \(n^2/(24p)\). Recall the false theta series \eqref{eq:false}:
\[
\Psi_{6p}^{(a)}(\rho)
=
\sum_{\substack{n\geq1\\n\equiv\pm a\ ({\rm mod}\ 12p)}}
\psi_{12p}^{(a)}(n)\rho^{n^2/(24p)},
\qquad
\psi_{12p}^{(a)}(n)=
\begin{cases}
1,&n\equiv a\pmod{12p},\\
-1,&n\equiv-a\pmod{12p},
\end{cases}
\]
Note that \(\Psi_{6p}^{(-a)}=-\Psi_{6p}^{(a)}\).

For \(p=6k\pm1\) and \(1\leq j\leq(p-1)/2\), the four indices in \eqref{eq:ls} are \(6j-p,\,7p-6j,\,6j+p,\,5p-6j\).  Writing \(d=6j-p\), these become
\(d,6p-d,d+2p,4p-d\), whose squares are all congruent to
\(d^2\pmod{24p}\).  Hence, with
\(\Delta_{(p,j)}=d^2/(24p)\), their sum can be written
\[
\Psi_{6p}^{(a_1)+(a_2)+(a_3)+(a_4)}(q^{-1})
=
q^{-\Delta_{(p,j)}}\Phi_j^{(p)}(q^{-1}),
\qquad
\Phi_j^{(p)}(q^{-1})\in\mathbb Z[[q^{-1}]].
\]
Consequently \(T_q\Phi_j^{(p)}(q^{-1})=\Phi_j^{(p)}(-q^{-1})\); restoring the
common fractional power, and continuation through the upper half-plane
contributes the phase \(e^{\pi i\Delta_{(p,j)}}\).

On the Stokes line, after extraction of this common fractional power,
\(\widehat L_j^{(p)}\) gives \(\Phi_j^{(p)}(q^{-1})\), whereas the
corresponding cosine combination gives
\((-1)^j\Phi_j^{(p)}(-q^{-1})\), with the permutation of components
described below.  Thus the sine-to-cosine passage realizes, on the
integer-power unary series,
\[
\Phi_j^{(p)}(q^{-1})\longmapsto \Phi_j^{(p)}(-q^{-1}).
\]

\subsection{Kernel realization of \(T\)}

Let \(P=6p\).  At the common Gaussian scale the Stokes-line kernels are
\(S_{P,a}(x)=\sin((P-a)x)/\sin(Px)\) and
\(C_{P,a}(x)=\cos((P-a)x)/\cos(Px)\).  Since \(P\) is even and the
relevant \(a\)'s are odd, the upper-lateral exponential expansion of
\(S_{P,a}\) finishes the proof of Proposition \ref{P:trT}. $\Box$

For the indices \(d,6p-d,d+2p,4p-d\), the integral shifts relative to
\(d^2/(24p)\) are \(0,\,2p-3j,\,j,\,p-2j\), and hence the parity
factors are \(1,(-1)^j,(-1)^j,-1\).  These are precisely the relative
signs in the corresponding cosine combination.

Finally, direct substitution in the definition of \(L_1^{(p,j)}\)
shows that
\[
T\mathcal S_{p,j}
=
(-1)^jL_1^{(p,\pi(j))},
\qquad
\pi(j)=
\begin{cases}
(p-j)/2,&j\ {\rm odd},\\
j/2,&j\ {\rm even}.
\end{cases}
\]
Thus the $T$ transformation maps the relevant sine kernels to cosine
kernels, providing the link between the transformation laws for series
in $q^2$ and those for series in $-q$.

\subsection{Proof of Proposition \ref{P:Stokes-decomposition}}(Stokes Line Transseries Decomposition)

\subsubsection{Decomposition at the Stokes line}\label{S:St-dec}

We first consider the Mordell-Appell integrals defined in \eqref{eq:JS def} for $\Re(t)>0$:
\begin{equation}
\label{eq:JS}
JS_{(p,a)}(t)
=
\frac1t\int_0^\infty
e^{-pu^2/t}
\frac{\sinh((p-a)u)}{\sinh(pu)}\,du,
\qquad 0<a<p.
\end{equation}
We analytically continue $t$ to the negative real axis through the upper
half-plane. Writing
\[
t=-T+i0,\qquad T>0,
\]
the $u$-contour is rotated simultaneously through half the angle, so as
to preserve the decay of the Gaussian factor in \eqref{eq:JS}. More precisely, as
$t=T e^{i(\pi-\varepsilon)}$, we rotate
\[
u=e^{i(\pi-\varepsilon)/2}z.
\]
Letting $\varepsilon\to0^+$, and writing $u=iz$, gives
\begin{equation}
\label{eq:JSrotated}
JS_{(p,a)}(-T+i0)
=
-\frac{i}{T}
\int_0^{\infty-i0}
e^{-pz^2/T}
\frac{\sin((p-a)z)}{\sin(pz)}\,dz.
\end{equation}
Thus, in the $z$-plane, the limiting contour approaches the positive
real axis from below. Hence
\[
\int_0^{\infty-i0}
=
\operatorname{PV}\int_0^\infty
+
\frac12\oint_{\mathcal H_+}.
\]
where $\mathcal H_+$ is a positively oriented Hankel contour surrounding
the poles on $\mathbb R_+$.

These poles are
\[
z_n=\frac{\pi n}{p},\qquad n\geq1,
\]
and
\[
\operatorname*{Res}_{z=z_n}
\left(
e^{-pz^2/T}\, 
\frac{\sin((p-a)z)}{\sin(pz)}
\right)
=
-\frac1p
\sin\left(\frac{\pi an}{p}\right)
e^{-\pi^2n^2/(pT)}.
\]
Therefore
\begin{equation}
\label{eq:JSHankel}
\frac12\oint_{\mathcal H_+}
e^{-pz^2/T}\, 
\frac{\sin((p-a)z)}{\sin(pz)}\,dz
=
-\frac{\pi i}{p}
\sum_{n=1}^\infty
\sin\left(\frac{\pi an}{p}\right)
e^{-\pi^2n^2/(pT)}.
\end{equation}

It remains to identify the principal-value part. For $\Im (z)>0$ and
$\Im (z)<0$, respectively,
\[
\frac{\sin((p-a)z)}{\sin(pz)}
=
\sum_{k=0}^\infty
\left(
e^{i(2pk+a)z}
-
e^{i(2p(k+1)-a)z}
\right)
\]
and
\[
\frac{\sin((p-a)z)}{\sin(pz)}
=
\sum_{k=0}^\infty
\left(
e^{-i(2pk+a)z}
-
e^{-i(2p(k+1)-a)z}
\right).
\]
Taking the half-sum of the corresponding lateral Gaussian integrals
gives
\begin{equation}
\label{eq:JSPVseries}
\operatorname{PV}\int_0^\infty
e^{-pz^2/T}\, 
\frac{\sin((p-a)z)}{\sin(pz)}\,dz
=
\frac12\sqrt{\frac{\pi T}{p}}
\sum_{k=0}^\infty
\left[
e^{-T(2pk+a)^2/(4p)}
-
e^{-T(2p(k+1)-a)^2/(4p)}
\right].
\end{equation}

Define the $2p$-periodic function
\[
\chi_{p,a}(m)
=
\begin{cases}
1, & m\equiv a\pmod{2p},\\
-1, & m\equiv -a\pmod{2p},\\
0, & \text{otherwise},
\end{cases}
\]
and recall the associated unary false theta series introduced in \eqref{eq:false}:
\[
 \Psi_p^{(a)}(e^{-T})
=
\sum_{m=1}^\infty
\chi_{p,a}(m)e^{-Tm^2/(4p)}.
\]
Then, \eqref{eq:JSPVseries} becomes
\begin{equation}
\label{eq:JSPVfalse}
\operatorname{PV}\int_0^\infty
e^{-pz^2/T}\, 
\frac{\sin((p-a)z)}{\sin(pz)}\,dz
=
\frac12\sqrt{\frac{\pi T}{p}}\,
 \Psi_p^{(a)}(e^{-T}).
\end{equation}

Combining \eqref{eq:JSrotated}, \eqref{eq:JSHankel}, and
\eqref{eq:JSPVfalse}, we obtain
\begin{equation}
\label{eq:JScontinued}
JS_{(p,a)}(-T+i0)
=
-\frac{i}{2}\sqrt{\frac{\pi}{pT}}\,
 \Psi_p^{(a)}(e^{-T})
-
\frac{\pi}{pT}
\sum_{n=1}^\infty
\sin\left(\frac{\pi an}{p}\right)
e^{-\pi^2n^2/(pT)}.
\end{equation}

The corresponding calculation for \eqref{eq:JC def}
\[
JC_{(p,a)}(t)
=
\frac1t\int_0^\infty
e^{-pu^2/t}\, 
\frac{\cosh((p-a)u)}{\cosh(pu)}\,du,
\qquad 0<a<p,
\]
is very similar. Define the $4p$-periodic function
\[
\widetilde\chi_{p,a}(m)
=
\begin{cases}
1, & m\equiv a,\ 2p-a\pmod{4p},\\
-1, & m\equiv 2p+a,\ 4p-a\pmod{4p},\\
0, & \text{otherwise},
\end{cases}
\]
\[
\widetilde \Psi_p^{(a)}(e^{-T})
=
\sum_{m=1}^\infty
\widetilde\chi_{p,a}(m)e^{-Tm^2/(4p)}.
\]
Then
\begin{equation}
\label{eq:JCcontinued}
JC_{(p,a)}(-T+i0)
=
-\frac{i}{2}\sqrt{\frac{\pi}{pT}}\,
\widetilde \Psi_p^{(a)}(e^{-T})
-
\frac{\pi}{pT}
\sum_{n=0}^\infty
\sin\left(\frac{\pi a(2n+1)}{2p}\right)
e^{-\pi^2(2n+1)^2/(4pT)}.
\end{equation}

The unary character of the principal-value terms is explicit. Indeed,
\begin{equation}
\label{eq:false-theta-unary-S}
 \Psi_p^{(a)}(e^{-T})
=
e^{-Ta^2/(4p)}
\sum_{k=0}^{\infty}
\left(
e^{-T(pk^2+ak)}
-
e^{-T(p(k+1)^2-a(k+1))}
\right),
\end{equation}
while
\begin{equation}
\label{eq:false-theta-unary-C}
\widetilde \Psi_p^{(a)}(e^{-T})
=
e^{-Ta^2/(4p)}
\sum_{k=0}^{\infty}(-1)^k
\left(
e^{-T(pk^2+ak)}
+
e^{-T(p(k+1)^2-a(k+1))}
\right).
\end{equation}
Thus, after extraction of a rational exponential factor, the
principal-value terms are unary series in nonnegative integer powers
of $e^{-T}$.

The residue terms in \eqref{eq:JScontinued} and
\eqref{eq:JCcontinued} are likewise finite linear combinations of ordinary 
unary theta-type series,
now in the exponentials dual to $e^{-T}$, namely $e^{-\pi^2/T}$.

The functions $\Phi_j^{(p)}$ in Proposition~\ref{P:Stokes-decomposition}
are obtained by combining the elementary unary false theta series
$\Psi_p^{(a)}$ and $\widetilde \Psi_p^{(a)}$ occurring above according
to the four Mordell--Appell blocks defining each $L_j^{(p)}$ in \eqref{eq:ls}.
Using some  trigonometric identities listed below in \eqref{eq:trig}  to combine  the
corresponding residue terms then gives the mixing matrix $M^{(p)}$ and
yields \eqref{eq:mock_Q_unary}. The argument for order 3 mock thetas is very similar
and is omitted. 
\qed

The identities in \eqref{eq:trig}  can be derived in a number of elementary ways. A
particularly short proof is to notice that they are the spectral
decompositions and inverse finite sine transforms of the Green kernels
for the second-order difference operator on a finite interval, with
respectively the two relevant boundary conditions.
\begin{eqnarray}
\label{eq:trig}
\frac{\sinh[(p-a)u]}{\sinh(pu)}
&=&
\frac{1}{p}\sum_{b=1}^{p}
\sin\left(\frac{ab\pi}{p}\right)
\frac{\sin\left(\frac{b\pi}{p}\right)}
{\cosh(u)-\cos\left(\frac{b\pi}{p}\right)},
\nonumber
\\
\frac{\sin\left(\frac{a\pi}{p}\right)}
{\cosh(u)-\cos\left(\frac{a\pi}{p}\right)}
&=&
2\sum_{b=1}^{p}
\sin\left(\frac{ab\pi}{p}\right)
\frac{\sinh[(p-b)u]}{\sinh(pu)},
\nonumber
\\
\frac{\cosh[(p-a)u]}{\cosh(pu)}
&=&
\frac{1}{p}\sum_{b=1}^{p}
\sin\left(\frac{a(b-\frac12)\pi}{p}\right)
\frac{\sin\left((b-\frac12)\frac{\pi}{p}\right)}
{\cosh(u)-\cos\left((b-\frac12)\frac{\pi}{p}\right)},
\nonumber
\\
\frac{\sin\left(\frac{(a-\frac12)\pi}{p}\right)}
{\cosh(u)-\cos\left(\frac{(a-\frac12)\pi}{p}\right)}
&=&
\frac{\cos(a\pi)}{\cosh(pu)}
+
2\sum_{b=1}^{p}
\sin\left(\frac{(a-\frac12)b\pi}{p}\right)
\frac{\cosh[(p-b)u]}{\cosh(pu)}.
\end{eqnarray}

\subsubsection{Rotation to the Stokes line and decomposition for the order 7 mock theta vector}
\label{S:unarymf7}
$ $ 

In this section we explicitly illustrate the preceding Stokes line decomposition  for the order 7 mock theta vector.
For \(1\le k\le10\), let \(\delta_k=k\pi/21\) and define
\[
G_k( t)
:=
\frac{\cos\delta_k}{\pi}
\int_0^\infty
e^{-\frac{21 t}{\pi^2}x^2}
\left(
\frac{1}{\cosh x-\sin\delta_k}
+
\frac{1}{\cosh x+\sin\delta_k}
\right)\,dx.
\]
The Fourier representation obtained above gives
\[
\boldsymbol{\mathcal L}_7( t)
=
\begin{pmatrix}
G_4( t)-G_{10}( t)
\\[2mm]
G_1( t)+G_8( t)
\\[2mm]
G_2( t)+G_5( t)
\end{pmatrix}.
\]
The poles of the integrand on the negative imaginary axis  are at
\[
-i\left(\frac{\pi}{2}\pm\delta_k+2\pi n\right),
\qquad
-i\left(\frac{3\pi}{2}\pm\delta_k+2\pi n\right),
\qquad n\ge0.
\]
Calculating their residues we obtain
\begin{equation}\label{eq:Hk-Stokes-jump}
G_k^{+\pi-}(-t)
=
G_k^{\rm med}(-t)+R_k(t),
\end{equation}
where
\[
R_k(t)
=
\sum_{n=0}^\infty
\bigg[
e^{-\frac{t}{84}(84n+21-2k)^2}
+
e^{-\frac{t}{84}(84n+21+2k)^2}
-
e^{-\frac{t}{84}(84n+63-2k)^2}
-
e^{-\frac{t}{84}(84n+63+2k)^2}
\bigg].
\]
Setting
\[
\mathbf E(t)
:=
\begin{pmatrix}
R_4(t)-R_{10}(t)
\\[2mm]
R_1(t)+R_8(t)
\\[2mm]
R_2(t)+R_5(t)
\end{pmatrix},
\]
we obtain
\[
\boldsymbol{\mathcal L}_7^{\pi-}(-t)
=
\boldsymbol{\mathcal L}_7^{\rm med}(-t)+\mathbf E(t).
\]

Using the principal branch of the square root, we have
\[
t\to - t\pm i0
\quad\Longrightarrow\quad
\frac{\pi^2}{ t}\mapsto-\frac{\pi^2}{t}\mp i0,
\qquad
-\sqrt{\frac{\pi}{ t}}
\mapsto
\pm i\sqrt{\frac{\pi}{t}}.
\]

Let
\(\boldsymbol{\mathcal L}_7^{-\pi+}(-t)\)
and
\(\boldsymbol{\mathcal L}_7^{+\pi-}(-t)\)
denote the lower and upper lateral values, respectively. The transformation law
\[
\boldsymbol{\mathcal L}_7( t)
=
-\sqrt{\frac{\pi}{ t}}\,
S_7\,
\boldsymbol{\mathcal L}_7
\!\left(\frac{\pi^2}{ t}\right),
\]
where $S_7$ has the explicit form
\[
S_7= -M^{(7)}
=
\frac{2}{\sqrt7}
\begin{pmatrix}
\sin\left(\frac{\pi}{7}\right)
&
\sin\left(\frac{2\pi}{7}\right)
&
\sin\left(\frac{3\pi}{7}\right)
\\[2mm]
\sin\left(\frac{2\pi}{7}\right)
&
-\sin\left(\frac{3\pi}{7}\right)
&
\sin\left(\frac{\pi}{7}\right)
\\[2mm]
\sin\left(\frac{3\pi}{7}\right)
&
\sin\left(\frac{\pi}{7}\right)
&
-\sin\left(\frac{2\pi}{7}\right)
\end{pmatrix}. 
\]This 
gives
\begin{equation}\label{eq:mf7pm}
\boldsymbol{\mathcal L}_7^{\pm\pi\mp}(-t)
=
\pm i\sqrt{\frac{\pi}{t}}\,
S_7\,
\boldsymbol{\mathcal L}_7^{\mp\pi\pm}\left(-\frac{\pi^2}{t}\right).
\end{equation}

The lateral decompositions are
\begin{equation}\label{eq:latpm}
\boldsymbol{\mathcal L}_7^{\pm\pi\mp}(-t)
=
\boldsymbol{\mathcal L}_7^{\rm med}(-t)
\pm\mathbf E(t),
\end{equation}
and likewise at \(-\pi^2/t\). Averaging the two identities \eqref{eq:mf7pm} and using
\eqref{eq:latpm} provides a formula for the medianized part:
\[
\boldsymbol{\mathcal L}_7^{\rm med}(-t)
=
-i\sqrt{\frac{\pi}{t}}\,
S_7\,
\mathbf E\left(\frac{\pi^2}{t}\right).
\]
Hence
\begin{equation}\label{eq:mf7-Stokes-TL}
\boldsymbol{\mathcal L}_7^{\pm\pi\mp}(-t)
=
\pm\mathbf E(t)
-
i\sqrt{\frac{\pi}{t}}\,
S_7\,
\mathbf E\!\left(\frac{\pi^2}{t}\right),
\qquad t>0.
\end{equation}
Thus the residue decomposition of the Mordell--Appell integrals reproduces directly the
mixing matrix $S_7$ of the order 7 mock theta transformation law.

\subsubsection{Identification with false theta series at the Stokes line}\label{S:indenf}
$ $

Recall \(Q=e^{-2t}\) and \(Q_1=e^{-2\pi^2/t}\). Thus, for $t \in \RR^-$ the residue series
converge in the reciprocal nomes \(Q^{-1}\) and \(Q_1^{-1}\). Let \(\rho_0,\rho_1,\rho_2:\mathbb Z\to\{-1,0,1\}\) be
\(84\)-periodic and odd. On \(1\le m<42\), define
\[
\rho_0(m)=
\begin{cases}
-1,&m=1,41,\\
1,&m=13,29,\\
0,&\text{otherwise},
\end{cases}
\]
\[
\rho_1(m)=
\begin{cases}
1,&m=5,19,23,37,\\
0,&\text{otherwise},
\end{cases}
\qquad
\rho_2(m)=
\begin{cases}
1,&m=11,17,25,31,\\
0,&\text{otherwise}.
\end{cases}
\]
Let $r_0=1, r_1=5$ and $r_2=11$, and define 
\[
\psi_j(z)
=
\sum_{m=1}^{\infty}
\rho_j(m)\,
z^{\frac{m^2-r_j^2}{168}},
\qquad j=0,1,2.
\]
$ $
For \(m\) in the support of \(\rho_j\), the exponent is a
nonnegative integer, and
\begin{equation}\label{eq:unary-at-zero}
   \psi_0(0)=-1,
\qquad
\psi_1(0)=1,
\qquad
\psi_2(0)=1. 
\end{equation}
\begin{equation}\label{eq:falsetheta}
 \boldsymbol{\Psi}^{(7)}(Q)
=
\begin{pmatrix}
Q^{-1/168}\psi_0(Q^{-1})
\\[2mm]
Q^{-25/168}\psi_1(Q^{-1})
\\[2mm]
Q^{-121/168}\psi_2(Q^{-1})
\end{pmatrix}
\end{equation}
A direct comparison with the residue sums gives
\[
\mathbf E(t)
=
\boldsymbol{\Psi}^{(7)}(Q),
\qquad
\mathbf E\!\left(\frac{\pi^2}{t}\right)
=
\boldsymbol{\Psi}^{(7)}(Q_1).
\]
Choosing one of the equalities in \eqref{eq:mf7-Stokes-TL} we get
\begin{equation}\label{eq:mf7-Stokes-false-theta}
\boldsymbol{\mathcal L}_7^{+\pi-}(-t)
=
\boldsymbol{\Psi}^{(7)}(Q)
-
i\sqrt{\frac{\pi}{t}}\,
S_7\,
\boldsymbol{\Psi}^{(7)}(Q_1),
\qquad t>0.
\end{equation}

\begin{remark}
{\rm
The completed Mordell vector can be reconstructed from one component
provided its decomposition into elementary cosh integrals is retained.
For example, applying the Fourier transform separately to
\(G_4\) and \(G_{10}\) in
\[
G_4-G_{10}
\]
produces two independent transformed combinations. Together with the
original combination, these determine the three-dimensional
Fourier-invariant kernel space and hence the remaining components.
This information is not contained in the vector functional equation
alone.
}
\end{remark}

\subsubsection{The Stokes-line decomposition for the order 3 mock theta vector}
\label{S:unarymf3}
$ $

The order 3 case follows from the same calculation, with the minor
difference that Watson's transformation laws \cite{Watson} involve the half-nomes
$q$ and $q_1$.  The two Mordell--Appell integrals needed here satisfy
\[
W_3(t)=JS_{(3,2)}(t),
\qquad
W_2\!\left(\frac{t}{2}\right)=2JC_{(3,2)}(t).
\]
Hence \eqref{eq:JScontinued} and \eqref{eq:JCcontinued} apply directly.

The principal-value parts take the form
\[
q^{2/3}\Phi_1(-q^{-1}),
\qquad
q^{2/3}\Phi_1(q^{-1}),
\]
where $\Phi_1(\pm q^{-1})$ are unary series in integer powers of
$q^{-1}$.  In the notation \eqref{eq:false}--\eqref{eq:ns},
\[
\Psi_3^{(2)}(q^{-1})
=
q^{2/3}\Phi_1(-q^{-1}),
\qquad
\Psi_6^{(2)+(4)}(q^{-2})
=
-q^{2/3}\Phi_1(q^{-1}).
\]
The false theta series associated with the $f(q_1^2)$ term is
\[
\Phi_2(q)
=
2q^{-1/24}\Psi_6^{(1)-(5)}(q)
\in\mathbb Z[[q]].
\]
The residue terms give the dual unary series
\[
q_1^{2/3}\Phi_1(-q_1^{-1}),
\qquad
q_1^{-1/12}\Phi_2(q_1^{-2}),
\]
respectively.
With the upper-lateral convention of \S\ref{S:St-dec}, substitution
into \eqref{eq:JScontinued} and \eqref{eq:JCcontinued} gives
\eqref{eq:omega tlaw vee} and \eqref{eq:omega f law vee}.

\subsection{Median--jump duality on the Stokes line}
\label{S:median-jump}
$ $

We describe the mechanism relating the principal-value and residue parts
of the Mordell--Appell integrals on the Stokes line.  This also explains
why the principal-value terms in
Proposition~\ref{P:Stokes-decomposition} are unary series in rational
powers of $q^{-1}$, once the corresponding statement for the Stokes
jumps is known from residue calculus.

Let $\mathcal L(t)$ and $\widetilde{\mathcal L}(t)$ be Mordell--Appell
vectors related for $t\in\RR^+$ by an $S$-transformation of the form
\begin{equation}
\mathcal L(t)
=
\sqrt{\frac{\pi}{t}}\,
M\,\widetilde{\mathcal L}\!\left(\frac{\pi^2}{t}\right).
\label{eq:median-jump-S}
\end{equation}
Denote their upper and lower lateral continuations to $\RR^-$ by
$\mathcal L^{\pi-}$ and $\mathcal L^{-\pi+}$, and set
\[
\mathcal L^{\rm med}
=
\frac12\left(\mathcal L^{\pi-}
+\mathcal L^{-\pi+}\right),
\qquad
\operatorname{Jump}\mathcal L
=
\frac12\left(\mathcal L^{\pi-}
-\mathcal L^{-\pi+}\right).
\]

\begin{prop}\label{P:median-jump}
Under the $S$-transformation \eqref{eq:median-jump-S}, the median value
on the Stokes line is mapped, up to the explicit weight-$1/2$
multiplier, to the Stokes jump of the $S$-dual integral.
\end{prop}

\begin{proof}
Analytically continue the two lateral values from $t\in\RR^-$ back to
$t\in\RR^+$, where \eqref{eq:median-jump-S} holds, apply the
$S$-transformation there, and continue back to the negative real axis.
Since
\[
t=|t|e^{i\theta}
\qquad\Longrightarrow\qquad
\frac{\pi^2}{t}
=
\frac{\pi^2}{|t|}e^{-i\theta},
\]
the transformation $t\mapsto\pi^2/t$ interchanges the upper and lower
lateral directions.  Thus, writing $t=-T$, $T>0$, and using the
principal branch,
\[
\mathcal L^{\pi-}(-T)
=
-i\sqrt{\frac{\pi}{T}}\,
M\widetilde{\mathcal L}^{-\pi+}
\left(-\frac{\pi^2}{T}\right),
\qquad
\mathcal L^{-\pi+}(-T)
=
i\sqrt{\frac{\pi}{T}}\,
M\widetilde{\mathcal L}^{\pi-}
\left(-\frac{\pi^2}{T}\right).
\]
Taking the half-sum gives
\[
\mathcal L^{\rm med}(-T)
=
-i\sqrt{\frac{\pi}{T}}\,
M\operatorname{Jump}\widetilde{\mathcal L}
\left(-\frac{\pi^2}{T}\right),
\]
with the convention
\[
\operatorname{Jump}\widetilde{\mathcal L}
=
\frac12\left(
\widetilde{\mathcal L}^{-\pi+}
-
\widetilde{\mathcal L}^{\pi-}
\right).
\]
\end{proof}

\begin{remark}
    For application to the order 5, 7, 11 mock theta vectors we note that $\mathcal L=\widetilde{\mathcal L}$, however for order 3 the S transformation mixes different Mordell-Appell integrals.
\end{remark}

\section{Appendix 1: Character definitions}
\label{sec:appendix1}

In this appendix, we record the characters from \cite{harvey2018hecke} that are related to the order 5 and order 7 mock theta homogeneous solutions $H^{(5)}$ and $H^{(7)}$ discussed in Section \ref{sec:faster}. See Eqs \eqref{eq: H non uniq 5} and \eqref{eq:p7-H}.

For order $p=5$, the characters $\chi^{G_2}(\tau)$ from \cite{harvey2018hecke} are defined by:
\begin{eqnarray}
    \chi^{G_2}_{0}(\tau) &=&  Q^{-7/60} \sum_{n \ge 0} c_0^{G_2}(n) Q^n, \\
    \chi^{G_2}_{2/5}(\tau) &=& Q^{17/60} \sum_{n \ge 0} c_{2/5}^{G_2}(n) Q^n.
\end{eqnarray}
where the coefficients are determined from Yang-Lee characters by
\begin{align} \label{G2 Hecke}
\begin{split}
c_0^{G_2} (n)=&
\begin{cases} 7c_{1/5}^{YL} (7n-1)
& \text{if  } 7 \nmid n , \\
7c_{1/5}^{YL} (7n-1) + c_0^{YL} (\frac{n}{7})
&  \text{if  } 7 | n ;  
\end{cases} \\
c_{2/5}^{G_2} (n)=&
\begin{cases} 7c_{0}^{YL} (7n+2)
& \text{if  } 7 \nmid (n-1) , \\
7c_{0}^{YL} (7n+2) - c_{1/5}^{YL} (\frac{n-1}{7})
&  \text{if  } 7 | (n-1) . 
\end{cases} 
\end{split}
\end{align}
Here, the characters of the Yang-Lee
model are given by
\begin{align} \label{ylone}
\chi^{YL}_0 (\tau) &= Q^{-1/60} G(Q)=: Q^{-1/60} \sum_{n\ge 0} c_0^{YL}(n) Q^n, \\
\chi^{YL}_{1/5}(\tau) &= Q^{11/60} H(Q)=: Q^{11/60} \sum_{n\ge 0} c_{1/5}^{YL}(n) Q^n
\end{align}
where the functions $G(Q)$, $H(Q)$ obey the Rogers-Ramanujan identities
\begin{align} \label{yltwo}
G(Q) &= \sum_{n=0}^\infty \frac{Q^{n^2}}{(Q;Q)_n} = \prod_{n=0}^\infty \frac{1}{(1-Q^{5n+1})(1-Q^{5n+4})}, \\
H(Q) &= \sum_{n=0}^\infty \frac{Q^{n^2+n}}{(Q;Q)_n} = \prod_{n=0}^\infty \frac{1}{(1-Q^{5n+2})(1-Q^{5n+3})}.
\end{align}
We also use $\chi^{E_8}$ from \cite{harvey2018hecke}, defined as 
\begin{eqnarray}
    \chi^{E_8}(\tau)=\frac{E_4(\tau)}{\eta(\tau)^8}.
\end{eqnarray}

For order $p=7$, the relevant characters from \cite{harvey2018hecke} are defined as:
\begin{align}
\begin{split}
\chi_{0}(\tau)&=\frac{Q^{1/56}}{\eta(\tau)}
\prod_{n=1}^{\infty} (1-Q^{7n})(1-Q^{7n-4})(1-Q^{7n-3}) \\
&= Q^{-1/42}(1+Q+ 2Q^2+ 2Q^3+ 3Q^4+\cdots) , \\
\chi_{1/7}(\tau)&=\frac{Q^{9/56}}{\eta(\tau)}
\prod_{n=1}^{\infty} (1-Q^{7n})(1-Q^{7n-5})(1-Q^{7n-2}) \\
&= Q^{5/42}(1+Q+ Q^2+ 2Q^3+ 3Q^4 + \cdots)  ,\\
\chi_{3/7}(\tau)&=\frac{Q^{25/56}}{\eta(\tau)}
\prod_{n=1}^{\infty} (1-Q^{7n})(1-Q^{7n-6})(1-Q^{7n-1}) \\
&= Q^{17/42}(1+Q^2+Q^3+ 2Q^4 + \cdots) \, .
\end{split}
\end{align}

\section{Acknowledgements}
We thank Griffen Adams for discussions and comments. GD acknowledges helpful discussions with D. Zagier at the Max Planck Institute for Mathematics workshop {\it Combinatorics, Resurgence and Algebraic Geometry in Quantum Field Theory}, August 2024. 
The work of OC is supported
in part by the U.S. National Science Foundation, Division of Mathematical Sciences, Award NSF DMS-2206241.
The work of GD is supported in part by the U.S. Department of Energy, Office of Science, Office of High Energy Physics under Award DE-SC0010339.
AS gratefully acknowledges support from the Department of Mathematics at The Ohio State University through a graduate research award.
  
\bibliographystyle{plain}
\bibliography{ref}

\end{document}